\pdfoutput=1 %arXiv admin added this line
\documentclass[12pt]{article}
\usepackage{modnatbib}
\usepackage{amssymb,amsmath,amsthm,epsfig,a4,verbatim,pstricks,ifthen,setspace} 
\usepackage{url,ifpdf}
\definecolor{mygreen}{rgb}{0.05,0.6,0.05}
\newtheorem{thm}{\sc Theorem.}[section]
\newtheorem{lem}[thm]{\sc Lemma.}
\newtheorem{cor}[thm]{\sc Corollary.}
\newtheorem{rem}[thm]{\sc Remark.}

\newcommand{\R}{{\mathbb R}}

\newcommand{\dx}{\; {\rm d}x}
\newcommand{\ds}{\; {\rm d}s}
\renewcommand{\vec}{}
\newcommand{\unitn}{\vec{\rm n}}
\newcommand{\nabs}{\nabla_{\!s}}
\def\conduct{\mathcal{K}}

\newcommand{\dd}[1]{\frac{\rm d}{{\rm d}#1}}
\newcommand{\ddt}{\dd{t}}

\def\epsilon{\varepsilon} 
\newcommand{\uD}{u_D}
\newcommand{\cPsi}{c_\Psi}
\newcommand{\Varrho}{{\rm P}}
\newcommand{\hatrho}{{\widehat\varrho}}

\newcommand{\ShD}{S^h_D}
\newenvironment{AMS}%
{{\upshape\bfseries AMS subject classifications. }\ignorespaces}{}
\newenvironment{keywords}{{\upshape\bfseries Key words. }\ignorespaces}{}

\def\vL{L\kern-0.08cm\char39}

\begin{document}
\title{%{\bf DRAFT \qquad {\rm \mydate}} \\
Stable Phase Field Approximations of \\ Anisotropic Solidification}

\author{John W. Barrett\footnotemark[2] \and 
        Harald Garcke\footnotemark[3]\ \and 
        Robert N\"urnberg\footnotemark[2]}

\renewcommand{\thefootnote}{\fnsymbol{footnote}}
\footnotetext[2]{Department of Mathematics, 
Imperial College London, London, SW7 2AZ, UK}
\footnotetext[3]{Fakult{\"a}t f{\"u}r Mathematik, Universit{\"a}t Regensburg, 
93040 Regensburg, Germany}

\date{}

\maketitle

\begin{abstract}
We introduce unconditionally stable finite element approximations for
a phase field model for solidification, which take highly anisotropic
surface energy and kinetic effects into account. We hence approximate Stefan
problems with anisotropic Gibbs--Thomson law with kinetic undercooling, and 
quasi-static variants thereof. The phase field model is given by
\begin{align*}
\vartheta\,w_t + \lambda\,\varrho(\varphi)\,\varphi_t & = 
\nabla \,.\, (b(\varphi)\,\nabla\, w ) \,, \\
\cPsi\,\tfrac{a}\alpha\,\varrho(\varphi)\,w & = 
\epsilon\,\tfrac\rho\alpha\,\mu(\nabla\,\varphi)\,\varphi_t
-\epsilon\,\nabla \,.\, A'(\nabla\, \varphi) + \epsilon^{-1}\,\Psi'(\varphi)
\end{align*}
subject to initial and boundary conditions for the phase variable $\varphi$ and
the temperature approximation $w$.
Here $\epsilon > 0$ is the interfacial parameter, 
$\Psi$ is a double well potential, 
$\cPsi = \int_{-1}^1 \sqrt{2\,\Psi(s)}\;{\rm d}s$, $\varrho$ is a shape
function and $A(\nabla\,\varphi) = \tfrac12\,|\gamma(\nabla\,\varphi)|^2$,
where $\gamma$ is the anisotropic density function. Moreover,
$\vartheta \geq 0$, $\lambda > 0$, $a > 0$, $\alpha > 0$ and $\rho \geq 0$
are physical parameters from the Stefan problem, while $b$ and $\mu$ are
coefficient functions which also relate to the sharp interface problem.

On introducing the novel fully practical finite element approximations 
for the anisotropic phase field model, we prove
their stability and demonstrate their applicability with some numerical
results.
\end{abstract} 


\renewcommand{\thefootnote}{\arabic{footnote}}

\section{Introduction} \label{sec:0}

Phase field models are a successful approach for interface evolution in
cases where interfacial energy is important, and many numerical
approaches for the underlying equations have been studied in the
literature. However, in situations where anisotropy is incorporated
only very few results related to the numerical analysis of
approximations to the phase field system have appeared in the
literature. The reason for this is that the underlying equations
involve highly nonlinear parabolic partial differential
equations. Since phase field models describe very unstable
solidification phenomena, it seems to be very important to use stable
approximation schemes which do not trigger additional instabilities
resulting from discretization errors. In this context we would like to
mention that there exist many computations on anisotropic
solidification, with the help of phase field equations, showing pattern
formation which is driven by the discretization rather than by the
underlying partial differential equations. The goal of this paper is
to introduce and analyze a new stable finite element approximation for
the anisotropic phase field system. The approach is based on earlier
work for the Allen--Cahn and the Cahn--Hilliard equations, see \cite{eck},
and on ideas on how to handle the anisotropy that have been used earlier for
sharp interface models by the same authors, see \cite{triplejANI, ani3d}.
To our knowledge, the introduced finite element approximation is the first
unconditionally stable approximation of a phase field model 
for anisotropic solidification in the literature.

As the phase field model and its quasi-stationary variant, the viscous
Cahn--Hilliard equation, converge to sharp interface models for
solidification in the asymptotic limit as the interfacial thickness
tends to zero we first introduce the sharp interface model. 
Let $\Gamma(t) \subset \R^d$, $d=2,3$, denote the interface
between a solid and liquid phase, say, or a solid phase and a gas phase.
Then the surface energy of $\Gamma(t)$ is defined as 
\begin{equation} \label{eq:EGamma}
\int_{\Gamma(t)} \gamma(\unitn) \ds \,,
\end{equation} 
where $\unitn$ denotes the unit normal of $\Gamma(t)$, and where 
the anisotropic density function $\gamma : \R^d \to \R_{\geq 0}$ with 
$\gamma\in C^2(\R^d \setminus \{ \vec 0 \}) \cap C(\R^d)$ is assumed to be 
absolutely homogeneous of degree one, i.e.\
\begin{equation}
\gamma(\lambda\,\vec{p}) = |\lambda|\gamma(\vec{p}) \quad \forall \
\vec{p}\in \R^d,\ \forall\ \lambda \in\R \quad \Rightarrow
\quad \gamma'(\vec{p})\,.\,\vec{p} = \gamma(\vec{p})
\quad \forall\ \vec{p}\in \R^d\setminus\{\vec0\}, \label{eq:homo}
\end{equation}
with $\gamma'$ denoting the gradient of $\gamma$.

Relevant for our considerations is the first variation, $-\kappa_\gamma$, of
(\ref{eq:EGamma}), which can be computed as
\begin{equation*}
\kappa_\gamma := - \nabs\,.\, \gamma'(\unitn)\,;
\end{equation*}
where $\nabs .$ is the tangential divergence of $\Gamma$, see e.g.\
\cite{CahnH74,ani3d,dendritic}.
Note that $\kappa_\gamma$ reduces to the sum of the principal curvatures of 
$\Gamma$ in the isotropic case, i.e.\ when $\gamma$ satisfies
\begin{equation} \label{eq:iso}
\gamma(\vec p) = |\vec p| \qquad \forall\ \vec p \in \R^d\,.
\end{equation}

Then the full Stefan problem that we want to consider in this paper is given as
follows, where $\Omega\subset\mathbb{R}^d$ is a given fixed domain with
boundary $\partial\Omega$ and outer normal $\vec\nu$.

Find $u : \Omega \times [0,T] \to \R$ 
and the interface $(\Gamma(t))_{t\in[0,T]}$ such that
for all $t\in (0,T]$ the following conditions hold:
\begin{subequations}
\begin{alignat}{2}
\vartheta\,u_t - \conduct_-\,\Delta u & = 0 
\qquad \mbox{in } \Omega_-(t), \qquad
\vartheta\,u_t - \conduct_+\,\Delta u = 0 
\qquad &&\mbox{in } \Omega_+(t), \label{eq:1a} \\
\left[ \conduct\,\frac{\partial u}{\partial \unitn} \right]_{\Gamma(t)}
&=-\lambda\,{\cal V} \qquad &&\mbox{on } \Gamma(t), 
\label{eq:1b} \\ % Stefan condition
\frac{\rho\,{\cal V}}{\beta(\unitn)} &= \alpha\,\kappa_\gamma - 
a\,u \qquad &&\mbox{on } \Gamma(t), 
\label{eq:1c} \\ % Gibbs--Thomson condition (with undercooling)
\frac{\partial u}{\partial \vec\nu} &= 0 
\qquad \mbox{on } \partial_N\Omega, \qquad
u = \uD \qquad &&\mbox{on }
\partial_D \Omega \,, \label{eq:1d} \\
\Gamma(0) & = \Gamma_0 \,, \qquad\qquad 
\vartheta\,u(\cdot,0) = \vartheta\,u_0 \quad && \mbox{in } 
\Omega \,. \label{eq:1e} 
\end{alignat}
\end{subequations}
In the above $u$ denotes the deviation from the melting temperature $T_M$,
i.e.\ $T_M$ is the melting temperature for a planar interface.
In addition, $\Omega_-(t)$ is the solid region, with boundary
$\Gamma(t) = \partial \Omega_-(t)$, so that the liquid region is given 
by $\Omega_+(t):=\Omega\setminus\overline{\Omega_-(t)}$. 
Here we assume that the solid region $\overline{\Omega_-(t)}$ has no
intersection with the external boundary $\partial\Omega$, but more general
situations can also be considered, as will be outlined in Section~\ref{sec:3}
below.
Moreover, here and throughout this paper, for
a quantity $v$ defined on $\Omega$, we use the shorthand notations
$v_- := v\!\mid_{\Omega_-}$ and $v_+ := v\!\mid_{\Omega_+}$.
The parameters $\vartheta\geq0$, $\lambda>0$, $\rho\geq0$, $\alpha>0$, $a>0$ 
are assumed to be constant, while $\conduct_\pm>0$ are
assumed to be constant in each phase. 
The mobility coefficient $\beta:\R^d \to \R_{\geq0}$ is assumed to satisfy
$\beta(\vec p) > 0$ for all $\vec p \not= \vec 0$ and to be positively
homogeneous of degree one.
In addition 
$[\conduct\,\frac{\partial u}{\partial \unitn}]_{\Gamma(t)}(\vec{z}) := 
(\conduct_+\,\frac{\partial u_+}{\partial \unitn} - 
\conduct_-\,\frac{\partial u_-}{\partial \unitn})
(\vec{z})$ for all $\vec{z}\in\Gamma(t)$, and
$\mathcal{V}$ is the velocity of $\Gamma(t)$ in the direction of its normal
$\unitn$, which from now on we assume is pointing into $\Omega_+(t)$. 
Finally, $\partial\Omega =
\overline{\partial_N\Omega}\cup\overline{\partial_D\Omega}$ with
$\partial_N\Omega\cap\partial_D\Omega = \emptyset$, 
$\uD : \partial_D\Omega \to \R$ is the applied supercooling at the boundary,
and
$\Gamma_0 \subset \overline\Omega$ and
$u_0 : \Omega \to \R$ are given initial data.

The model (\ref{eq:1a}--e) can be derived for example
within the theory of rational thermodynamics and we refer to
\cite{Gurtin88a} for details. We remark that a derivation from thermodynamics
would lead to the identity
$a = \frac{\lambda}{T_M}$.
We note that
(\ref{eq:1b}) is the well-known Stefan condition, while (\ref{eq:1c}) is the
Gibbs--Thomson condition, with kinetic undercooling if $\rho>0$. The
case $\vartheta>0$, $\rho>0$, $\alpha >0$ leads to the Stefan
problem with the Gibbs--Thomson law and kinetic undercooling.
In some models in the literature, see e.g.\ \cite{Luckhaus90}, the
kinetic undercooling is set to zero, i.e.\ $\rho=0$. Setting
$\vartheta=\rho=0$ but keeping $\alpha>0$ leads to the
Mullins--Sekerka problem with the Gibbs--Thomson law, see
\cite{MullinsS63}. 

For later reference, we introduce the function spaces
\begin{equation*}
S_0 := \{ \eta \in H^1(\Omega) : \eta = 0 \ \mbox{ on } \partial_D\Omega \}
\quad\mbox{and}\quad
S_D := \{ \eta \in H^1(\Omega) : \eta = \uD \ \mbox{ on } \partial_D\Omega\}
\,, % \label{eq:SD}
\end{equation*}
where we assume for simplicity of the presentation from now on that 
\begin{align}
\mbox{either\quad (i)~~} & \partial\Omega = \partial_D\Omega\,, 
\mbox{\qquad\quad (ii)~~} \partial\Omega = \partial_N\Omega\,, \nonumber \\
\mbox{or \quad (iii)~~} &
\Omega=(-H,H)^d,\quad \partial_D\Omega = [-H,H]^{d-1}\times\{H\},\quad H > 0\,;
\label{eq:assOmega}
\end{align}
and, in the cases (\ref{eq:assOmega})(i) and (iii), 
that $\uD \in H^{\frac{1}{2}} (\partial_D \Omega)$. For notational convenience,
we define $\uD:=0$ in the case (\ref{eq:assOmega})(ii). 

We recall from \cite{dendritic} that, on assuming that $\uD$ is constant,
for a solution $u$ and $\Gamma$ to
(\ref{eq:1a}--e) it can be shown that the following formal energy equality 
holds
\begin{align}
& \ddt\left(\frac\vartheta2\,|u-\uD|^2_0 + 
\frac{\lambda\,\alpha}{a}\, \int_{\Gamma(t)} \gamma(\unitn)\ds
-\lambda\,\uD\,|\Omega_+(t)|\right) + (\conduct\,\nabla\,u, \nabla\,u) 
\nonumber \\ & \hspace{10cm}
+\frac{\lambda\,\rho}{a}\, \int_{\Gamma(t)} \frac{{\cal V}^2}{\beta(\unitn)} \ds
= 0 \,, \label{eq:testD}
\end{align}
where $(\cdot,\cdot)$ denotes the $L^2$--inner product
over $\Omega$, with the corresponding norm given by $|\cdot|_0$, and where
$|\Omega_+(t)| := \int_{\Omega_+(t)} 1 \dx$.

In Section~\ref{sec:2} we will precisely state a phase field model
which approximates the free boundary problem (\ref{eq:1a}--e). We only
mention here that the phase field method is based on the idea of a
diffuse interface, which hence has a positive thickness.
Let us briefly discuss some relevant literature. For
solidification the phase field method was originally proposed by
\cite{Langer86} as a model for solidification of a pure
substance. It was \cite{Kobayashi93} who first was able to
simulate complicated dendritic patterns which resemble those %patterns
appearing during solidification. Since then an enormous effort has
gone into numerically studying phase field models. We refer only to
\cite{ElliottG96preprint,KarmaR96,KarmaR98} and to the reviews 
\cite{BoettingerWBK02,Chen02,McFadden02,SingerLoginovaS08}.

A phase field model, and its numerical approximation, for the sharp interface 
problem (\ref{eq:1a}--e) with $\vartheta=\rho=0$ and
$\conduct_+=\conduct_-$ has been considered in the recent paper \cite{eck}. 
In particular, the authors were able to present unconditionally stable finite
element approximations, where the treatment of the anisotropy does not lead to
new nonlinearities compared to the isotropic situation.
It is one of the aims of the present article to extend the discretizations 
in \cite{eck} to the more general problem (\ref{eq:1a}--e),
i.e.\ in particular to the case $\vartheta > 0$, and $\rho > 0$, and to a wider 
class of anisotropies than considered in \cite{eck}. The new anisotropies
considered in the present article will lead to more nonlinear schemes, however.

The remainder of the paper is organized as follows. 
In Section~\ref{sec:2}
we state the two phase field models for the approximation of
the sharp interface problem (\ref{eq:1a}--e) that we want to consider in this
paper.
In Section~\ref{sec:3} we introduce our finite element approximations for 
these problems, and we prove stability results 
for these approximations. 
Solution methods for the discrete equations are shortly reviewed in
Section~\ref{sec:4}.
In addition, we present several numerical experiments in
Section~\ref{sec:5}. 

\setcounter{equation}{0} 
\section{Phase field models and anisotropies} \label{sec:2}

Phase field models are a computational tool to compute approximations for sharp
interface evolutions such as (\ref{eq:1a}--e), without having to capture the
sharp interface $\Gamma(t)$ directly. On introducing a phase field
$\varphi:\Omega \times (0,T) \to \R$, where the sets
$\Omega_\pm^\epsilon(t) := 
\{ x \in \Omega : \pm \varphi(x,t) > 0\}$ are approximations to
$\Omega_\pm(t)$, a system of partial differential equations for $\varphi$ can
be derived so that the zero level sets of $\varphi$ formally approximate the
interface $\Gamma(t)$, satisfying e.g.\ (\ref{eq:1a}--e), 
in a well defined limit. 
For more details on phase field
methods and other approaches to the approximation of the evolution of
interfaces we refer to the review article \cite{DeckelnickDE05} and the
references therein.

On introducing the small interfacial parameter $\epsilon>0$, it can be shown
that
$$
\frac1\cPsi\,\mathcal{E}_\gamma(\varphi) \approx 
\int_{\Gamma} \gamma(\unitn) \ds\,,
$$
for $\epsilon$ sufficiently small, where
\begin{equation} \label{eq:Eg}
{\cal E}_\gamma(\varphi) := 
\int_\Omega \tfrac\epsilon2\,|\gamma(\nabla\, \varphi)|^2 +
\epsilon^{-1}\,\Psi(\varphi) \dx
\quad \text{with} \quad
\cPsi := \int_{-1}^1 \sqrt{2\,\Psi(s)}\;{\rm d}s\,.
\end{equation}
Here $\Psi : \R \to [0,\infty]$ is a double well potential, which for
simplicity we assume to be symmetric and to have its global minima at $\pm1$.
The canonical example is
\begin{equation} \label{eq:quartic}
\Psi(s) := \tfrac14\,(s^2 - 1)^2
\qquad\Rightarrow\qquad
\Psi'(s) = s^3-s
\quad \text{and}\quad
\cPsi = \tfrac13\,{2^{\frac32}}\,.
\end{equation}
Another possibility is to choose
\begin{equation} \label{eq:obstacle}
\Psi(s):= \begin{cases}
\textstyle \frac 12 \left(1-s^2\right)  & |s|\leq 1\,,\\
\infty & |s|> 1\,,
\end{cases} 
\qquad\Rightarrow\qquad
\cPsi = \tfrac\pi2\,;
\end{equation}
see e.g.\ \cite{BloweyE92,ElliottG96preprint,Elliott97}. 
Clearly the obstacle potential (\ref{eq:obstacle}), 
which forces $\varphi$ to stay within the interval $[-1,1]$,
is not differentiable at $\pm1$. 
Hence, whenever we write $\Psi'(s)$ in the case (\ref{eq:obstacle}) in this
paper, we mean that the
expression holds only for $|s|<1$, and that in general a variational inequality
needs to be employed. While it can be shown that the asymptotic interface 
thickness in phase field models with (\ref{eq:Eg}) for the isotropic surface
energy (\ref{eq:iso}) is proportional to $\epsilon$, for anisotropic energy
densities the asymptotic interface thickness is no longer uniform, but now also
depends on $\gamma$ and on $\nabla\,\varphi$, see
e.g.\ \cite{BellettiniP96, ElliottS96, WheelerM96}.

We remark that other, non-classical,
phase field models are based on the energy
\begin{equation} \label{eq:TL}
\int_\Omega |\nabla\,\varphi|^{-1}\,\gamma(\nabla\,\varphi)
\left(\tfrac\epsilon2\,|\nabla\, \varphi|^2 +
\epsilon^{-1}\,\Psi(\varphi) \right)\dx
\end{equation}
for e.g.\ the smooth double-well potential (\ref{eq:quartic}),
see \cite{TorabiLVW09}.
The energy (\ref{eq:TL}) has the
advantage that the asymptotic interface thickness is now only determined by
$\epsilon$ (independently of $\gamma$ and the orientation of the interface),  
whereas the disadvantage is that the resultant partial differential equations 
become more nonlinear and are singular at $\nabla\,\varphi = \vec 0$.
We note that higher order regularizations of the energies 
(\ref{eq:Eg}) and (\ref{eq:TL}) in the case of a
non-convex anisotropy density function $\gamma$, 
which lead to sixth order Cahn--Hilliard type equations, 
have been considered in e.g.\ \cite{LiLRV09}. 

We are not aware of any numerical analysis for discretizations of 
anisotropic phase field
models for (\ref{eq:1a}--e) involving either (\ref{eq:Eg}) or (\ref{eq:TL}).

We now state the two phase field models that we are going to consider in this
paper. 
To this end, for $\vec p \in \R^d$, let
\begin{equation}
A(\vec p) = \tfrac12\,|\gamma(\vec p)|^2 
\quad\Rightarrow\quad
A'(\vec p) = 
\begin{cases}
\gamma(\vec p)\,\gamma'(\vec p) & \vec p \not= 0\,, \\
\vec 0 & \vec p = \vec 0\,,
\end{cases}
\label{eq:Ap}
\end{equation}
and define
\begin{equation} \label{eq:mu}
\mu(\vec p) = \begin{cases}
\dfrac{\gamma(\vec p)}{\beta(\vec p)} & \vec p \not = \vec 0\,, \\
\bar\mu & \vec p = \vec 0\,,
\end{cases}
\end{equation}
where $\bar\mu\in\R$ is a constant satisfying
$\min_{\vec{p} \not= \vec 0} \frac{\gamma(\vec p)}{\beta(\vec p)}
\leq \bar\mu \leq 
\max_{\vec{p} \not= \vec 0} \frac{\gamma(\vec p)}{\beta(\vec p)}$.

\subsection{Viscous Cahn--Hilliard equation}

A phase field model for (\ref{eq:1a}--e) with $\vartheta = \rho = 0$
has been recently studied by the authors in \cite{eck}. The case
$\vartheta = 0$ and $\rho \geq 0$ gives rise to the following
viscous Cahn--Hilliard equation for the anisotropic Ginzburg--Landau energy
(\ref{eq:Eg}), where $w$ is a phase field approximation
to the (rescaled) temperature $u$:
\begin{subequations}
\begin{alignat}{2}
\tfrac12\,\lambda\,\varphi_t & = 
\nabla \,.\, (b(\varphi)\,\nabla\, w )
&& \mbox{in} \;\;\Omega_T:=\Omega\times(0,T)\,, \label{eq:CHa} \\ 
\tfrac12\,\cPsi\,\frac{a}\alpha\,w & = 
\epsilon\,
\frac\rho\alpha\,\mu(\nabla\,\varphi)\,\varphi_t
-\epsilon\,\nabla \,.\, A'(\nabla\, \varphi) +\epsilon^{-1}\,\Psi'(\varphi) 
\qquad && \mbox{in} \;\;\Omega_T\,, \label{eq:CHb} \\ 
\frac{\partial \varphi}{\partial \vec\nu} & = 0\,, \qquad
&& \mbox{on}\;\;\partial \Omega\times(0,T)\,, \label{eq:CHc} \\
w & = \uD
&& \mbox{on}\;\;\partial_D \Omega\times(0,T)\,, \label{eq:CHd} \\
b(\varphi)\,\frac{\partial w}{\partial \vec\nu} & = 0 \,,\qquad
&& \mbox{on}\;\;\partial_N \Omega\times(0,T)\,, \label{eq:CHe} \\
\qquad \varphi(\cdot,0) & = \varphi_0 && \mbox{in} \;\;\Omega\,, \label{eq:CHf}
\end{alignat}
\end{subequations}
where
\begin{equation} \label{eq:b}
b(s) = \tfrac12\,(1+s)\,\conduct_+ + \tfrac12\,(1-s)\,\conduct_-\,.
\end{equation}
With the help of formal asymptotics, see e.g.\
\cite{McFaddenWBCS93,WheelerM96,BellettiniP96}, it can be shown that the sharp
interface limit of (\ref{eq:CHa}--f), i.e.\ the limit as $\epsilon\to0$,
is given by the 
quasi-static Stefan problem (or Mullins--Sekerka problem)
(\ref{eq:1a}--e) with $\vartheta = 0$, and with $u$ denoting the sharp
interface limit of $w$. 

We remark that the phase field analogue of the sharp interface energy identity
(\ref{eq:testD}) in the case $\vartheta=0$ is given by the formal 
energy bound
\begin{equation} \label{eq:pfLyap}
\ddt\left(\frac{\lambda\,\alpha}{a}\,\frac1\cPsi\,\mathcal{E}_\gamma(\varphi) 
- \tfrac12\,\lambda\,\uD\,\int_\Omega \varphi \dx \right) 
+ (b(\varphi)\,\nabla\,w, \nabla\, w)
+ \epsilon\,
\frac{\lambda\,\rho}{a}\,\frac1\cPsi
\left(\mu(\nabla\,\varphi), (\varphi_t)^2\right)
 \leq 0
\end{equation}
for the phase field model (\ref{eq:CHa}--f) with the potential
(\ref{eq:obstacle}). For smooth potentials such as (\ref{eq:quartic}) the
energy law (\ref{eq:pfLyap}) holds with equality.

\subsection{Heat equation coupled to Allen--Cahn}

The second phase field model is based on the work in
\cite{ElliottG96preprint}, see also \cite{Kobayashi93,KarmaR96} 
for other related approaches, and allows the sharp interface limit 
(\ref{eq:1a}--e) with $\vartheta \geq 0$. 
It consists of a heat equation for the phase field temperature approximation 
$w$ coupled to an Allen--Cahn phase field
equation for $\varphi$. In particular, we have the modified heat equation
\begin{subequations}
\begin{alignat}{2} \label{eq:heata} 
\vartheta\,w_t + \lambda\,\varrho(\varphi)\,\varphi_t & =
\nabla \,.\, (b(\varphi)\,\nabla\, w )  \qquad && \mbox{in } \Omega_T\,, \\
w & = \uD 
&& \mbox{on}\;\;\partial_D \Omega\times(0,T)\,, \label{eq:heatb} \\
b(\varphi)\,\frac{\partial w}{\partial \vec\nu} & = 0 \qquad
&& \mbox{on}\;\;\partial_N \Omega\times(0,T)\,, \label{eq:heatc} \\
\vartheta\,w(\cdot,0) & = \vartheta\,w_0 \quad && \mbox{in } 
\Omega \,, \label{eq:heatd} 
\end{alignat}
\end{subequations}
where $b$ is defined in (\ref{eq:b}), and where the %shape 
function $\varrho \in C^1(\R)$ is such that
\begin{equation*} % \label{eq:Varrho}
\varrho(s) \geq 0\quad\forall\ s \in[-1,1]\,,\quad
\int^1_{-1} \varrho(y) \; {\rm d}y = 1 \quad \text{and} \quad
\Varrho(s) := \int^s_{-1} \varrho(y) \; {\rm d}y\,.
\end{equation*}
We note that $\Varrho$, which is a monotonically increasing function over the 
interval $[-1,1]$ with $\Varrho(-1) = 0$ and $\Varrho(1) = 1$, 
is often called the interpolation
function. In this paper, we follow the convention from 
\cite{ElliottG96preprint}, where 
$\varrho = \Varrho'$ is called the shape function. 
More details on interpolation functions $\Varrho$, respectively
shape functions $\varrho$, can be found in e.g.\
\cite{WangSWMCM93,GarckeS06,EckGS06,CaginalpCE08}. 
In particular, if one also assumes
symmetry, i.e.\
\begin{equation*} % \label{eq:varrhosym}
\varrho(s) = \varrho(-s) \quad\forall\ s \in[-1,1]\,,
\end{equation*}
then a faster convergence of the phase field model to the sharp interface
limit, as $\epsilon\to0$, can be shown on prescribing suitable first order
corrections in $\epsilon$ for the remaining phase field parameters;
see \cite{KarmaR96,KarmaR98,EckGS06,GarckeS06} for details.
Possible choices of $\varrho$ that will be considered in this paper are
\begin{equation} \label{eq:varrho}
\text{(i)}\ 
\varrho(s) = \tfrac12\,,\qquad
\text{(ii)}\ 
\varrho(s) = \tfrac12\,(1 - s)\,,\qquad
\text{(iii)}\ 
\varrho(s) = \tfrac{15}{16}\,(s^2 - 1)^2\,.
\end{equation}
The heat equation (\ref{eq:heata}--d) 
is coupled to the following modified Allen--Cahn equation:
\begin{subequations}
\begin{alignat}{2}
\cPsi\,\frac{a}\alpha\,\varrho(\varphi)\,w & = 
\epsilon\,\frac\rho\alpha\,\mu(\nabla\,\varphi)\,\varphi_t
-\epsilon\,\nabla \,.\, A'(\nabla\, \varphi) +\epsilon^{-1}\,\Psi'(\varphi) 
\qquad && \mbox{in} \;\;\Omega_T\,, \label{eq:ACa} \\ 
\frac{\partial \varphi}{\partial \vec\nu} & = 0 \qquad
&& \mbox{on}\;\;\partial \Omega\times(0,T)\,, \label{eq:ACb} \\
\qquad \varphi(\cdot,0) & = \varphi_0 && \mbox{in} \;\;\Omega\,. \label{eq:ACc}
\end{alignat}
\end{subequations}
We remark that the phase field analogue of the sharp interface energy identity
(\ref{eq:testD}) is given by the formal energy bound
\begin{align} \label{eq:pf2Lyap}
& \ddt\left(
\frac\vartheta2\,|w - \uD|_0^2 +
\frac{\lambda\,\alpha}{a}\,\frac1\cPsi\,\mathcal{E}_\gamma(\varphi) - 
\lambda\,\uD\,\int_\Omega \Varrho(\varphi) \dx \right) 
+ (b(\varphi)\,\nabla\,w, \nabla\, w)
\nonumber \\ & \hspace{7cm}
+ \epsilon\,
\frac{\lambda\,\rho}{a}\,\frac1\cPsi
\left(\mu(\nabla\,\varphi), (\varphi_t)^2\right)  \leq 0
\end{align}
for the phase field model (\ref{eq:heata}--d), (\ref{eq:ACa}--c)
with the potential
(\ref{eq:obstacle}). For smooth potentials such as (\ref{eq:quartic}) the
energy law (\ref{eq:pf2Lyap}) holds with equality.
We remark that the energy decay in (\ref{eq:pf2Lyap}) for the phase field model
{\rm (\ref{eq:heata}--d)}, {\rm (\ref{eq:ACa}--c)} means that the model can
be said to be thermodynamically consistent. For more details on 
thermodynamically consistent phase field models we refer to e.g.\
\cite{PenroseF90,WangSWMCM93}.

\begin{rem} \label{rem:one}
We remark that in the special case $\vartheta = 0$, and if we choose 
{\rm (\ref{eq:varrho})(i)}, 
then clearly {\rm (\ref{eq:heata}--d)}, {\rm (\ref{eq:ACa}--c)} collapses
to the system {\rm (\ref{eq:CHa}--f)}. 
Similarly, the energy law {\rm (\ref{eq:pf2Lyap})} in this case collapses to 
{\rm (\ref{eq:pfLyap})}. 
Hence from now on in this paper, we will only consider the
more general model {\rm (\ref{eq:heata}--d)}, {\rm (\ref{eq:ACa}--c)}.
Finally we note that the phase field model {\rm (\ref{eq:CHa}--f)} in the case
$\rho = 0$ was recently considered in \cite{eck}.
\end{rem}

We observe that for $\epsilon$ small, on recalling that the thickness of the
interfacial region goes to zero as $\epsilon\to0$, it holds that
\begin{equation} \label{eq:Pterm}
\int_\Omega \Varrho(\varphi) \dx \approx
\Varrho(1)\,|\Omega_+^\epsilon(t)| + \Varrho(-1)\,|\Omega_-^\epsilon(t)|
= |\Omega_+^\epsilon(t)| \,,
\end{equation}
which is a consequence of the fact that %i.e.\ 
$\Varrho(\varphi)$ approximates the characteristic function of the liquid
phase $\Omega_+(t)$.
It is clear from (\ref{eq:pf2Lyap}) and (\ref{eq:Pterm}) that for negative values of $\uD$, $\varphi$
is encouraged to take on negative values, so that the approximate
liquid region $\Omega_+^\epsilon(t)$ shrinks, 
whereas positive values of $\uD$ encourage $\varphi$ to
take on positive values, so that the liquid region grows. Of course, this is
simply the phase field analogue of the sharp interface behaviour induced by 
(\ref{eq:testD}).
A side effect of the interpolation function $\Varrho$ in (\ref{eq:pf2Lyap}),
however, is that the function
\begin{equation} \label{eq:G}
G(s) = \alpha\,(a\,\cPsi\,\epsilon)^{-1}\,\Psi(s) - \uD\,\Varrho(s)
\end{equation}
need no longer have local minima at $s = \pm1$. This can result, for example, 
in undesired, artificial boundary layers for strong supercoolings, i.e.\ when
$-\uD$ is large; see also Remarks~\ref{rem:bl} and \ref{rem:sbl} below.
For smooth potentials $\Psi$, 
sufficient conditions for $s=\pm1$ to be local minimum points of $G(s)$ 
are $\varrho(\pm1)=\varrho'(\pm1)=0$, which is evidently satisfied by
(\ref{eq:varrho})(iii). 
In fact, in applications phase field models for solidification almost
exclusively use the quartic potential {\rm (\ref{eq:quartic})} together with
this shape function; see e.g.\ \cite{BoettingerWBK02,Chen02,McFadden02}. 

For the obstacle potential (\ref{eq:obstacle}) the situation is similar,
although there is more flexibility in the possible choices of $\varrho$. 
In particular, here a
sufficient condition for $G(s)$ to have local minima at $s=\pm1$ is given by
\begin{equation} \label{eq:obstcond}
\alpha\,(a\,\cPsi\,\epsilon)^{-1} \pm \uD\,\varrho(\pm1) \geq 0\,.
\end{equation}
Clearly, (\ref{eq:obstcond}) is always satisfied for (\ref{eq:varrho})(iii),
while for $\uD < 0$ it is sufficient to require $\varrho(1) = 0$, e.g.\ by
choosing (\ref{eq:varrho})(ii). A major advantage of (\ref{eq:varrho})(ii)
over (\ref{eq:varrho})(iii) is that for the former it will be possible to
derive almost linear finite element approximations that are 
unconditionally stable. The corresponding unconditionally stable schemes for
the nonlinear shape function (\ref{eq:varrho})(iii), on the other hand, turn
out to be more nonlinear.
Conversely, if $\uD>0$, then only $\varrho(-1) = 0$ is needed in order to
satisfy (\ref{eq:obstcond}). 
The natural analogue for {\rm (\ref{eq:varrho})(ii)} in
this situation is then 
\begin{equation} \label{eq:varrhoii}
\varrho(s) = \tfrac12\,(1 + s)\,,
\end{equation}
and once again it is
possible to derive almost linear finite element approximations that are 
unconditionally stable for this choice of $\varrho$.

Finally we note that the quartic potential {\rm (\ref{eq:quartic})} is often 
preferred in applications because the
discretized equations can then be solved with smooth solution methods, such as
the Newton method. However, the quartic potential has the disadvantage that 
a priori it cannot be guaranteed that $|\varphi| \leq 1$ at all times,
and in practice it can in general be observed that discretizations of $\varphi$
exceed the interval $[-1,1]$. Hence from a practical and from a numerical
analysis point of view it is preferable to use the obstacle potential 
{\rm (\ref{eq:obstacle})}. Here we note that the discretized equations, which
feature variational inequalities, can be efficiently solved with a variety of
modern solution methods; see e.g.\ 
\cite{voids,GraserK07,voids3d,mgch,BlankBG11,HintermullerHT11,%
GraserKS12preprint}.

\subsection{Anisotropies}

In this paper, we will only consider smooth and convex anisotropies, i.e.\
they satisfy
\begin{equation}
\gamma'(\vec{p})\,.\,\vec{q} \leq \gamma(\vec{q})
\quad \forall\ \vec p \in \R^d\setminus\{\vec 0\}\,, \vec q \in \R^d\,,
\label{eq:gest}
\end{equation}
which, on recalling (\ref{eq:homo}), is equivalent to
\begin{equation} \label{eq:convex}
\gamma(\vec{p}) + \gamma'(\vec{p})\,.\,(\vec{q}-\vec{p})\leq\gamma(\vec{q})
\quad \forall\ \vec p \in \R^d\setminus\{\vec 0\}\,, \vec q \in \R^d\,.
\end{equation}

It is the aim of this paper to introduce unconditionally stable finite element
approximations for the phase field models (\ref{eq:CHa}--f) and 
(\ref{eq:heata}--d), (\ref{eq:ACa}--c). 
Based on earlier work by the authors in the context of the
parametric approximation of anisotropic geometric evolution equations
\cite{triplejANI,ani3d}, the crucial idea here is to restrict the class of
anisotropies under consideration. 
The special structure of the chosen anisotropies
can then be exploited to develop discretizations that are stable without the
need for regularization and without a restriction on the time step size.

In particular, the class of anisotropies that we will consider in this paper
is given by
\begin{equation} \label{eq:g}
\gamma(\vec{p}) = \left(\sum_{\ell=1}^L
[\gamma_{\ell}(\vec{p})]^r\right)^{\frac1r}, \quad
\gamma_\ell(\vec{p}):= [{\vec{p}\,.\,G_{\ell}\,\vec{p}}]^\frac12\,,
\qquad \forall\ \vec p \in \R^d\,,\qquad r \in [1,\infty)\,,
\end{equation}
where $G_{\ell} \in \R^{d\times d}$, for $\ell=1\to L$, 
are symmetric and positive definite matrices. 
This class of anisotropies 
has been previously considered by the authors in \cite{ani3d,dendritic}. 
We remark
that anisotropies of the form (\ref{eq:g}) are always strictly convex norms.
In particular, they satisfy (\ref{eq:convex}).
However, despite this seemingly restrictive choice, it is possible with
(\ref{eq:g}) to model and approximate a wide variety of anisotropies that are
relevant in materials science. For the sake of brevity, we refer to the
exemplary Wulff shapes in the authors' previous papers 
\cite{triplejANI,ani3d,clust3d,dendritic,ejam3d,crystal}. 
We remark that in the case $r=1$
all of the numerical schemes introduced in 
Section~\ref{sec:3}, below, will feature no additional nonlinearities compared
to the isotropic case (\ref{eq:iso}). In particular, the finite element
approximation in Section~\ref{sec:31} for the obstacle potential
(\ref{eq:obstacle}) will feature only linear equations and linear 
variational inequalities; see also \cite{eck}. 
Finally, we note that in the two-dimensional case ($d=2$),
the anisotropies (\ref{eq:g}) with the choice $r=1$
adequately approximate most relevant anisotropies.
However, in the three-dimensional setting ($d=3$), it is often necessary to use
$r>1$ in (\ref{eq:g}) in order to model a chosen anisotropy. See \cite{ani3d}
for more details.

In the following, we establish some crucial results for anisotropies of the
form (\ref{eq:g}). 
Note that for $\gamma$ satisfying (\ref{eq:g}) it holds that
\begin{equation} \label{eq:Aprime}
A'(\vec p) = \gamma(\vec p)\,\gamma'(\vec p) %= \gamma(\vec p)\,
\,, \quad\text{where}\quad \gamma'(\vec p) = 
\sum_{\ell = 1}^L \left[\frac{\gamma_\ell(\vec p)}{\gamma(\vec p)}\right]^{r-1}
\gamma_\ell'(\vec p)
\qquad \forall\ \vec p \in \R^d\setminus\{\vec 0\}\,.
\end{equation}

For later use we recall 
the elementary identity 
\begin{equation} \label{eq:element}
2\,y\,(y-z) = y^2 - z^2 + (y-z)^2\,.
\end{equation}
Moreover, from now on we use the convention that
\begin{equation} \label{eq:00}
\frac{\gamma_\ell(\vec p)}{\gamma(\vec p)} := 1 \qquad \text{if}\quad 
\vec p = \vec 0\,, \qquad \ell = 1\to L\,.
\end{equation}

\begin{lem}\label{lem:1}
Let $\gamma$ be of the form {\rm (\ref{eq:g})}. Then it holds that
\begin{equation} \label{eq:gr1}
\gamma(\vec p) \leq L^{\frac1{r\,(r+1)}}\,
\left(\sum_{\ell=1}^L
[\gamma_{\ell}(\vec{p})]^{r+1}\right)^{\frac1{r+1}}
\qquad \forall\ \vec p \in \R^d\,.
\end{equation}
Moreover, $\gamma$ is convex and
the anisotropic operator $A$ satisfies
\begin{alignat}{2} \label{eq:mono}
 A'(\vec p) \,.\, (\vec p-\vec q) & \geq 
 \gamma(\vec p)\,[\gamma(\vec p)-\gamma(\vec q) ]
 \qquad && \forall\ \vec p \in \R^d\setminus\{\vec 0\}\,, \vec q \in \R^d\,, \\
 A(\vec p) & \leq \tfrac12\,\gamma(\vec q)\,
 \sum_{\ell=1}^L \left[ \frac{\gamma_\ell(\vec p)}{\gamma(\vec p)}\right]^{r-1}
[\gamma_\ell(q)]^{-1}\,[\gamma_\ell(p)]^2 
 \qquad && \forall\ \vec p\in \R^d\,, \vec q \in \R^d\setminus\{\vec 0\}\,, 
 \label{eq:CS}
\end{alignat}
where in {\rm (\ref{eq:CS})} we recall the convention {\rm (\ref{eq:00})}. 
\end{lem}
\begin{proof}
It follows from a H\"older inequality that
\begin{equation*} 
[\gamma(\vec p)]^r \leq L^{\frac1{r+1}}\,
\left(\sum_{\ell=1}^L
[\gamma_{\ell}(\vec{p})]^{r+1}\right)^{\frac{r}{r+1}}
\qquad \forall\ \vec p \in \R^d\,,
\end{equation*}
which immediately yields the desired result (\ref{eq:gr1}). 
Next we prove (\ref{eq:gest}). It follows from (\ref{eq:Aprime}), a 
Cauchy--Schwarz and a H\"older inequality that
\begin{align*}
\gamma'(\vec p) \,.\,\vec q & = 
\sum_{\ell = 1}^L \left[\frac{\gamma_\ell(\vec p)}{\gamma(\vec p)}\right]^{r-1}
[\gamma_\ell(\vec p)]^{-1}\,(G_\ell\,\vec p)\,.\,\vec q \leq 
\sum_{\ell = 1}^L \left[\frac{\gamma_\ell(\vec p)}{\gamma(\vec p)}\right]^{r-1}
\gamma_\ell(\vec q) \\ & \leq
\left( \sum_{\ell = 1}^L 
\left[\frac{\gamma_\ell(\vec p)}{\gamma(\vec p)}\right]^r \right)^{\frac{r-1}r}
\left( \sum_{\ell = 1}^L [\gamma_\ell(\vec q)]^r \right)^{\frac1r}
= \gamma(\vec q)
\qquad \forall\ \vec p \in \R^d\setminus\{\vec 0\}\,, \vec q \in \R^d\,.
\end{align*}
Together with (\ref{eq:homo}) this implies (\ref{eq:convex}), i.e.\ $\gamma$ is
convex. Multiplying (\ref{eq:convex}) with $\gamma(\vec p)$ yields the desired 
result (\ref{eq:mono}).
Moreover, we have from a H\"older inequality that 
\begin{align*} % \label{eq:CS1}
[\gamma(\vec p)]^r & = \sum_{\ell = 1}^L 
[\gamma_\ell(q)]^{\frac{r}{r+1}}\,
\frac{[\gamma_\ell(p)]^r}{[\gamma_\ell(q)]^{\frac{r}{r+1}}} \leq
\left( \sum_{\ell = 1}^L [\gamma_\ell(\vec q)]^r \right)^\frac1{r+1}
\left( \sum_{\ell = 1}^L \frac{[\gamma_\ell(p)]^{r+1}}{\gamma_\ell(q)} 
\right)^{\frac{r}{r+1}} \\ \Rightarrow \quad 
[\gamma(\vec p)]^{r+1} & \leq \gamma(\vec q)\,
\sum_{\ell = 1}^L [\gamma_\ell(p)]^{r+1}\,[\gamma_\ell(q)]^{-1}
\qquad \forall\ \vec p \in \R^d\,, \vec q \in \R^d\setminus\{\vec 0\}\,.
\end{align*}
This immediately yields the desired result (\ref{eq:CS}), on recalling 
(\ref{eq:Ap}).
\end{proof}

Our aim now is to replace the highly nonlinear operator 
$A'(\vec p):\R^d\to\R^d$ in 
(\ref{eq:Aprime}) with an almost linear approximation (linear for $r=1$) 
that still maintains the
crucial monotonicity property (\ref{eq:mono}). It turns out that a
natural linearization is already given in (\ref{eq:Aprime}). 
In particular, we let  
\begin{equation} \label{eq:Br}
B_r(\vec q, \vec p) := \begin{cases}
\gamma(\vec q)\,\displaystyle\sum_{\ell = 1}^L 
\left[\frac{\gamma_\ell(\vec p)}{\gamma(\vec p)}\right]^{r-1}
[\gamma_\ell(\vec q)]^{-1}\,G_\ell
& \vec q \not= \vec 0\,, \\
L^{\frac1r}\,\displaystyle\sum_{\ell = 1}^L 
\left[\frac{\gamma_\ell(\vec p)}{\gamma(\vec p)}\right]^{r-1}
G_\ell & \vec q = \vec 0\,,
\end{cases}
\qquad \forall\ \vec p \in \R^d\,,
\end{equation}
where in the case $\vec p = \vec 0$ we recall (\ref{eq:00}). For later use we
note for $\vec q \in \R^d$ that
\begin{equation} \label{eq:r1}
B_1(\vec q, \vec p) = B_1(\vec q, \vec 0) =: B_1(\vec q) \qquad 
\forall\ \vec p \in \R^d\,.
\end{equation}
Clearly it holds that
$$
B_r(\vec p, \vec p)\,\vec p = A'(\vec p) 
\qquad \forall\ \vec p \in \R^d\setminus\{\vec 0\}\,,
$$
and it turns out that approximating $A'(\vec p)$ with 
$B_r(\vec q, \vec p)\,\vec p$
maintains the monotonicity property (\ref{eq:mono}).

\begin{lem} \label{lem:B}
Let $\gamma$ be of the form {\rm (\ref{eq:g})}. Then it holds that
\begin{equation} \label{eq:Bmono}
[B_r(\vec q, \vec p)\,\vec p]\,.\,(\vec p-\vec q) 
\geq \gamma(p)\,\left[\gamma(p)-\gamma(q)\right] 
\qquad \forall\ \vec p \,, \vec q \in \R^d \,.
\end{equation}
\end{lem}
\begin{proof}
If $\vec p = 0$ then (\ref{eq:Bmono}) trivially holds.
Now let $\vec p \in \R^d\setminus\{\vec0\}$.
If $\vec q \neq \vec 0$ it holds, on recalling (\ref{eq:CS}), that
\begin{align*}
[B_r(\vec q, \vec p)\,\vec p]&\,.\,(p-q) = 
\gamma(q)\, \sum_{\ell=1}^L
\left[\frac{\gamma_\ell(\vec p)}{\gamma(\vec p)}\right]^{r-1}
 [\gamma_\ell(q)]^{-1}\,(p-q)\,.\,G_\ell\,p \\ & 
\geq \gamma(q) \,\sum_{\ell=1}^L 
\left[\frac{\gamma_\ell(\vec p)}{\gamma(\vec p)}\right]^{r-1}
 \gamma_\ell(p)\,([\gamma_\ell(q)]^{-1}\,\gamma_\ell(p) - 1) \\
& = \gamma(q) \,\sum_{\ell=1}^L 
\left[\frac{\gamma_\ell(\vec p)}{\gamma(\vec p)}\right]^{r-1}
[\gamma_\ell(q)]^{-1}\,[\gamma_\ell(p)]^2 
 - \gamma(q)\, \gamma(p) %\\ &
\geq \gamma(p)\,\left[\gamma(p)-\gamma(q)\right] .
\end{align*}
If $\vec q = \vec 0$, on the other hand, then it follows from (\ref{eq:gr1}) 
that
\begin{align*}
[B_r(\vec q, \vec p)\, \vec p]\,.\,(p-q) & = 
[B_r(\vec 0, \vec p)\,\vec p]\,.\,p  
= L^{\frac1r} [\gamma(\vec p)]^{1-r}\,
\sum_{\ell=1}^L [\gamma_\ell(\vec p)]^{r+1} 
\geq %[\gamma(\vec p)]^{1-r}\,[\gamma(\vec p)]^{r+1} = 
[\gamma(p)]^2\,.
\end{align*}
\end{proof}

\begin{cor} \label{cor:B}
Let $\gamma$ be of the form {\rm (\ref{eq:g})}. Then it holds that
\begin{equation} \label{eq:Bstab}
[B_r(\vec q, \vec p)\, \vec p]\,.\,(\vec p-\vec q) 
\geq A(\vec p) - A(q) 
\qquad \forall\ \vec p \,, \vec q \in \R^d \,.
\end{equation}
\end{cor}
\begin{proof}
The desired result follows immediately from Lemma~\ref{lem:B} on noting the
elementary identity (\ref{eq:element}). %$2\,r\,(r-s) = r^2 - s^2 + (r-s)^2$.
\end{proof}

\setcounter{equation}{0} 
\section{Finite element approximations} \label{sec:3}
Let $\Omega$ be a polyhedral domain and 
let $\{{\cal T}^h\}_{h>0}$ be a family of partitionings of $\Omega$ into
disjoint open simplices $\sigma$ with $h_{\sigma}:={\rm diam}(\sigma)$
and $h:=\max_{\sigma \in {\cal T}^h}h_{\sigma}$, so that
$\overline{\Omega}=\cup_{\sigma\in{\cal T}^h}\overline{\sigma}$.
Associated with ${\cal T}^h$ is the finite element space
\begin{equation*} % \label{eq:Sh}
 S^h :=  \{\chi \in C(\overline{\Omega}) : \chi \mid_{\sigma} \mbox{ is linear }
 \forall\ \sigma \in {\cal T}^h\} \subset H^1(\Omega).
\end{equation*}
Let $J$ be the set of nodes of ${\cal T}^h$ and $\{p_{j}\}_{j \in J}$ the
coordinates of these nodes.
Let $\{\chi_{j}\}_{j\in J}$ be the standard basis
functions for $S^h$; that is $\chi_{j} \in S^h$ and $\chi_j(p_{i})=\delta_{ij}$
for all $i,j \in J$.
We introduce $\pi^h:C(\overline{\Omega})\rightarrow S^h$, the interpolation
operator, such that $(\pi^h \eta)(p_j)= \eta(p_j)$ for all $j \in J$. A
discrete semi-inner product on $C(\overline{\Omega})$ is then defined by
\begin{equation*} %\label{eq:dip}
 (\eta_1,\eta_2)^h := \int_\Omega \pi^h(\eta_1(x)\,\eta_2(x))\dx 
\end{equation*}
with the induced discrete semi-norm given by
$|\eta|_h := [\,(\eta,\eta)^h\,]^{\frac{1}{2}}$, for
$\eta\in C(\overline{\Omega})$. We extend these definitions to
functions that are piecewise continuous on $\mathcal{T}^h$ in the usual way,
i.e.\ by setting
\begin{equation*}
( \eta_1, \eta_2 )^h  :=
\sum_{\sigma \in\mathcal{T}^h} ( \eta_1, \eta_2 )^h_\sigma\,,
\end{equation*}
where
$$
( \eta_1, \eta_2 )^h_\sigma := \frac{|\sigma|}{d+1}\,
\sum_{k=0}^{d} 
(\eta_1\,\eta_2) ( (\vec{p}_{j_k})^-),
$$
with $\{\vec{p}_{j_k}\}_{k=0}^{d}$ denoting the vertices of $\sigma$,
and where we define $\eta((\vec{p}_{j_k})^-):=
\underset{\sigma\ni \vec{q}\to \vec{p}_{j_k}}{\lim}\, \eta(\vec{q})$,
$k=0\to d$.

We introduce also
\begin{align*} 
 K^h & := \{\chi \in S^h : |\chi| \leq 1 \mbox{ in } \Omega \}
 \subset K 
 :=\{\eta \in H^1(\Omega) : |\eta| \leq 1 \mbox{ $a.e.$ in }\Omega\}\,, \\
 S^h_0& := \{\chi \in S^h : \chi = 0 \ \mbox{ on $\partial_D\Omega$} \}  
\quad \mbox{and} \quad
 \ShD:= \{\chi \in S^h : \chi = \pi^h\uD\ 
\mbox{ on $\partial_D\Omega$} \} 
\,,
\end{align*}
where in the definition of $\ShD$ we allow for
$\uD \in H^\frac12(\partial_D\Omega) \cap C(\overline{\partial_D\Omega})$.

In addition to ${\cal T}^h$, let
$0= t_0 < t_1 < \ldots < t_{N-1} < t_N = T$ be a
partitioning of $[0,T]$ into possibly variable time steps $\tau_n := t_n -
t_{n-1}$, $n=1\rightarrow N$. We set
$\tau := \max_{n=1\rightarrow N}\tau_n$. 

In the following we will present stable finite element approximations for the
phase field model (\ref{eq:heata}--d), (\ref{eq:ACa}--c) for the obstacle 
potential (\ref{eq:obstacle}) and for the case of a smooth potential such as
(\ref{eq:quartic}), respectively. In order to obtain stable approximations,
the three nonlinearities arising in (\ref{eq:ACa}) from $\varrho(\varphi)$, 
from $A'(\nabla\,\varphi)$ and from $\Psi'(\varphi)$ need to be discretized
appropriately in time. Here the discretization of $A'(\nabla\,\varphi)$ induced
by Corollary~\ref{cor:B} is novel, and is one of the main contributions of this
paper.
The employed splitting of $\Psi'(\varphi)$ 
into implicit/explicit time
discretizations according to a convex/concave splitting of $\Psi$, 
on the other hand, is standard; 
see e.g.\ \cite{ElliottS93,BarrettBG99}.
We employ the same idea to the splitting of $\varrho(\varphi)$, for which we
now introduce some notation. A similar notation will be used in
Section~\ref{sec:32} for the splitting of $\Psi'(\varphi)$ in the case of a
smooth potential $\Psi$.

Let $\varrho^\pm \in C^1(\R)$ such that 
$\varrho(s) = \varrho^+(s) + \varrho^-(s)$. In our finite element
schemes $\varrho^+$ will play the role of the implicit part of the
approximation of $\varrho$, while $\varrho^-$ corresponds to the
explicit part.
We now define
\begin{equation} \label{eq:hatrho}
\hatrho(s_0, s_1) := \varrho^+(s_1) + \varrho^-(s_0) \quad \forall\
s_0,s_1\in\R\,,
\end{equation}
as well as $\Varrho^\pm(s) := \int_{-1}^s \varrho^\pm(y)\;{\rm d}y$.
Of particular interest will be splittings such that
\begin{equation} \label{eq:rhosplit}
\pm\,\uD\,(\varrho^\pm)'(s) \leq 0 \qquad 
\forall\ s \leq \tfrac2{\sqrt{3}} \,. %3^{-\frac12}\,2\,.
\end{equation}
If $\uD < 0$, then (\ref{eq:rhosplit}) enforces $\Varrho^+(s)$ to be convex for 
$s \leq \frac2{\sqrt{3}}$, while $\Varrho^-(s)$ is concave over the same 
region. Possible splittings satisfying 
(\ref{eq:rhosplit}) for the shape functions in 
(\ref{eq:varrho}) are then given by
\begin{alignat}{2}
\text{(i)\ }\ &
\varrho^+(s) = 0\,,\quad && \varrho^-(s) = \varrho(s) = \tfrac12\,, \nonumber\\
\text{(ii)\ }\ &
\varrho^+(s) = 0\,, \quad &&\varrho^-(s) = \varrho(s) = \tfrac12\,(1-s)\,, 
\label{eq:split} \\
\text{(iii) }\ &
\varrho^+(s) = \tfrac{3}{2}\,s\,, \quad &&
\varrho^-(s) = \varrho(s) - \tfrac{3}{2}\,s = 
\tfrac{15}{16}\,(s^4-2\,s^2-\tfrac85\,s+1)\,. \nonumber
\end{alignat}
The fact that the splitting (\ref{eq:split})(iii) satisfies (\ref{eq:rhosplit}) 
follows from the observation that in that case 
$\max_{s \leq \frac2{\sqrt{3}}} \varrho'(s) = \varrho'(\frac2{\sqrt{3}})
= \tfrac{15}{16}\,\frac8{\sqrt{27}} < \tfrac32$.
Note that the above splittings were chosen such that the implicit part of the
approximation of $\varrho$ is as simple as possible.
If $\uD > 0$, on the other hand, then swapping the roles of $\varrho^\pm$ in
(\ref{eq:split}) will satisfy (\ref{eq:rhosplit}). 
However, as the implicit parts
$\varrho^+$ are then unnecessarily nonlinear in the cases (\ref{eq:varrho})(ii)
and (\ref{eq:varrho})(iii), it is more convenient, on recalling
(\ref{eq:varrhoii}), to use the splittings 
\begin{alignat}{2}
\text{(ii)\ }\ &
\varrho^+(s) = 0\,, \quad &&\varrho^-(s) = \varrho(s) = \tfrac12\,(1+s)\,, 
\label{eq:split+} \\
\text{(iii)\ }\ &
\varrho^+(s) = -\tfrac{3}{2}\,s\,,\quad && 
\varrho^-(s) = \varrho(s) + \tfrac32\,s = 
\tfrac{15}{16}\,(s^4-2\,s^2+\tfrac85\,s+1)\,, \nonumber
\end{alignat}
which will then satisfy 
\begin{equation} \label{eq:rhosplit+}
\pm\,\uD\,(\varrho^\pm)'(s) \leq 0 \qquad 
\forall\ s \geq -\tfrac2{\sqrt{3}} \,. %3^{-\frac12}\,2\,.
\end{equation}

\subsection{The obstacle potential} \label{sec:31}

We then consider the following fully practical finite element
approximation for (\ref{eq:heata}--d), (\ref{eq:ACa}--c) in the case of the
obstacle potential (\ref{eq:obstacle}). 
This approximation is an adaptation of the scheme from
\cite{ElliottG96preprint} which, 
with the help of Corollary~\ref{cor:B}, can be shown to be stable. 
Let $\Phi^0 \in K^h$ be an approximation of $\varphi_0 \in K$, e.g.\
$\Phi^0 = \pi^h \varphi_0$ for $\varphi_0 \in C(\overline{\Omega})$. 
Similarly, if $\vartheta>0$ let $W^0 \in \ShD$ be an approximation of 
$u_0$.
Then, for $n \geq 1$, find $(\Phi^n,W^n) \in K^h \times \ShD$ such that
\begin{subequations}
\begin{align}
& \vartheta \left(\dfrac{W^{n}-W^{n-1}}{\tau_n},
 \chi \right)^h + 
\lambda\,\left(\hatrho(\Phi^{n-1}, \Phi^n)\,\dfrac{\Phi^{n}-\Phi^{n-1}}{\tau_n},
 \chi \right)^h 
+ ( \pi^h[b(\Phi^{n-1})]\,\nabla\, W^{n} , \nabla\, \chi )
 = 0 \nonumber \\ & \hspace{10cm}
\qquad \forall\ \chi \in S^h_0, \label{eq:U} \\
& 
\epsilon\,\frac\rho\alpha\left(\mu(\nabla\,\Phi^{n-1})\,
\dfrac{\Phi^{n}-\Phi^{n-1}}{\tau_n}, \chi - \Phi^n\right)^h +
\epsilon\,(B_r(\nabla\, \Phi^{n-1}, \nabla\, \Phi^n)
\, \nabla\, \Phi^{n}, \nabla\, [\chi -\Phi^n]) \nonumber \\ & \qquad\qquad
 \geq \left(\cPsi\,\frac{a}\alpha\,\hatrho(\Phi^{n-1}, \Phi^n)\,W^{n} 
 + \epsilon^{-1}\,\Phi^{n-1},\chi - \Phi^n\right)^h
 \qquad \forall\ \chi \in K^h\,. \label{eq:ACh}
\end{align}
\end{subequations}
The main differences between (\ref{eq:U},b) for $\varrho^+=0$, so that
$\hatrho(\Phi^{n-1}, \Phi^n) = \varrho(\Phi^{n-1})$, 
and the basic scheme in \citet[Eqs.\ (3.1), (3.2)]{ElliottG96preprint} are
our novel approximation of $A'(\nabla\,\varphi)$ in (\ref{eq:ACh}) and the fact
that we evaluate the discrete temperature on the new time level $W^n$ in
(\ref{eq:ACh}). The latter implies that the system (\ref{eq:U},b) is coupled,
and this is needed in order to derive a stability bound, see
Theorem~\ref{thm:stab2}, below. We stress that there is no stability result for
the scheme \citet[Eqs.\ (3.1), (3.2)]{ElliottG96preprint}. 
In addition, we allow for
the splitting $\varrho = \varrho^+ + \varrho^-$, so that unconditional
stability can still be shown for nonlinear functions $\varrho$.

Let 
\begin{equation*} % \label{eq:Eh}
\mathcal{E}_\gamma^h(W, \Phi) = 
\frac\vartheta2\,|W - \uD|_h^2 +
\frac{\lambda\,\alpha}{a}\,\frac1\cPsi
\left[ \tfrac12\,\epsilon\,|\gamma(\nabla\, \Phi)|_0^2 + 
\epsilon^{-1}\,(\Psi(\Phi), 1)^h \right] \,,
\end{equation*}
and define
\begin{equation*} % \label{eq:Fh2}
\mathcal{F}_\gamma^h(W, \Phi) = \mathcal{E}_\gamma^h(W, \Phi)
- \lambda\,\uD\,(\Varrho(\Phi), 1)^h
\end{equation*}
for all $W, \Phi \in S^h$,
as the natural discrete analogue of the energy appearing in
(\ref{eq:pf2Lyap}). We can then show that the solutions to (\ref{eq:U},b)
satisfy a discrete analogue of (\ref{eq:pf2Lyap}).

We begin with considerations for the almost linear scheme (\ref{eq:U},b) with
$\varrho^+ = 0$ and $r=1$.

\begin{lem} \label{lem:ex}
Let $\gamma$ be of the form {\rm (\ref{eq:g})} with $r=1$,
let $\varrho^+=0$ and let $\uD \in \R$.
Then there exists a solution $(\Phi^n,W^n) \in K^h \times \ShD$ to 
{\rm (\ref{eq:U},b)} and $\Phi^n$, $W^n$ are unique up to additive constants.
If $\rho + (|\varrho(\Phi^{n-1})|, 1)^h + |(\Phi^{n-1},1)| > 0$, 
then $\Phi^n$ is unique. 
If $\vartheta > 0$ or $\partial_N\Omega\not=\partial\Omega$
then $W^n$ is unique if $\Phi^n$ is unique.
If $\vartheta= 0$ and $\partial_N\Omega=\partial\Omega$, and
if $\Phi^n$ is unique, then 
$W^n$ is unique if there exists a $j\in J$ such that
$|\Phi^n(p_j)| < 1$ and $\varrho(\Phi^{n-1}(p_j)) \not = 0$.
\end{lem}
\begin{proof}
The proof follows the ideas in \cite{BarrettBG99}, see also \cite{BloweyE92}.
At first we assume that $\partial_N\Omega \not= \partial\Omega$ or that
$\vartheta > 0$, so that $\mathcal{G}^h : S^h \to S^h_0$ such that
\begin{equation} \label{eq:Gh}
 ( \pi^h[b(\Phi^{n-1})]\,\nabla\, [\mathcal{G}^h\, v^h] , \nabla\, \eta )
+ \frac\vartheta{\tau_n} \,( \mathcal{G}^h\,v^h, \eta )^h 
= (v^h, \eta)^h \qquad \forall\ \eta \in S^h_0\,,\quad\ v^h \in S^h
\end{equation}
is clearly well-defined, on recalling that
\begin{equation*} % \label{eq:bbound}
b(s) \geq \min\{\conduct_+, \conduct_-\} > 0 \qquad \forall\ s \in[-1,1]\,.
\end{equation*}
Moreover, it follows from (\ref{eq:U}) and (\ref{eq:Gh}) that
\begin{equation} \label{eq:Wn}
W^n - \uD = \mathcal{G}^h\left[
\frac\vartheta{\tau_n} \,(W^{n-1} - \uD) - \lambda\,
\pi^h\Bigl[\hatrho(\Phi^{n-1}, \Phi^n)\,\dfrac{\Phi^{n}-\Phi^{n-1}}{\tau_n}\Bigr]
\right]\,.
\end{equation}
Substituting (\ref{eq:Wn}) into (\ref{eq:ACh}), and noting (\ref{eq:Gh}) with
$v^h = \pi^h[\hatrho(\Phi^{n-1},\Phi^n)\,(\chi-\Phi^{n})]$ and
$\eta = \mathcal{G}^h\,\pi^h[\hatrho(\Phi^{n-1},\Phi^n)\,(\Phi^n-\Phi^{n-1})]$ 
yields that
\begin{subequations}
\begin{align}
& \frac{\lambda\,\cPsi\,a}{\alpha\,\tau_n} \,\Bigl\{
( \pi^h[b(\Phi^{n-1})]\,\nabla\, [\mathcal{G}^h\,
\pi^h[\hatrho(\Phi^{n-1},\Phi^n)\,(\Phi^n-\Phi^{n-1})]], \nabla\,[
\mathcal{G}^h\,\pi^h[\hatrho(\Phi^{n-1},\Phi^n)\,(\chi-\Phi^{n})]] ) 
\nonumber \\ & \qquad 
+  \frac\vartheta{\tau_n} \,( 
\mathcal{G}^h\,\pi^h[\hatrho(\Phi^{n-1}, \Phi^n)\,(\Phi^n-\Phi^{n-1})],
\mathcal{G}^h\,\pi^h[\hatrho(\Phi^{n-1}, \Phi^n)\,(\chi-\Phi^{n})] )^h \Bigr\}
\nonumber \\ & \qquad
+ \frac{\epsilon\,\rho}{\alpha\,\tau_n}\,(\mu(\nabla\,\Phi^{n-1})\,\Phi^n, 
\chi - \Phi^n)^h
+ \epsilon\,(B_r(\nabla\, \Phi^{n-1}, \nabla\, \Phi^n)
\, \nabla\, \Phi^{n}, \nabla\, [\chi -\Phi^n])
\nonumber \\ & \quad
\geq (f^h, \chi - \Phi^n)^h \qquad \forall\ \chi \in K^h\,,
\label{eq:VI}
\end{align}
where
\begin{align}
f^h := (\frac{\epsilon\,\rho}{\alpha\,\tau_n}\,\mu(\nabla\,\Phi^{n-1})
+ \epsilon^{-1})\,\Phi^{n-1} + \cPsi\,\frac{a}\alpha\,\hatrho(\Phi^{n-1},\Phi^n)\,
(\uD +  \frac\vartheta{\tau_n} \, \mathcal{G}^h\,[W^{n-1}-\uD] ) 
\label{eq:fh}
\end{align}
\end{subequations}
is piecewise continuous on $\mathcal{T}^h$.
As we consider the case $\varrho^+ = 0$, from now on we use the fact that 
$\hatrho(\Phi^{n-1},\Phi^n) = \varrho(\Phi^{n-1})$.
We recall from (\ref{eq:Br}) and (\ref{eq:r1}) 
that $B_1(q) \in \R^{d\times d}$ is symmetric and
positive definite for all $\vec q \in \R^d$, and hence (\ref{eq:VI}) are the
Euler--Lagrange equations for the convex minimization problem
\begin{align*}
\min_{\chi \in K^h} & \left [
\frac{\lambda\,\cPsi\,a}{2\,\alpha\,\tau_n} \,\Bigl\{
( \pi^h[b(\Phi^{n-1})], |\nabla\, [\mathcal{G}^h\,
\pi^h[\varrho(\Phi^{n-1})\,(\chi-\Phi^{n-1})]]|^2 )
\right. \nonumber \\ & \quad \left. 
+ \frac\vartheta{2\,\tau_n} \, 
|\mathcal{G}^h\,\pi^h[\varrho(\Phi^{n-1})\,(\chi-\Phi^{n-1})] |^2_h \Bigr\}
+ \frac{\epsilon\,\rho}{2\,\alpha\,\tau_n}\,(\mu(\nabla\,\Phi^{n-1}), |\chi|^2
)^h
\right. \nonumber \\ & \quad \left. 
+ \frac\epsilon2\,(B_1(\nabla\,\Phi^{n-1})\,\nabla\,\chi, \nabla\,\chi) 
- (f^h, \chi)^h
\right] . % \label{eq:min}
\end{align*}
Therefore there exists a $\Phi^n \in K^h$ solving (\ref{eq:VI}) that is unique
if $\rho>0$ or $\pi^h[\varrho(\Phi^{n-1})] \not= 0 \in S^h$, and is unique up
to an additive constant otherwise. In the latter case, if 
$(\Phi^{n-1},1)\not=0$, then it immediately follows from (\ref{eq:ACh}) that
$\Phi^n$ is unique.
If $\Phi^n$ is unique, then
the existence of a unique $W^n \in S^h_D$, such that $(\Phi^n, W^n)$ solve 
(\ref{eq:U},b), follows from (\ref{eq:Wn}). 

For the remainder of the proof we assume that 
$\partial_N\Omega = \partial\Omega$ and that $\vartheta = 0$. 
Then it follows
immediately on choosing $\chi=1$ in (\ref{eq:U}) that
$(\varrho(\Phi^{n-1}), \Phi^n)^h = (\varrho(\Phi^{n-1}), \Phi^{n-1})^h$. Taking
this into account, we define
$\widehat{\mathcal{G}}^h : \widehat{S}^h \to \widehat{S}^h$ such that
\begin{equation*} % \label{eq:wGh}
 ( \pi^h[b(\Phi^{n-1})]\,\nabla\, [\widehat{\mathcal{G}}^h\, v^h], 
\nabla\, \eta)
= (v^h, \eta)^h \qquad \forall\ \eta \in S^h\,,\quad\ v^h \in \widehat{S}^h\,,
\end{equation*}
where $\widehat{S}^h := \{ \chi \in S^h : (\chi, 1) = 0 \}$, and observe that 
(\ref{eq:U}) then implies that
\begin{equation} \label{eq:wW}
W^n = - \frac\lambda{\tau_n}\,
\widehat{\mathcal{G}}^h\,\pi^h\,[\varrho(\Phi^{n-1})\,(\Phi^{n}-\Phi^{n-1})]
+ \xi^n\,,
\end{equation}
where $\xi^n\in \R$ is a Lagrange multiplier. It follows that (\ref{eq:VI})
holds with $\vartheta = 0$, with 
$\mathcal{G}^h$ replaced by $\widehat{\mathcal{G}}^h$, and with $K^h$ replaced
by $\widehat{K}^h := 
\{ \chi \in K^h : (\varrho(\Phi^{n-1}), \chi - \Phi^{n-1})^h = 0\}$. As before
we can interpret this variational inequality as the Euler--Lagrange equations
of a convex minimization problem, which yields the existence of a %its unique
solution $\Phi^n \in \widehat{K}^h$ that is unique
unless $\rho=0$, $\pi^h[\varrho(\Phi^{n-1})] = 0$ 
and $(\Phi^{n-1},1)=0$. 
Therefore, on noting (\ref{eq:wW}), we have existence of a solution 
$(\Phi^n,W^n) \in K^h\times S^h$ to
(\ref{eq:U},b). If $\Phi^n$ is unique, and if
$|\Phi^n(p_j)|<1$ and $\varrho(\Phi^{n-1}(p_j)) \not = 0$
for some $j\in J$ then (\ref{eq:ACh}) 
holds with equality for $\chi = \chi_j$, which uniquely determines
$\xi^n$ and hence yields the uniqueness of $W^n$. 
\end{proof}

It turns out that most of the technical assumptions in Lemma~\ref{lem:ex} are 
trivially satisfied for the shape function choices in (\ref{eq:varrho}). In
particular, we obtain the following result.

\begin{cor} \label{cor:ex}
Let $\gamma$ be of the form {\rm (\ref{eq:g})} with $r=1$,
let $\varrho$ be given by one of the choices in {\rm (\ref{eq:varrho})}
or by {\rm (\ref{eq:varrhoii})}, 
let $\varrho^+=0$ and let $\uD \in \R$.
Then there exists a solution $(\Phi^n,W^n) \in K^h \times \ShD$ to 
{\rm (\ref{eq:U},b)} and $\Phi^n$, $W^n$ are unique up to additive constants.
Moreover, $\Phi^n$ is unique unless 
$\varrho$ is of the form {\rm (\ref{eq:varrho})(iii)}, and $\rho=0$,
$(|\Phi^{n-1}|, 1)^h = |\Omega|$ and $|(\Phi^{n-1},1)| = 0$.
\end{cor}
\begin{proof}
The desired results follow immediately from Lemma~\ref{lem:ex}.
\end{proof}

\begin{rem} \label{rem:bl}
Let the assumptions of {\rm Lemma~\ref{lem:ex}} hold and let 
$\partial_N\Omega \not= \partial\Omega$. 
Then it is easy to prove that if $\Phi^{n-1} = 1$ and
$\vartheta\,(W^{n-1} - \uD) = 0$, and if
\begin{equation} \label{eq:bl}
-\frac{a}\alpha\, \varrho(1)\,\uD \leq \frac1\cPsi\,\epsilon^{-1}
\end{equation}
then the unique solution to {\rm (\ref{eq:U},b)} is given by 
$\Phi^n = 1$ and $W^n = \uD$. 
If the phase field parameter $\epsilon$ does not satisfy 
{\rm (\ref{eq:bl})}, then $\Phi^n = 1$ and $W^n = \uD$ is no longer the 
solution to {\rm (\ref{eq:U},b)}. %In fact, 
In practice
it is observed that if $\epsilon$ does not satisfy {\rm (\ref{eq:bl})}, then
the solution $\Phi^n$
exhibits a boundary layer close to $\partial\Omega$ where $\Phi^n < 1$. This
artificial boundary layer
is an undesired effect of the phase field approximation for the sharp 
interface problem {\rm (\ref{eq:1a}--e)}.
In fact, and not surprisingly, {\rm (\ref{eq:bl})} is precisely the condition
on $\varrho(1)$ in {\rm (\ref{eq:obstcond})}. 
This motivates the use of shape functions with $\varrho(1) = 0$, such as 
{\rm (\ref{eq:varrho})(ii)} and {\rm (\ref{eq:varrho})(iii)}, in practice. 
An obvious advantage over e.g.\ {\rm (\ref{eq:varrho})(i)} then is to be
able to use larger values of $\epsilon$, which in itself means that less fine
discretization parameters may be employed.

For completeness we note that if, and only if, the condition
\begin{equation} \label{eq:bl-}
\frac{a}\alpha\,\varrho(-1)\,\uD \leq \frac1\cPsi\,\epsilon^{-1}
\end{equation}
holds, then $\Phi^n = -1$, $W^n = \uD$ is the unique solution to 
{\rm (\ref{eq:U},b)} for $\Phi^{n-1} = -1$ and 
$\vartheta\,(W^{n-1} - \uD) = 0$. Satisfying both {\rm (\ref{eq:bl})} and
{\rm (\ref{eq:bl-})} is equivalent to satisfying {\rm (\ref{eq:obstcond})}. 
\end{rem}

\begin{rem} \label{rem:rgt1}
Let $\gamma$ be of the form {\rm (\ref{eq:g})} with $r>1$, 
and let the remaining assumptions of {\rm Lemma~\ref{lem:ex}} hold.
Then the highly nonlinear system
{\rm (\ref{eq:U},b)} for $(\Phi^{n},W^n)$ is no longer
continuously dependent on the variable $\Phi^{n}$, recall 
{\rm (\ref{eq:Br})}. 
Due to this fact it is not possible to show existence
of solutions to {\rm(\ref{eq:U},b)} with the help
of Brouwer's fixed point theorem. However, in practice we 
have no difficulties in finding solutions to the nonlinear system
{\rm(\ref{eq:U},b)}, and the employed iterative solvers always converge; 
see {\rm Section~\ref{sec:42}}.
We recall that the same situation occurred in \cite{ani3d}, see Remark~3.3
there, where discretizations for anisotropic geometric evolution equations for
anisotropic energies of the form {\rm (\ref{eq:g})} were considered for the
very first time.
\end{rem}

We now extend the existence result from Lemma~\ref{lem:ex} to the case of a
general splitting $\varrho = \varrho^+ + \varrho^-$. On recalling from 
(\ref{eq:obstcond}) and from
Remark~\ref{rem:bl} that nontrivial choices of $\varrho$, i.e.\ alternatives to
(\ref{eq:varrho})(i), are only of interest when 
$\partial_N\Omega\not=\partial\Omega$, we consider the case $\varrho^+\not=0$
only in the presence of Dirichlet boundary conditions on $W^n$.

\begin{thm} \label{thm:ex}
Let $\gamma$ be of the form {\rm (\ref{eq:g})} with $r=1$ and let $\uD \in \R$. In addition let $\rho+ |(\Phi^{n-1},1)| >0$ or
\begin{equation} \label{eq:hatrhoold}
( |\hatrho(\Phi^{n-1},\chi)|, 1 )^h > 0  \qquad \forall\ \chi \in K^h\,.
\end{equation}
Moreover we assume that either $\varrho^+=0$,
or $\vartheta>0$, or $\partial_N\Omega\not=\partial\Omega$.
Then there exists a solution $(\Phi^n,W^n) \in K^h \times \ShD$ to 
{\rm (\ref{eq:U},b)}.
\end{thm}
\begin{proof}
The desired result for the case $\varrho^+=0$ has been shown in
Lemma~\ref{lem:ex}. We now consider the case $\varrho^+\not=0$, so that
either $\vartheta > 0$ or $\partial_N\Omega \not= \partial\Omega$.
Then we can apply Brouwer's fixed point theorem to
prove existence of a solution $\Phi^n$ as follows.
Let the map ${\rm T} : K^h \to K^h$ be defined such that 
$\Phi^{\rm new} = {\rm T}(\Phi^{\rm old})$ is the solution of 
(\ref{eq:VI},b) with $\hatrho(\Phi^{n-1}, \Phi^n)$ replaced by
$\hatrho(\Phi^{n-1}, \Phi^{\rm old})$, and with all other occurrences of
$\Phi^n$ replaced by $\Phi^{\rm new}$. 
It follows from the proof of Lemma~\ref{lem:ex} and our assumptions
that there exists a unique $\Phi^{\rm new} \in K^h$, 
and the continuity of the map 
$\Phi^{\rm old} \mapsto \Phi^{\rm new} = {\rm T}(\Phi^{\rm old})$ 
together with the fact that $K^h$ is
compact and convex then yields the existence of a solution $\Phi^n \in K^h$
to {\rm (\ref{eq:U},b)}. The existence of a solution $W^n \in S^h_D$ 
then follows from (\ref{eq:Wn}).
\end{proof}

The following stability theorem is the main result of this paper. 

\begin{thm} \label{thm:stab2}
Let $\gamma$ be of the form {\rm (\ref{eq:g})} %Let $r\in [1,\infty)$ 
and let $\uD \in \R$.
Then it holds for a solution $(\Phi^n,W^n) \in K^h \times \ShD$ to 
{\rm (\ref{eq:U},b)} that
\begin{align} \label{eq:stab2}
& \mathcal{E}_\gamma^h(W^n, \Phi^n)  
- \uD\,\lambda\,(\hatrho(\Phi^{n-1}, \Phi^n), \Phi^n - \Phi^{n-1})^h
+ \tau_n\,( \pi^h[b(\Phi^{n-1})] \,\nabla\, W^{n} , \nabla\, W^n ) 
\nonumber \\ & \hspace{2cm}
+ \tau_n\,\frac{\lambda\,\rho}{a}\,\frac\epsilon\cPsi
\left| [\mu(\nabla\,\Phi^{n-1})]^{\frac12}\, 
\dfrac {\Phi^{n}-\Phi^{n-1}}{\tau_n}\right|_h^2 
\leq \mathcal{E}_\gamma^h(W^{n-1}, \Phi^{n-1})
\,.
\end{align}
In particular, if the splitting $\varrho = \varrho^+ + \varrho^-$ satisfies
\begin{equation} \label{eq:rhosplit-}
\pm\,\uD\,(\varrho^\pm)'(s) \leq 0 \qquad 
\forall\ s \in [-1,1]
\end{equation}
then it holds that
\begin{align} \label{eq:stab3}
& \mathcal{F}_\gamma^h(W^n, \Phi^n)  
+ \tau_n\,( \pi^h[b(\Phi^{n-1})] \,\nabla\, W^{n} , \nabla\, W^n ) 
+ \tau_n\,\frac{\lambda\,\rho}{a}\,\frac\epsilon\cPsi
\left| [\mu(\nabla\,\Phi^{n-1})]^{\frac12}\, 
\dfrac {\Phi^{n}-\Phi^{n-1}}{\tau_n}\right|_h^2 
\nonumber \\ & \hspace{9cm}
\leq \mathcal{F}_\gamma^h(W^{n-1}, \Phi^{n-1})
\,.
\end{align}
\end{thm}
\begin{proof}
Choosing $\chi = W^n - \uD$ in (\ref{eq:U}) and 
$\chi = \Phi^{n-1}$ in (\ref{eq:ACh}) yields that
\begin{subequations}
\begin{align}
& \vartheta\, (W^n - W^{n-1}, W^n - \uD)^h + \lambda\,
(\hatrho(\Phi^{n-1}, \Phi^n)\,[\Phi^{n}-\Phi^{n-1}], W^n - \uD )^h 
\nonumber \\ & \hspace{5cm}
+ \tau_n\,( \pi^h[b(\Phi^{n-1})]
 \,\nabla\, W^n , \nabla\, W^n )  = 0 \,, \label{eq:stab2u} \\
& 
\epsilon\,\frac\rho\alpha\,\tau_n^{-1}
\left(\mu(\nabla\,\Phi^{n-1})\,
\Phi^{n}-\Phi^{n-1}, \Phi^{n-1} - \Phi^n\right)^h +
\epsilon \, 
 (B_r(\nabla\, \Phi^{n-1}, \nabla\,\Phi^n)\, \nabla\, \Phi^{n}, 
 \nabla\, [\Phi^{n-1} -\Phi^n]) \nonumber \\ & \hspace{4cm}
 \geq 
 \left(\cPsi\,\frac{a}\alpha\,\hatrho(\Phi^{n-1}, \Phi^n)\,W^{n} 
 + \epsilon^{-1}\,\Phi^{n-1}, \Phi^{n-1} - \Phi^n \right)^h\,. 
\label{eq:stab2w}
\end{align}
\end{subequations}
It follows from (\ref{eq:stab2u},b), on recalling (\ref{eq:element})
and (\ref{eq:Bstab}), that
\begin{align*}
& \tfrac12\,\epsilon\,|\gamma(\nabla\, \Phi^n)|_0^2 - \tfrac12\,
\epsilon^{-1}\,|\Phi^n|_h^2 + 
\frac\vartheta2\,\frac{a}{\lambda\,a}\,\cPsi\,|W^n - \uD|_h^2 +
\tau_n\,\epsilon\,\frac\rho\alpha\left|
[\mu(\nabla\,\Phi^{n-1})]^{\frac12}\, 
\dfrac {\Phi^{n}-\Phi^{n-1}}{\tau_n}\right|_h^2 
\nonumber \\ & \hspace{2cm}
- \uD\,\frac{a}\alpha\,\cPsi\,
 (\hatrho(\Phi^{n-1},\Phi^n), \Phi^n - \Phi^{n-1})^h
+ \tau_n\,\frac{a}{\lambda\,\alpha}\,\cPsi\,
( \pi^h[b(\Phi^{n-1})] \,\nabla\, W^{n} , \nabla\, W^n ) 
\nonumber \\ & \hspace{5cm}
\leq \tfrac12\,\epsilon\,|\gamma(\nabla\, \Phi^{n-1})|_0^2 - \tfrac12\,
\epsilon^{-1}\,|\Phi^{n-1}|_h^2 +
\frac\vartheta2\,\frac{a}{\lambda\,a}\,\cPsi\,|W^{n-1} - \uD|_h^2 \,.
\end{align*}
This yields the desired result (\ref{eq:stab2}) on adding the constant
$\frac12\,\epsilon^{-1}\,\int_\Omega 1 \dx$ on both sides, and then multiplying
the inequality with $\frac{\lambda\,\alpha}{a}\,\frac1\cPsi$.
In addition, it follows from
$\Phi^{n-1}, \Phi^n \in K^h$ and (\ref{eq:rhosplit-}) %with $\omega=1$ 
that
\begin{align} \label{eq:rhohat}
\uD\,(\hatrho(\Phi^{n-1}, \Phi^n), \Phi^n - \Phi^{n-1})^h
& = \uD\,(\varrho^-(\Phi^{n-1}), \Phi^n - \Phi^{n-1})^h
 - \uD\,(\varrho^+(\Phi^{n}), \Phi^{n-1} - \Phi^{n})^h \nonumber \\ & 
\leq \uD\,(\Varrho^-(\Phi^{n}) - \Varrho^-(\Phi^{n-1}) 
+ \Varrho^+(\Phi^n) - \Varrho^+(\Phi^{n-1}), 1)^h \nonumber \\ & 
= \uD\,(\Varrho(\Phi^{n}) - \Varrho(\Phi^{n-1}), 1)^h \,.
\end{align}
The desired result (\ref{eq:stab3}) now follows on applying (\ref{eq:rhohat})
to (\ref{eq:stab2}).
\end{proof}

\subsection{Smooth potentials} \label{sec:32}
The unconditionally stable approximation (\ref{eq:U},b) for the obstacle
potential (\ref{eq:obstacle}) can be easily adapted to the case of a smooth
potential such as (\ref{eq:quartic}). To this end, let $\phi:=\Psi'$ for an
arbitrary smooth potential and let $\phi = \phi^+ + \phi^-$, with $\phi^\pm$
being the derivatives of the convex/concave parts of $\Psi$, i.e.\
\begin{subequations}
\begin{equation} \label{eq:phi}
\pm (\phi^\pm)'(s) \geq 0 \quad \forall\ s \in \R\,,\qquad
\Psi^\pm := \int_0^s \phi^\pm(y)\;{\rm d}y\,.
\end{equation}
We will make the
mild assumption that there exist constants $\psi_0, \psi_1, \delta > 0$ such 
that
\begin{equation} \label{eq:Psi+}
\Psi^+(s) \geq \psi_1\,|s|^{1+\delta} - \psi_0 \quad \forall\ s \in \R\,.
\end{equation}
\end{subequations}
For the quartic potential (\ref{eq:quartic}) the natural choices are
\begin{equation} \label{eq:phi+-}
\phi^+(s) = s^3 \quad \text{and} \quad \phi^-(s) = -s\,, 
\end{equation}
so that (\ref{eq:phi},b) are clearly satisfied.

As before, given $\Phi^0 \in K^h$ and, if $\vartheta>0$, $W^0 \in \ShD$, 
for $n \geq 1$, find $(\Phi^n,W^n) \in S^h \times \ShD$ such that
\begin{subequations}
\begin{align}
& \vartheta \left(\dfrac{W^{n}-W^{n-1}}{\tau_n},
 \chi \right)^h + 
\lambda\,\left(\hatrho_m(\Phi^{n-1},\Phi^n)
\,\dfrac{\Phi^{n}-\Phi^{n-1}}{\tau_n},
 \chi \right)^h 
+ ( \pi^h[\widetilde{b}(\Phi^{n-1})]\,\nabla\, W^{n} , \nabla\, \chi )
 = 0 \nonumber \\ & \hspace{10cm}
\qquad \forall\ \chi \in S^h_0, \label{eq:sU} \\
& \epsilon\,\frac\rho\alpha\left(\mu(\nabla\,\Phi^{n-1})\,
\dfrac{\Phi^{n}-\Phi^{n-1}}{\tau_n}, \chi\right)^h +
\epsilon\,(B_r(\nabla\, \Phi^{n-1}, \nabla\, \Phi^n)
\, \nabla\, \Phi^{n}, \nabla\, \chi) 
+ \epsilon^{-1}\,(\phi^+(\Phi^n),\chi)^h \nonumber \\ & \qquad\qquad
 = \left(\cPsi\,\frac{a}\alpha\,\hatrho_m(\Phi^{n-1},\Phi^n)
 \,W^{n}  - \epsilon^{-1}\,\phi^-(\Phi^{n-1}),\chi \right)^h
 \qquad \forall\ \chi \in S^h\,, \label{eq:sACh}
\end{align}
\end{subequations}
where in order to avoid degeneracies we have defined
\begin{equation*} % \label{eq:btilde}
\widetilde{b}(s) = \begin{cases}
b(1) & s \geq 1\,,\\
b(s) & |s| \leq 1\,,\\
b(-1) & s \leq -1\,,
\end{cases}
\end{equation*}
and where for technical reasons we have introduced
\begin{equation} \label{eq:rhom}
\hatrho_m(s_0, s_1) := \varrho^-(s_0) + \varrho^+_m(s_1)\,,
\quad\text{where}\quad
\varrho^+_m(s) := \begin{cases}
\varrho^+(m) & s \geq m\,, \\
\varrho^+(s) & |s| \leq m\,, \\
\varrho^+(-m) & s \leq -m\,,
\end{cases}
\end{equation}
for some fixed parameter $m \geq 2$. We note that these modifications of
(\ref{eq:b}) and (\ref{eq:hatrho}) are such that 
\begin{subequations}
\begin{alignat}{2}
\widetilde{b}(s) & \geq \min \{\conduct_+,\conduct_-\} > 0 \qquad && \forall\ 
s\in\R\,, \label{eq:tb} \\ \text{and}\qquad
\max_{s \in \R} |\hatrho_m(s_0,s)| & = \max_{|s|\leq m} |\hatrho(s_0,s)|
= C(m, s_0) \qquad && \forall\ s_0 \in \R\,. \label{eq:hrm}
\end{alignat}
\end{subequations}

\begin{thm} \label{thm:exs}
Let $\gamma$ be of the form {\rm (\ref{eq:g})} with $r=1$ and let $\uD \in \R$. 
If $\varrho^+=0$ and if $\phi^+$ is strictly monotonically increasing,
then there exists a unique solution
$(\Phi^n,W^n) \in S^h \times \ShD$ to {\rm (\ref{eq:sU},b)}.
If $\varrho^+\not=0$, and if either $\vartheta>0$ or 
$\partial_N\Omega\not=\partial\Omega$, then 
there exists a solution $(\Phi^n,W^n) \in S^h \times \ShD$ to 
{\rm (\ref{eq:sU},b)} if $\Psi^+$ satisfies the assumption 
{\rm (\ref{eq:Psi+})}. % holds.
\end{thm}
\begin{proof}
The existence and uniqueness proof for the case $\varrho^+=0$, 
which is a simple modification of the proof of Lemma~\ref{lem:ex},
is left to the reader. Note that this proof 
makes use of the strict monotonicity of $\phi^+$.

In order to proof existence for the case $\varrho^+\not=0$, 
we apply Brouwer's fixed point theorem. It is this part of the proof
that requires the cut-off of $\hatrho$ defined in (\ref{eq:rhom}), as well as
the mild assumption (\ref{eq:Psi+}). 
The application of Brouwer's fixed point theorem is
similar to the proof of Theorem~\ref{thm:ex}. Setting up the map
$\Phi^{\rm old} \mapsto \Phi^{\rm new} = {\rm T}(\Phi^{\rm old})$ 
analogously to the 
proof there, we immediately see that the map ${\rm T}$ is
well-defined and continuous, where we recall that our assumptions yield that
$\vartheta > 0$ or $\partial_N\Omega\not=\partial\Omega$.
It remains to show that ${\rm T}: Y^h \to Y^h$ for
a bounded subset $Y^h\subset S^h$. To this end, on recalling (\ref{eq:VI},b), 
we note that $\Phi^{\rm new}\in S^h$ satisfies
\begin{align*}
& \frac{\lambda\,\cPsi\,a}{\alpha\,\tau_n} \,\Bigl\{
( \pi^h[\widetilde{b}(\Phi^{n-1})]\,\nabla\, [\mathcal{\widetilde{G}}^h\,
\pi^h[\hatrho_m(\Phi^{n-1},\Phi^{\rm old})\,(\Phi^{\rm new}-\Phi^{n-1})]], \nabla\,[
\mathcal{\widetilde{G}}^h\,\pi^h[\hatrho_m(\Phi^{n-1},\Phi^{\rm old})\,\chi]] ) 
\nonumber \\ & \qquad 
+  \frac\vartheta{\tau_n} \,( 
\mathcal{\widetilde{G}}^h\,\pi^h[\hatrho_m(\Phi^{n-1}, \Phi^{\rm old})\,(\Phi^{\rm new}-\Phi^{n-1})],
\mathcal{\widetilde{G}}^h\,\pi^h[\hatrho_m(\Phi^{n-1}, \Phi^{\rm old})\,\chi] )^h \Bigr\}
\nonumber \\ & \qquad
+ \frac{\epsilon\,\rho}{\alpha\,\tau_n}\,
(\mu(\nabla\,\Phi^{n-1})\,\Phi^{\rm new}, \chi )^h
+ \epsilon\,(B_1(\nabla\, \Phi^{n-1})\, \nabla\, \Phi^{\rm new}, \nabla\, \chi)
+ \epsilon^{-1}\,(\phi^+(\Phi^{\rm new}), \chi)^h
\nonumber \\ & \quad
= (g^h, \chi)^h \qquad \forall\ \chi \in S^h\,, %\label{eq:sVI}
\end{align*}
where
$g^h := \frac{\epsilon\,\rho}{\alpha\,\tau_n}\,\mu(\nabla\,\Phi^{n-1})\,
\Phi^{n-1} - \epsilon^{-1}\,\phi^-(\Phi^{n-1}) + 
\cPsi\,\frac{a}\alpha\,\hatrho_m(\Phi^{n-1},\Phi^{\rm old})\,
(\uD +  \frac\vartheta{\tau_n} \, \mathcal{\widetilde{G}}^h\,[W^{n-1}-\uD])$,
and where $\mathcal{\widetilde{G}}^h$ is defined by (\ref{eq:Gh}) with
$b$ replaced by $\widetilde{b}$.
These are the
Euler--Lagrange equations for the convex minimization problem
\begin{subequations}
\begin{equation} \label{eq:smin}
\min_{\chi \in S^h} \left[ J(\chi) - (g^h, \chi)^h \right],
\end{equation}
where
\begin{align}
J(\chi) := & 
\frac{\lambda\,\cPsi\,a}{2\,\alpha\,\tau_n} \,\Bigl\{
( \pi^h[\widetilde{b}(\Phi^{n-1})], |\nabla\, [\mathcal{\widetilde{G}}^h\,
\pi^h[\hatrho_m(\Phi^{n-1},\Phi^{\rm old})\,(\chi-\Phi^{n-1})]]|^2 )
 \nonumber \\ & \quad 
+ \frac\vartheta{2\,\tau_n} \, 
|\mathcal{\widetilde{G}}^h\,\pi^h[\hatrho_m(\Phi^{n-1},\Phi^{\rm old})\,
 (\chi-\Phi^{n-1})] |^2_h \Bigr\}
+ \frac{\epsilon\,\rho}{2\,\alpha\,\tau_n}\,(\mu(\nabla\,\Phi^{n-1}), |\chi|^2
)^h
 \nonumber \\ & \quad 
+ \frac\epsilon2\,(B_1(\nabla\,\Phi^{n-1})\,\nabla\,\chi, \nabla\,\chi) 
+ \epsilon^{-1}\,(\Psi^+(\chi),1)^h
\qquad \forall\ \chi \in S^h\,.
\label{eq:J}
\end{align}
\end{subequations}
It follows from (\ref{eq:smin},b) and (\ref{eq:tb},b) that
\begin{equation} \label{eq:Jphi}
\epsilon^{-1}\,(\Psi^+(\Phi^{\rm new}),1)^h - (g^h, \Phi^{\rm new})^h
\leq J(\Phi^{\rm new}) - (g^h, \Phi^{\rm new})^h
\leq J(0) \leq C(\Phi^{n-1})\,.
\end{equation}
Applying the elementary inequality $y\,z \leq \frac1p\,|y|^p +
\frac1q\,|z|^q$, for $p,q\in(1,\infty)$ with $\frac1p+\frac1q=1$, 
to the second term in (\ref{eq:Jphi}) yields that
\begin{equation} \label{eq:Jphi2}
(g^h, \Phi^{\rm new})^h \leq
\tfrac12\,\epsilon^{-1}\,\psi_1\,(|\Phi^{\rm new}|^{1+\delta},1)^h + 
C(\epsilon, \delta, g^h)\,,
\end{equation}
where $\psi_1$ and $\delta$ are as in (\ref{eq:Psi+}).
Now combining (\ref{eq:Jphi}) and (\ref{eq:Jphi2}), on recalling the mild 
assumption 
(\ref{eq:Psi+}), yields that $(|\Phi^{\rm new}|^{1+\delta},1)^h \leq C$ for 
some constant $C>0$ independent of $\Phi^{\rm old}$, i.e.\
\begin{equation} \label{eq:Tbound}
(|{\rm T}(\chi)|^{1+\delta},1)^h \leq C \quad\forall\ \chi \in S^h\,.
\end{equation}
Hence ${\rm T}: Y^h \to Y^h$ for
a bounded subset $Y^h\subset S^h$, and so Brouwer's fixed point theorem 
yields the existence of a solution $\Phi^n$ to {\rm (\ref{eq:sU},b)}.
The existence of a solution $W^n \in S^h_D$ 
then follows from (\ref{eq:Wn}) with
$\mathcal{G}^h$ replaced by $\mathcal{\widetilde{G}}^h$ and with
$\hatrho$ replaced by $\hatrho_m$.
\end{proof}

We stress that the cut-off introduced in (\ref{eq:rhom}) is for technical
reasons only. If a solution $(\Phi^n,W^n) \in S^h \times \ShD$ to 
{\rm (\ref{eq:sU},b)} is such that $|\Phi^n| \leq m$, then clearly
$(\Phi^n,W^n) \in S^h \times \ShD$ also solves {\rm (\ref{eq:sU},b)}
with $\hatrho_m$ replaced by $\hatrho$. In practice, this is always the case
for $m$ chosen sufficiently large. Hence for practical implementations, only
{\rm (\ref{eq:sU},b)} with $\hatrho_m$ replaced by $\hatrho$ needs to be
considered.

\begin{cor} \label{cor:exs}
Let $\gamma$ be of the form {\rm (\ref{eq:g})} with $r=1$ and let $\uD \in \R$.
Let $\Psi$ be given by {\rm (\ref{eq:quartic})} and let {\rm (\ref{eq:phi+-})} 
hold.
Let $\varrho$ and its splitting be defined by one of the choices in 
{\rm (\ref{eq:split})} or {\rm (\ref{eq:split+})}.
Then there exists a solution
$(\Phi^n,W^n) \in S^h \times \ShD$ to {\rm (\ref{eq:sU},b)}
unless in cases {\rm (\ref{eq:split})(iii)} and {\rm (\ref{eq:split+})(iii)} 
it holds that
$\vartheta=0$ and $\partial_N\Omega=\partial\Omega$.
Moreover, $(\Phi^n,W^n)$ is unique for the choices {\rm (\ref{eq:split})(i)},
{\rm (\ref{eq:split})(ii)} and {\rm (\ref{eq:split+})(ii)}.
\end{cor}
\begin{proof}
The desired results follow immediately from Theorem~\ref{thm:exs} on noting
that (\ref{eq:phi+-}) satisfies the assumptions on $\phi^+$ and $\Psi^+$ stated
there.
\end{proof}

\begin{rem} \label{rem:sbl}
Similarly to {\rm Remark~\ref{rem:bl}}, 
the following observation holds for the scheme {\rm (\ref{eq:sU},b)}
when $r=1$, $\varrho^+=0$, $\uD \in \R$ and $\phi^+$ is strictly monotonically
increasing. 
Then, if $\Phi^{n-1} = 1$ and $\vartheta\,(W^{n-1} - \uD) = 0$, then the
unique solution to {\rm (\ref{eq:sU},b)} is given by $\Phi^n = 1$ and 
$W^n = \uD$ if and only if
\begin{equation} \label{eq:sbl}
\uD\,\varrho(1) = 0 \,.
\end{equation}
For nonzero $\uD$ this is precisely the necessary condition 
for $G(s)$ in {\rm (\ref{eq:G})} to have a local minimum at $s=1$.
In practice, if the condition {\rm (\ref{eq:sbl})} is violated, then for 
certain values of $\uD$ and $\epsilon$ artificial boundary layers develop.
This undesired effect for the choice {\rm (\ref{eq:varrho})(i)} 
once again motivates the
use of the alternatives {\rm (\ref{eq:varrho})(ii)} and 
{\rm (\ref{eq:varrho})(iii)} in practice.
\end{rem}

The following stability result is the natural analogue of
Theorem~\ref{thm:stab2} for the case of a smooth potential $\Psi$.

\begin{thm} \label{thm:stabs}
Let $\gamma$ be of the form {\rm (\ref{eq:g})} %Let $r\in [1,\infty)$ 
and let $\uD \in \R$. 
Then it holds that a solution $(\Phi^n,W^n) \in S^h \times \ShD$ to 
{\rm (\ref{eq:sU},b)} satisfies {\rm (\ref{eq:stab2})} with $b$ replaced by
$\widetilde{b}$, and with $\hatrho$ replaced by $\hatrho_m$.
In particular, if the splitting $\varrho = \varrho^+ + \varrho^-$ satisfies
{\rm (\ref{eq:rhosplit})}, % with $1 \leq \omega \leq m$, 
and if
\begin{subequations}
\begin{equation} \label{eq:Phiomega}
\Phi^{n-1} \leq \tfrac2{\sqrt{3}} \quad\text{and}\quad 
\Phi^n \leq \tfrac2{\sqrt{3}}\,,
\end{equation}
or if it %the splitting $\varrho = \varrho^+ + \varrho^-$ 
satisfies {\rm (\ref{eq:rhosplit+})}, and if
\begin{equation} \label{eq:Phiomega+}
\Phi^{n-1} \geq -\tfrac2{\sqrt{3}} \quad\text{and}\quad 
\Phi^n \geq -\tfrac2{\sqrt{3}}\,,
\end{equation}
\end{subequations}
then the solution $(\Phi^n,W^n)$ %, which now also solves {\rm (\ref{eq:sU},b)}
satisfies the stability bound {\rm (\ref{eq:stab3})} with $b$ replaced by 
$\widetilde{b}$.
\end{thm}
\begin{proof}
The proof of the stability bounds, 
which is a simple modification of the proof of Theorem~\ref{thm:stab2}, 
is left to the reader. 
Note that the proof makes use of the splittings $\phi = \phi^+ + \phi^-$ and
$\varrho= \varrho^+ + \varrho^-$, recall (\ref{eq:rhohat}). 
\end{proof}

\begin{cor} \label{cor:stabs}
Let $\gamma$ be of the form {\rm (\ref{eq:g})} %Let $r\in [1,\infty)$ 
and let $\uD \in \R$. 
Then for the choices of $\varrho$ and its splittings in
{\rm (\ref{eq:split})(i)}, {\rm (\ref{eq:split})(ii)} and
{\rm (\ref{eq:split+})(ii)} 
it holds that the unique solution $(\Phi^n,W^n) \in S^h \times \ShD$ to 
{\rm (\ref{eq:sU},b)} satisfies 
the stability bound {\rm (\ref{eq:stab3})} with $b$ replaced by 
$\widetilde{b}$.
For the choice {\rm (\ref{eq:varrho})(iii)}, with 
the splittings {\rm (\ref{eq:split})(iii)} or {\rm (\ref{eq:split+})(iii)},  
it holds that a solution $(\Phi^n,W^n) \in S^h \times \ShD$ to 
{\rm (\ref{eq:sU},b)} satisfies the same stability bound if
{\rm (\ref{eq:Phiomega})} or {\rm (\ref{eq:Phiomega+})}
hold, respectively.
\end{cor}
\begin{proof}
The desired results for the splittings 
(\ref{eq:split})(i), (\ref{eq:split})(ii) and (\ref{eq:split+})(ii),
on recalling Corollary~\ref{cor:exs}, follow
from the fact that these splittings satisfy the
inequalities in (\ref{eq:rhosplit}) for all $s \in \R$. The results for
(\ref{eq:varrho})(iii) follow immediately from Theorem~\ref{thm:stabs}.
\end{proof}

In practice, in general, the values of $\Phi^n$ are either within the
interval $[-1,1]$, or very close to it. In our experience, 
for the scheme (\ref{eq:sU},b) with (\ref{eq:varrho})(iii) and with
the splittings (\ref{eq:split})(iii) and (\ref{eq:split+})(iii) 
for $\uD \leq 0$ and $\uD > 0$, respectively, in practice 
(\ref{eq:Phiomega}) and (\ref{eq:Phiomega+}) always hold.
Here we note that $\tfrac2{\sqrt{3}} \approx 1.15$.

\setcounter{equation}{0} 
\section{Solution of the algebraic systems of equations} \label{sec:4}
The system of nonlinear equations for $(\Phi^n,W^n)$ arising at each time level 
from the approximation (\ref{eq:sU},b) can be solved with a Newton method or
with a nonlinear multigrid method, see e.g.\ \cite{KimKL04}.

For the remainder of this section we %We now 
discuss the solution of the systems of algebraic
equations for $(\Phi^n,W^n)$ arising at each time level 
from the approximation (\ref{eq:U},b).
Adopting the obvious notation, the system (\ref{eq:U},b) can be rewritten as:
Find $(\Phi^n,W^n) \in [-1,1]^{\mathcal{J}} \times \R^{\mathcal{J}}$, ${\mathcal{J}} := \# J$, such that
\begin{subequations}
\begin{align}
& \lambda\,M_\varrho(\Phi^n)\,\Phi^n + (\vartheta\,M+\tau_n\,A)\,W^n  = 
\tilde{f}(\Phi^n)
\label{eq:GS2a} \\
& \epsilon\,\frac\rho\alpha\,\tau_n^{-1}\,(V - \Phi^n)^T\,M_\mu\,
 \Phi^n + \epsilon\,(V - \Phi^n)^T\,\mathcal{B}_r(\Phi^n)\,
 \Phi^n - \cPsi\,\frac{a}\alpha\,(V - \Phi^n)^T\,M_\varrho(\Phi^n)\,
 W^n \nonumber \\ & \hspace{8cm} \geq 
 (V - \Phi^n)^T\,\tilde{g} \qquad
 \forall\ V \in [-1,1]^{\mathcal{J}}\,, \label{eq:GS2b}
\end{align}
\end{subequations}
where $M$, $M_\mu$, $M_\varrho(\eta)$, $A$ and $\mathcal{B}_r(\eta)$, 
for $\eta \in S^h$, 
are symmetric ${\mathcal{J}} \times {\mathcal{J}}$ matrices.
In the case of pure Neumann boundary conditions, (\ref{eq:assOmega})(ii), 
their entries are given by
$M_{ij} := (\chi_{i},\chi_{j})^h$,
$[M_\mu]_{ij} := (\mu(\nabla\,\Phi^{n-1})\,\chi_{i},\chi_{j})^h$,
$[M_\varrho(\eta)]_{ij} := (\hatrho(\Phi^{n-1},\eta)\,\chi_{i},\chi_{j})^h$,
\begin{equation*} 
[\mathcal{B}_r(\eta)]_{ij} 
:= (B_r(\nabla\,\Phi^{n-1}, \nabla\,\eta)\,\nabla\, \chi_i,
\nabla\,\chi_j),
\quad A_{ij} := (\pi^h[b(\Phi^{n-1})]\,\nabla\, \chi_i,\nabla\, \chi_j) \,,
\end{equation*}
while the right hand sides in this case are defined as
$\tilde{f}(\Phi^n) := \lambda\,M_\varrho(\Phi^n)\,\Phi^{n-1}
+ \vartheta\,M\,W^{n-1}$ and
$\tilde{g} := \epsilon\,\frac\rho\alpha\,\tau_n^{-1}\,M_\mu\,\Phi^{n-1} +
\epsilon^{-1}\,M\,\Phi^{n-1} \in \R^{\mathcal{J}}$. Of course, for the cases
(\ref{eq:assOmega})(i) and (\ref{eq:assOmega})(iii) these entries need to be
appropriately manipulated.

Clearly, the algebraic system (\ref{eq:GS2a},b) can
be written as a (symmetric) nonsmooth saddle point problem of the form:
Find $(U,W) \in [-1,1]^{\mathcal{J}} \times \R^{\mathcal{J}}$, 
\begin{subequations}
\begin{align}
 \mathcal{M}_\varrho(U)\,U + \mathcal{A}\,W & = f_\varrho(U) \label{eq:SPPa} \\
 (V - U)^T\,\mathcal{C}_r(U)\,U - (V - U)^T\,\mathcal{M}_\varrho(U)\, W & \geq 
 (V - U)^T\,g \qquad  \forall\ V \in [-1,1]^{\mathcal{J}}\,, \label{eq:SPPb}
\end{align}
\end{subequations}
where we prefer to write the unknowns as $(U,W)$ in place of $(\Phi^n,W^n)$, in
order to highlight the connection to discretizations of Cahn--Hilliard
equations, where the former notation is standard.
On recalling (\ref{eq:r1}) and (\ref{eq:hatrho}), 
we note that (\ref{eq:SPPa},b) in the case $r=1$ and $\varrho^+ = 0$ 
collapses to 
\begin{subequations}
\begin{align}
 \mathcal{M}\,U + \mathcal{A}\,W & = f \label{eq:SPP1a} \\
 (V - U)^T\,\mathcal{C}\,U - (V - U)^T\,\mathcal{M}\, W & \geq 
 (V - U)^T\,g \qquad \forall\ V \in [-1,1]^{\mathcal{J}}\,, \label{eq:SPP1b}
\end{align}
\end{subequations}
where $\mathcal{C} := \mathcal{C}_1(\vec 0)$,
$\mathcal{M}:=\mathcal{M}_\varrho(0)$ and $f:=f_\varrho(0)$. 
Nonsmooth
saddle point problems of the form (\ref{eq:SPP1a},b) are well-known from the
numerical approximation of (isotropic) Cahn--Hilliard equations.
Various different solution methods for the system (\ref{eq:SPP1a},b) are
discussed in
\cite{voids,GraserK07,voids3d,mgch,BlankBG11,HintermullerHT11,%
GraserKS12preprint}.
In the case $r=1$ we use the solution method from \cite{voids3d} in order to
solve (\ref{eq:SPP1a},b). In the remainder of this section we consider the case
$r \geq 1$. 
We now state possible solution methods for the nonlinear nonsmooth saddle point 
problem (\ref{eq:SPPa},b).

\subsection{Nonlinear Uzawa-multigrid iteration}
In what follows, we will extend the Uzawa-multigrid iteration from 
\cite{voids3d}, which is based on the ideas in \cite{GraserK07}, to the highly
nonlinear saddle point problem (\ref{eq:SPPa},b). The method 
from \cite{voids3d} can be interpreted
as a primal active set method, where the approximation of the active set is
driven by the current iterate $W_k$ in (\ref{eq:SPP1b}), rather than via a dual
parameter as in e.g.\ \cite{BlankBG11}. 

Given an initial iterate 
$(U_0,W_0) \in [-1,1]^{\mathcal{J}} \times \R^{\mathcal{J}}$, for $k\geq0$
let $U_{k+\frac12} \in [-1,1]^{\mathcal{J}}$ be the solution of
\begin{subequations}
\begin{equation}
 (V - U_{k+\frac12})^T\,\mathcal{C}_r(U_{k})\,U_{k+\frac12} 
 \geq (V - U_{k+\frac12})^T\,(g + \mathcal{M}_\varrho(U_k)\, W_{k})
\qquad  \forall\ V \in [-1,1]^{\mathcal{J}}\,. \label{eq:Uk}
\end{equation}
Then we define the active sets as
\begin{equation} \label{eq:Jpm}
J_k^{\pm} := \{ j \in J : [U_{k+\frac12}]_j = \pm 1 \} 
\quad\text{and let}\quad J_k = J_k^+ \cup J_k^-\,.
\end{equation}
Now we seek the solution $(U_{k+1}, W_{k+1}) \in \R^{\mathcal{J}} \times \R^{\mathcal{J}}$ to the
linear system
\begin{equation} \label{eq:blockA}
\begin{pmatrix}
\widehat{\mathcal{C}}_r(J_k, U_k) & -\widehat{\mathcal{M}}_\varrho(J_k,U_k) \\
\mathcal{M}_\varrho(U_k) & \mathcal{A} 
\end{pmatrix}
\begin{pmatrix} U_{k+1} \\ W_{k+1} \end{pmatrix}
=
\begin{pmatrix} \widehat g(J_k^+,J_k^-) \\ f_\varrho(U_k) \end{pmatrix},
\end{equation}
\end{subequations}
where, for $j\in J$,
\begin{equation*} %\label{BMa}
[\widehat{\mathcal{C}_r}(J_k, U_k)]_{ij} = \begin{cases}
\delta_{ij} & i \in J_k\,, \\
[{\mathcal{C}_r}(U_k)]_{ij} & i \in J\setminus J_k\,,
\end{cases} \quad
[\widehat{\mathcal{M}}_\varrho(J_k,U_k)]_{ij} =
\begin{cases}
0 & i \in J_k\,, \\
[\mathcal{M}_\varrho(U_k)]_{ij} & i \in J\setminus J_k\,,
\end{cases}% \qquad j\in J\,, 
\end{equation*}
and
\begin{equation*} %\label{BMb}
[\widehat{g}(J_k^+,J_k^-)]_i=
\begin{cases}
\pm 1 & i \in J_k^\pm\,, \\
g_{i} & i \in J\setminus J_k\,.
\end{cases}
\end{equation*}
Now we continue the iteration (\ref{eq:Uk}--c), until convergence is obtained, 
i.e.\ until 
\begin{equation} \label{eq:tol}
J_{k+1}^\pm = J_k^\pm \quad\text{and}\quad 
\max \left\{ \max_{j \in J} \left|[U_{k+1}]_j - [U_{k}]_j\right|, 
\max_{j \in J} \left|[W_{k+1}]_j - [W_{k}]_j\right|\right\} < tol\,,
\end{equation}
where $tol=10^{-8}$ is a given fixed tolerance. If a good initial guess $W_0$ 
is not available, then for $k=0$ it can be beneficial to set 
$U_\frac12=U_0$, rather than to employ (\ref{eq:Uk}).
Observe that since the iterates
$U_{k+\frac12}$ are only needed to define the active sets $J_{k}^\pm$ in 
(\ref{eq:Jpm}), an iterative procedure to find the solution of (\ref{eq:Uk})
can be stopped as soon as the active sets $J_{k}^\pm$ have been found. In
practice we stop the iteration as soon as two successive iterates for
(\ref{eq:Uk}) have the same active sets, which is usually the case
after a few projected block Gauss--Seidel
iterations. Alternatively, a monotone multigrid method
could be employed to solve (\ref{eq:Uk}), see \cite{Kornhuber94a}.
The linear saddle point problems (\ref{eq:blockA}) can be solved with a
multigrid method using block Gauss--Seidel smoothers or, alternatively, 
with a direct solution method such as UMFPACK (\cite{Davis04}) or
LDL (\cite{Davis05}), together with the 
sparse matrix ordering package AMD (\cite{AmestoyDD04}). Here for the multigrid
solver and the LDL factorization package, the linear system (\ref{eq:blockA}) 
needs to be equivalently reformulated with a symmetric block matrix, which is
easily possible. Finally, we observe that in the case $r=1$ and $\varrho^+=0$, 
the first stopping
criterion in (\ref{eq:tol}) immediately implies the second criterion in 
(\ref{eq:tol}), as then the linear system (\ref{eq:blockA}) 
does not depend on the iterates $U_k$.

\begin{rem} \label{rem:uzawa}
In practice, in our computations, the iteration {\rm (\ref{eq:Uk}--c)} did not
converge for values of $r > 3$, while it usually converged for smaller values
of $r$. In particular, it always converged in the case $r=1$ for the nonlinear
approximation {\rm (\ref{eq:U},b)} with the splitting 
{\rm (\ref{eq:split})(iii)}.
However, as we are interested in performing simulations for much larger values 
of $r$, e.g.\ $r=9$ for {\sc ani$_9$}, below, 
we also consider a more robust solution method in the next subsection.
\end{rem}

\subsection{Lagged fixed point iteration} \label{sec:42}
In this subsection we consider a lagged fixed point iteration, where at each
iteration a subproblem of the form (\ref{eq:SPP1a},b) needs to be solved.

Let $k=0$. Given an initial iterate 
$(U_0,W_0) \in [-1,1]^{\mathcal{J}} \times \R^{\mathcal{J}}$, we let 
$(U_{k+\frac12}, W_{k+\frac12}) \in
 [-1,1]^{\mathcal{J}} \in \times \R^{\mathcal{J}}$ be the solution of
\begin{subequations}
\begin{align}
 \mathcal{M}_\varrho(U_k)\,U_{k+\frac12} + \mathcal{A}\,W_{k+\frac12} & = 
f_\varrho(U_k)
 \label{eq:SPPk1a} \\
 (V - U_{k+\frac12})^T\,\mathcal{C}_r(U_{k})\,U_{k+\frac12} 
  - (V - U_{k+\frac12})^T\,\mathcal{M}_\varrho(U_k)\, W_{k+\frac12} & \geq 
 (V - U_{k+\frac12})^T\,g \nonumber \\ & 
 \qquad\qquad \forall\ V \in [-1,1]^{\mathcal{J}}\,. 
 \label{eq:SPPk1b}
\end{align}
On obtaining $(U_{k+\frac12}, W_{k+\frac12})$, we set
\begin{equation}
(U_{k+1}, W_{k+1}) = (1-\mu)\,(U_{k}, W_{k}) + 
\mu\,(U_{k+\frac12}, W_{k+\frac12})\,,
\label{eq:itnew}
\end{equation} 
\end{subequations}
where $\mu\in(0,1]$ is a fixed relaxation parameter.
The iteration (\ref{eq:SPPk1a}--c) is repeated until 
\begin{equation*} % \label{eq:tol2}
\max \left\{ \max_{j \in J} \left|[U_{k+1}]_j - [U_{k}]_j\right|, 
\max_{j \in J} \left|[W_{k+1}]_j - [W_{k}]_j\right|\right\} < tol\,.
\end{equation*}
In practice, the iteration  (\ref{eq:SPPk1a}--c)
always converged, provided $\mu$ was chosen sufficiently small.

\setcounter{equation}{0} 
\section{Numerical experiments} \label{sec:5}
In this section we report on numerical experiments for the proposed finite
element approximations. %Unless otherwise stated, 
Apart from a single computation for the approximation (\ref{eq:sU},b) in the
case of the quartic potential (\ref{eq:quartic}), where we employ the splitting
(\ref{eq:phi+-}),
we will present results for
the approximation (\ref{eq:U},b) for the obstacle potential (\ref{eq:obstacle}) 
only. Our preference for the scheme (\ref{eq:U},b) over the alternative
approximation (\ref{eq:sU},b) stems from the fact that in the former the phase
field approximation $\Phi^n$ is guaranteed to stay inside the interval 
$[-1,1]$, while the latter scheme in general admits values $|\Phi^n| >1$,
which in practice is observed if e.g.\ a
well developed interface is present. Moreover, the bulk regions for the
approximation (\ref{eq:U},b) are easily identified through $\Phi^n = \pm1$,
whereas for the scheme (\ref{eq:sU},b) this is less straightforward.
For the implementation of the approximations we have
used the adaptive finite element toolbox ALBERTA, see \cite{Alberta}. 
For the approximation (\ref{eq:U},b) we employ
the adaptive mesh strategy introduced in \cite{voids} and
\cite{voids3d}, respectively, for $d=2$ and $d=3$.
This results in a fine mesh of uniform mesh size 
$h_f$ inside the interfacial region $|\Phi^{n-1}|<1$ 
and a coarse mesh of uniform 
mesh size $h_c$ further away from it. 
Here $h_{f} = \frac{2\,H}{N_{f}}$ and $h_{c} =  \frac{2\,H}{N_{c}}$
are given by two integer numbers $N_f >  N_c$, where we assume from now on that
$\Omega = (-H,H)^d$. In all of the experiments below we have $H=\frac12$ with
(\ref{eq:assOmega})(i), unless otherwise stated.

Throughout this section the initial data $\varphi_0\in C(\overline\Omega)$ 
is either chosen constant, $\varphi_0 = 1$, or is chosen
with a well developed interface of width
$\epsilon\,\pi$, in which $\varphi_0$ varies smoothly and such that 
$\Gamma_0 = \{ x \in \Omega : \varphi_0(x) = 0 \}$.
Details of such initial data can be found in e.g.\ 
\cite{voids,wccmproc,voids3d}. 
In general the initial interface $\Gamma_0$ is a circle/sphere of radius 
$R_0 \in (0,H)$ around the origin. We use $R_0=0.1$ unless otherwise stated.
If $\vartheta > 0$, we set
\begin{equation*}
u_0(\vec{z}) = \begin{cases} 0 & |\vec{z}| \leq R_0\,, \\
\dfrac{u_D}{1 - e^{R_0 - H}}\left(1 - e^{R_0 - |\vec{z}|}\right) 
& R_0 < |\vec{z}| < H\,, \\
u_D & |\vec{z}| \geq H \,.
\end{cases}
\end{equation*}
We always fix 
$\Phi^0 = \pi^h\varphi_0$ and, if $\vartheta>0$, $W^0=\pi^h u_0$.

Unless otherwise stated we always let 
$\epsilon^{-1}=16\,\pi$ and $N_f = 128$, $N_c=16$. In addition, we employ
uniform time steps $\tau_n = \tau$, $n = 1 \to N$. As an indication for the
computational effort that is involved in producing the simulations presented in
this section, we state for each simulation an exemplary CPU time for a 
single-thread run on an Intel i7-860 (2.8 GHz) processor.

For the anisotropies in our numerical results we always choose among
\begin{alignat*}{2}
& \text{\sc ani$_1^{(\delta)}$:} \qquad
\gamma(\vec{p}) = \sum_{j=1}^d\, \left[ \delta^2\,
|\vec{p}|^2+ p_j^2\,(1-\delta^2) \right]^{\frac12} \,,\quad
\text{with}\quad \delta > 0\,, 
\qquad && [r=1\,, L=d]\,,
\\ %= 10^{-2}\,, %\label{eq:l1norm}
& \text{{\sc ani$_2$:} $\gamma$ as on the bottom of Figure~3 in \cite{ani3d}},
\qquad && [r=1\,, L=2]\,,
\\ % cylindrical
& \text{{\sc ani$_3$:} $\gamma$ as on the right of Figure~2 in 
\cite{dendritic}}, 
\qquad && [r=1\,, L=3]\,,
\\ % hexagonal 2d
& \text{{\sc ani$_4$:} $\gamma$ as in Figure~3 in \cite{jcg}}\,, 
\qquad && [r=1\,, L=4]\,,
\\ % hexagonal 3d
& \text{{\sc ani$_9$:} $\gamma$ as on the right of Figure~3 in 
\cite{dendritic}}\,,
\qquad && [r=9\,, L=3]\,.
\end{alignat*}
We remark that {\sc ani$_1^{(\delta)}$} is a regularized $l_1$--norm, so that
its Wulff shape for $\delta$ small is given by a smoothed square (in 2d) or a
smoothed cube (in 3d) with nearly flat sides/facets. Anisotropies with such 
flat sides or facets are called crystalline. Also the choices
{\sc ani$_i$}, $i=2\to4$, represent nearly crystalline anisotropies. Here the
Wulff shapes are given by a smoothed cylinder, a smoothed hexagon and a
smoothed hexagonal prism, respectively. 
The Wulff shape for the cubic anisotropy {\sc ani$_9$} is given by a smoothed 
octahedron.
Finally, we denote by
{\sc ani$_k^\star$} the anisotropies {\sc ani$_k$}, $k=3\to4$, 
rotated by $\tfrac\pi{12}$ in
the $x_1-x_2$-plane. %, and by {\sc ani$_1^{\star}$} that anisotropy

Finally, unless otherwise stated, we choose
$\lambda = a = \conduct_\pm = 1$ and 
$\beta = \gamma$, where we recall (\ref{eq:mu}).

\subsection{Mullins--Sekerka in two space dimensions}
In this subsection we always choose $\vartheta = 0$.
We begin with an investigation into the choice of $\varrho$. At first we choose
(\ref{eq:varrho})(i).
In order to visualize the possible onset of a boundary layer as explained in
Remark~\ref{rem:bl} for the obstacle potential (\ref{eq:obstacle}), 
we present a computation for (\ref{eq:U},b) with the
initial data $\Phi^0 = \varphi_0 = 1$. For this experiment we use
$\rho=10^{-3}$.
On setting $\alpha = 1$, the critical 
value for $\uD$ in (\ref{eq:bl}) is 
$-\frac2\cPsi\,\epsilon^{-1} = -\frac4\pi\,16\,\pi = -64$. 
In our numerical
computations this lower bound appears to be sharp. In particular, we observe
that $U^n=1$ is a steady state whenever $\uD \geq -64$, but a boundary layer
forms already for e.g.\ $\uD = -64 - 10^{-8}$. The same behaviour has been
observed by the authors in \cite{eck} for the choice $\rho=0$.
As an example for the case $\rho=10^{-3}$ considered here, we present a run
for $\uD = -65$ in Figure~\ref{fig:bl}, where we can clearly see how the 
boundary layer develops. Note that this phenomenon is completely
independent from the choice of anisotropy $\gamma$. The discretization
parameters for this experiment were $N_f = N_c = 128$ and $\tau = 10^{-5}$.
Note that in the presence of $\rho>0$ we observe a convex shape in
Figure~\ref{fig:bl}, in contrast to the corresponding evolution in 
\citet[Fig.\ 4]{eck}, where $\rho=0$.
\begin{figure}
\center
\ifpdf
\includegraphics[angle=-0,width=0.19\textwidth]{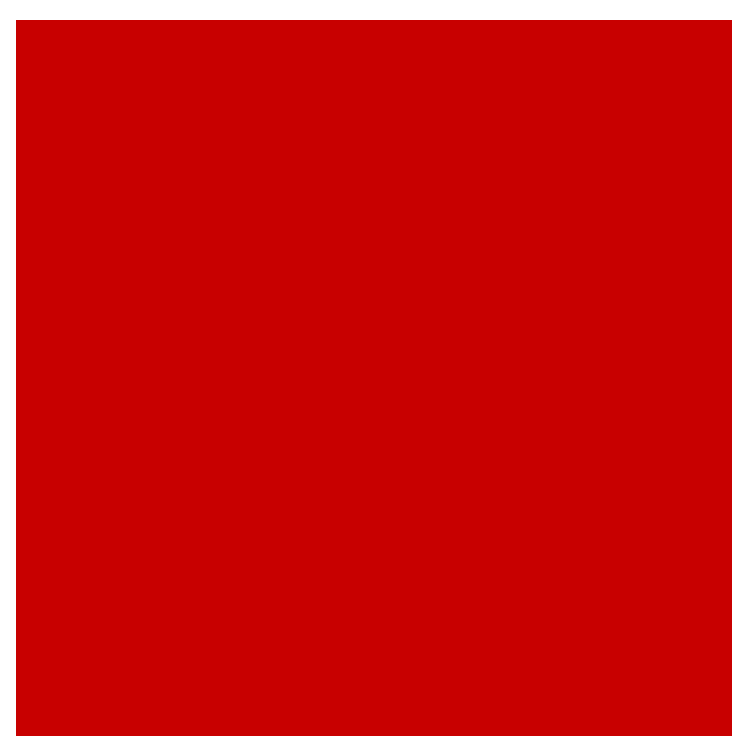}
\includegraphics[angle=-0,width=0.19\textwidth]{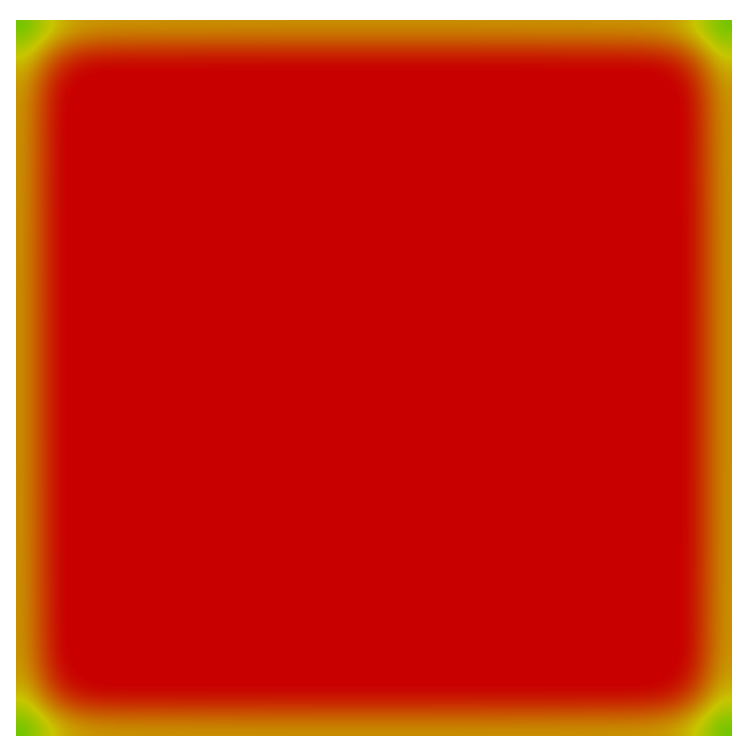}
\includegraphics[angle=-0,width=0.19\textwidth]{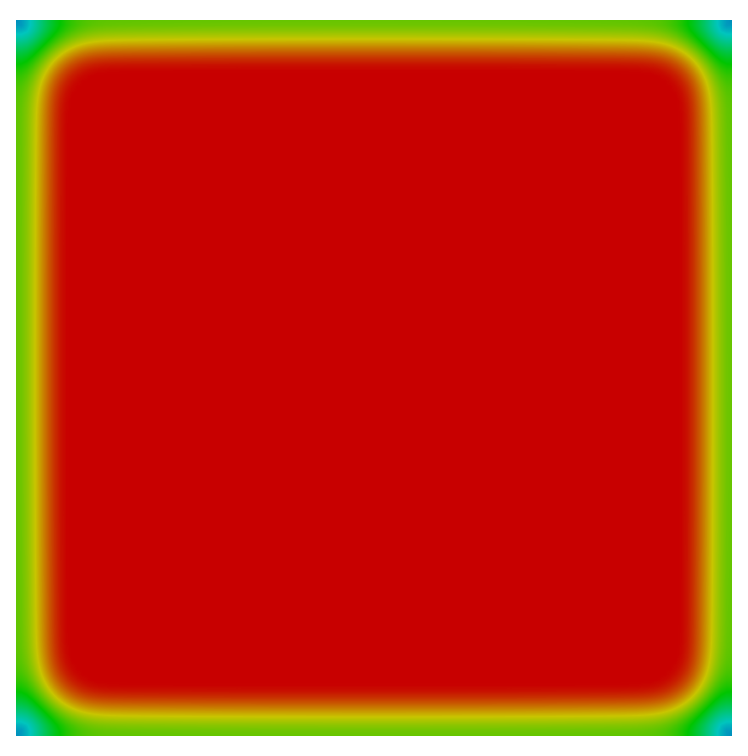}
\includegraphics[angle=-0,width=0.19\textwidth]{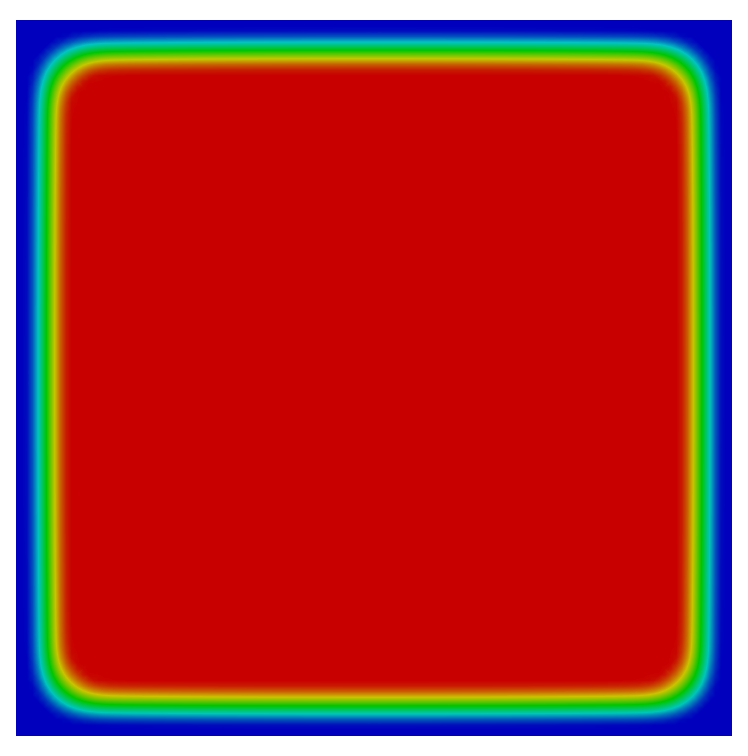}
\includegraphics[angle=-0,width=0.19\textwidth]{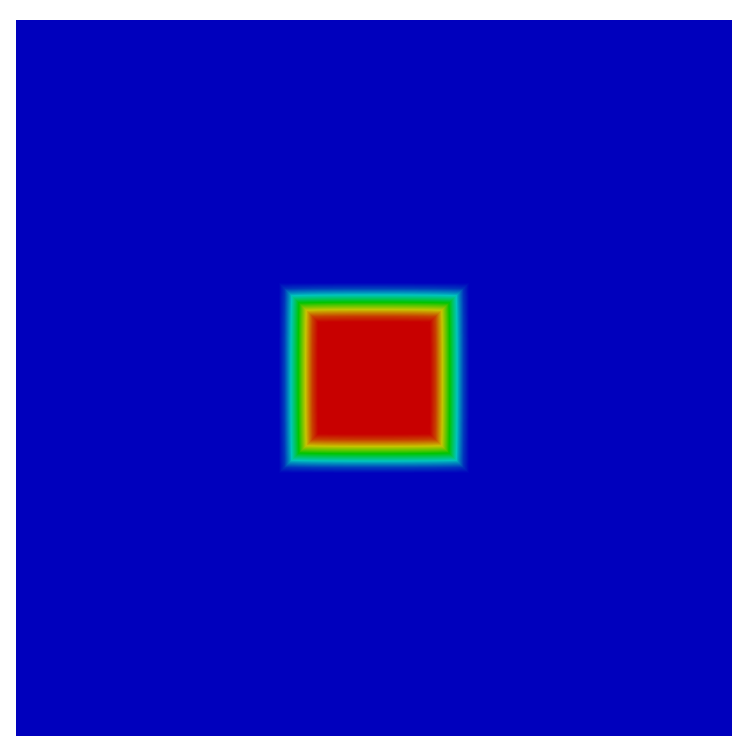}
\fi
\caption{($\epsilon^{-1} = 16\,\pi$, {\sc ani$_1^{(0.01)}$}, 
(\ref{eq:varrho})(i), $\alpha=1$, $\rho=10^{-3}$, $\uD = -65$,
$\Omega=(-\frac12,\frac12)^2$)
Creation of a boundary layer for the scheme (\ref{eq:U},b).
Snapshots of the solution at times $t=0,\,4\times10^{-5},\,5\times10^{-5},\,
\,7\times10^{-5},\,10^{-3}$.
}
\label{fig:bl}
\end{figure}%

For the approximation (\ref{eq:sU},b), i.e.\ in the case of the smooth quartic
potential (\ref{eq:quartic}), 
we observe that the criterion
(\ref{eq:sbl}) is of course not sharp, in the sense that even for values
$0 > \uD \geq -41$ no boundary layer forms in practice, even though 
(\ref{eq:sbl}) is then violated for the shape function (\ref{eq:varrho})(i). 
What happens in practice
is that $\Phi^n$ attains values less than $1$, without forming an interface,
i.e.\ $\min_{x \in \Omega} \Phi^n(x) > 0$. However, for the value 
$\uD=-42$, with the remaining parameters fixed as in Figure~\ref{fig:bl}, we do
observe the creation of a boundary layer. The evolution can be seen in
Figure~\ref{fig:sbl}. We remark that the colour range in Figure~\ref{fig:sbl}
is from red for $1$ to blue for $-1$, as in Figure~\ref{fig:bl}, even though
the extremal values for $\Phi^n$ during the evolution are approximately 
$1.03$ and $-1.16$, respectively.
\begin{figure}
\center
\ifpdf
\includegraphics[angle=-0,width=0.19\textwidth]{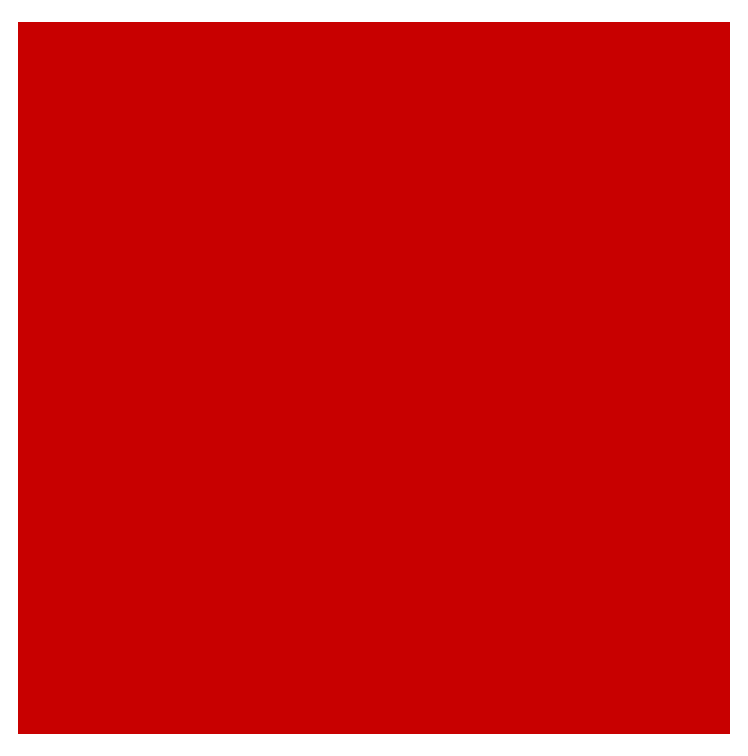}
\includegraphics[angle=-0,width=0.19\textwidth]{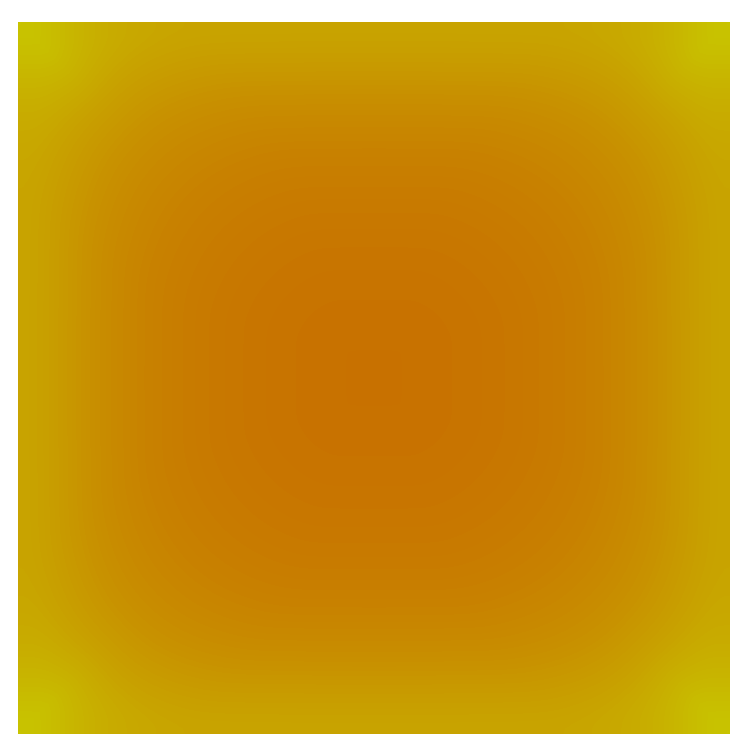}
\includegraphics[angle=-0,width=0.19\textwidth]{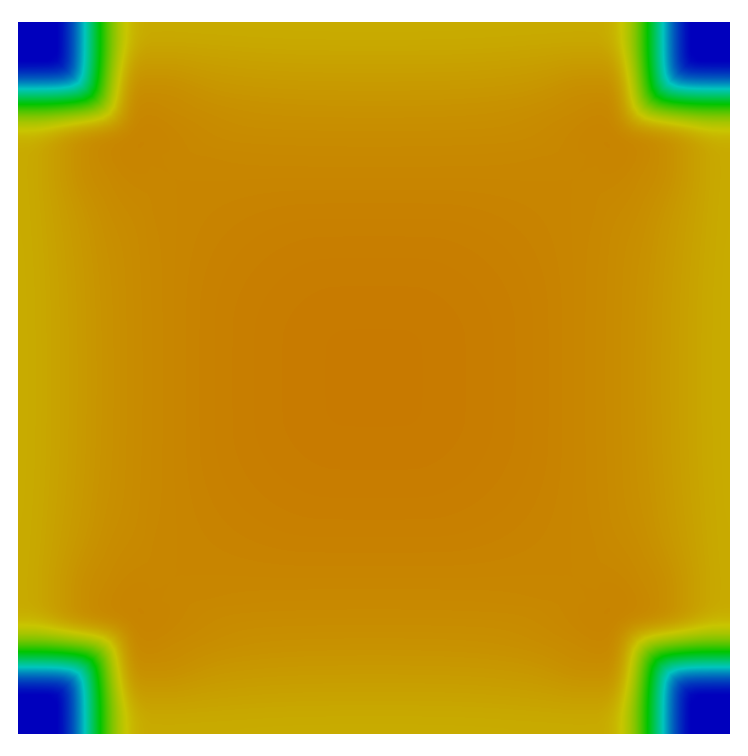}
\includegraphics[angle=-0,width=0.19\textwidth]{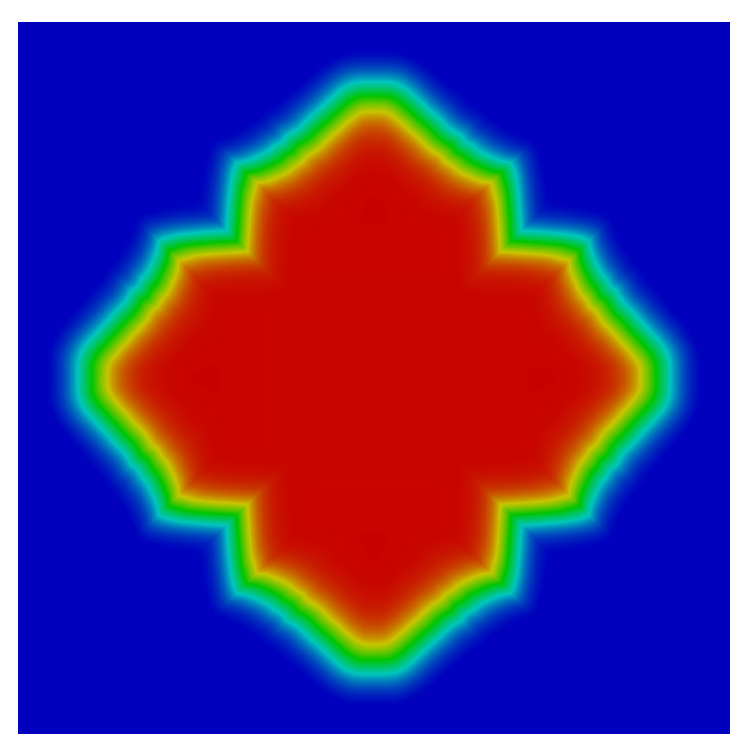}
\includegraphics[angle=-0,width=0.19\textwidth]{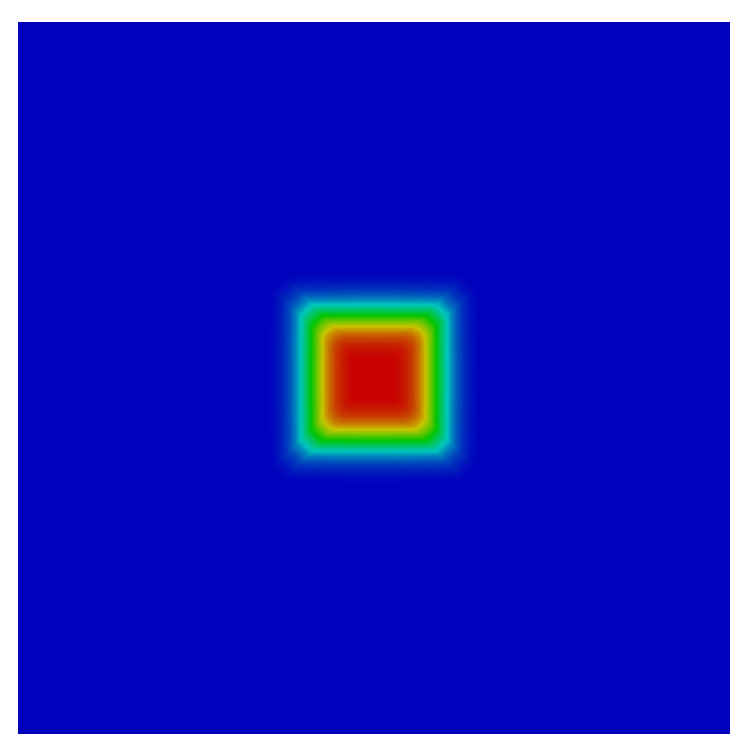}
\fi
\caption{
($\epsilon^{-1} = 16\,\pi$, {\sc ani$_1^{(0.01)}$}, 
(\ref{eq:varrho})(i), $\alpha=1$, $\rho=10^{-3}$, $\uD = -42$,
$\Omega=(-\frac12,\frac12)^2$)
Creation of a boundary layer for the scheme (\ref{eq:sU},b).
Snapshots of the solution at times $t=0,\,8\times10^{-4},\,10^{-3},\,
1.5\times10^{-3},\,2.5\times10^{-3}$.
}
\label{fig:sbl}
\end{figure}%
Finally, we recall that if we choose the shape functions 
(\ref{eq:varrho})(ii) or (\ref{eq:varrho})(iii)
instead, then the conditions (\ref{eq:bl}) and (\ref{eq:sbl}) yield that
$\Phi^n=1$ and $W^n=\uD$ for all $n\geq1$ for the two schemes (\ref{eq:U},b) 
and (\ref{eq:sU},b), respectively.

In the following experiments, we return to the initial data described
previously, so that $\varphi_0$ models a circular interface of radius
$R_0=0.1$. We also set $\rho = 0$.
For convenience we recall the simulation from \citet[Fig.\ 5]{eck}, so that
(\ref{eq:varrho})(i) applies. 
Here $H=8$ and $\vartheta = \rho = 0$.
Moreover, $\uD = -2$ and $\alpha = 0.03$; and we observe  
that for this choice of parameters the condition (\ref{eq:bl}) is satisfied
if we choose $\epsilon^{-1} = 32\,\pi > \frac{50}3\,\pi$.
A run for (\ref{eq:U},b), 
with the discretization parameters $N_f = 4096$, $N_c = 128$, $\tau = 10^{-4}$
and $T=7$ is shown in Figure~\ref{fig:BSnew32pi}.
\begin{figure}
\center
\ifpdf
\includegraphics[angle=-0,width=0.19\textwidth]{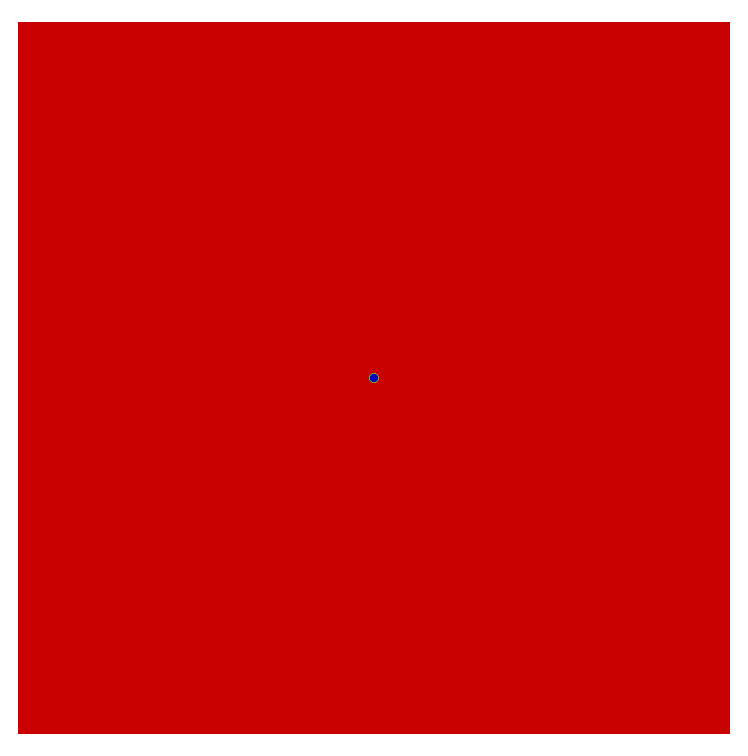}
\includegraphics[angle=-0,width=0.19\textwidth]{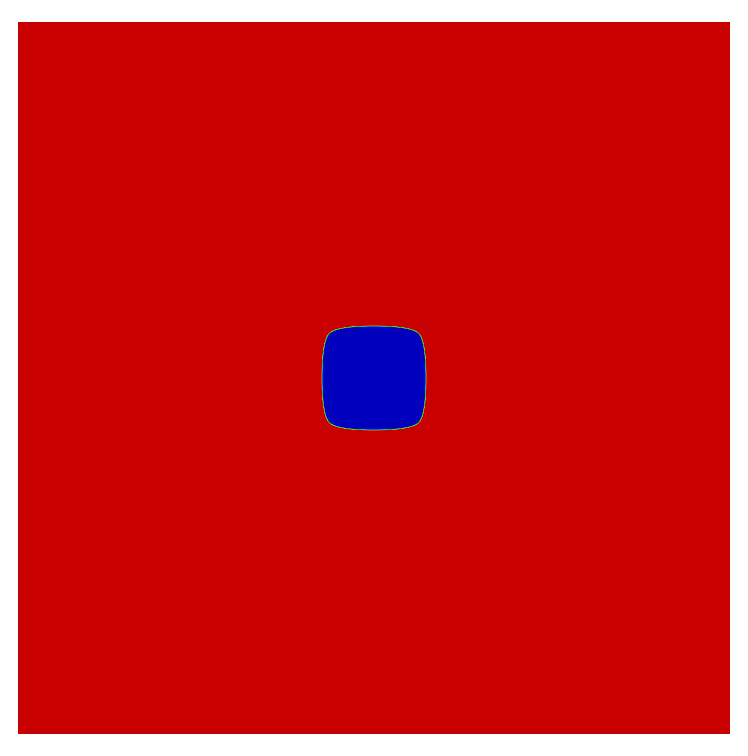}
\includegraphics[angle=-0,width=0.19\textwidth]{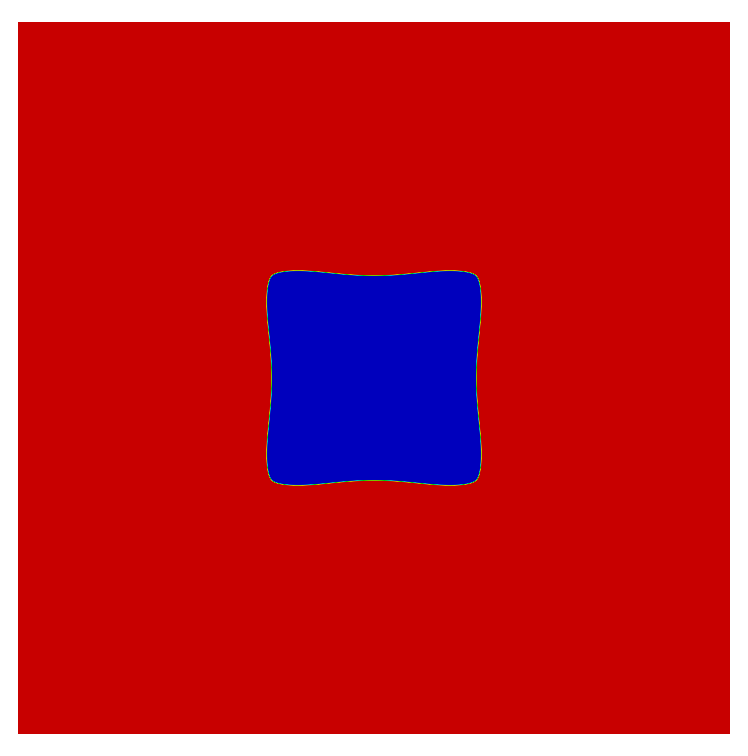}
\includegraphics[angle=-0,width=0.19\textwidth]{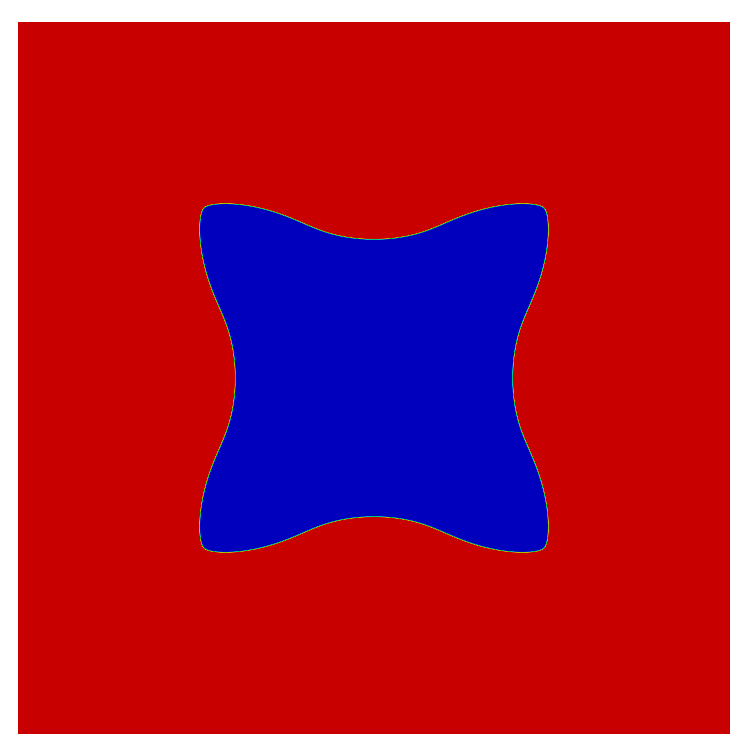}
\includegraphics[angle=-0,width=0.19\textwidth]{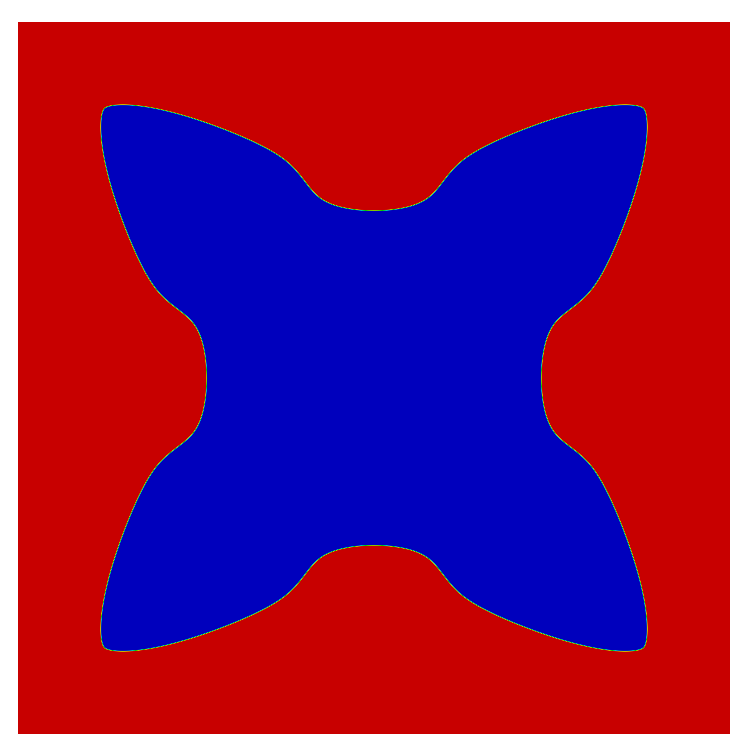}
\fi
\caption{($\epsilon^{-1} = 32\,\pi$, {\sc ani$_1^{(0.3)}$}, 
(\ref{eq:varrho})(i), $\alpha=0.03$, $\rho = 0$, $\uD = -2$, $\Omega=(-8,8)^2$)
Snapshots of the solution at times $t=0,\,1,\,3,\,5,\,7$.
[This computation took $6$ days.]
}
\label{fig:BSnew32pi}
\end{figure}%
Now the advantage of the shape function choices (\ref{eq:varrho})(ii) and
(\ref{eq:varrho})(iii) over the simple choice (\ref{eq:varrho})(i), as
highlighted in Remark~\ref{rem:bl}, is that for the same physical parameters a
larger value of $\epsilon$ can be chosen. To illustrate this, 
we repeat the simulation from Figure~\ref{fig:BSnew32pi} but now for the
choices (\ref{eq:varrho})(ii) and (\ref{eq:varrho})(iii). This means that 
we can use e.g.\ $\epsilon^{-1} =
8\,\pi$ together with the coarser discretization parameters 
$N_f = 1024$, $N_c = 128$ and $\tau = 10^{-3}$. %and $T=7$. 
The new results are shown in Figures~\ref{fig:BSnewii} and \ref{fig:BSnewv}, 
where we observe the
good qualitative agreement with Figure~\ref{fig:BSnew32pi}. We draw
particular attention to the dramatic reduction in CPU time necessary to compute
the respective simulations.
\begin{figure}
\center
\ifpdf
\includegraphics[angle=-0,width=0.19\textwidth]{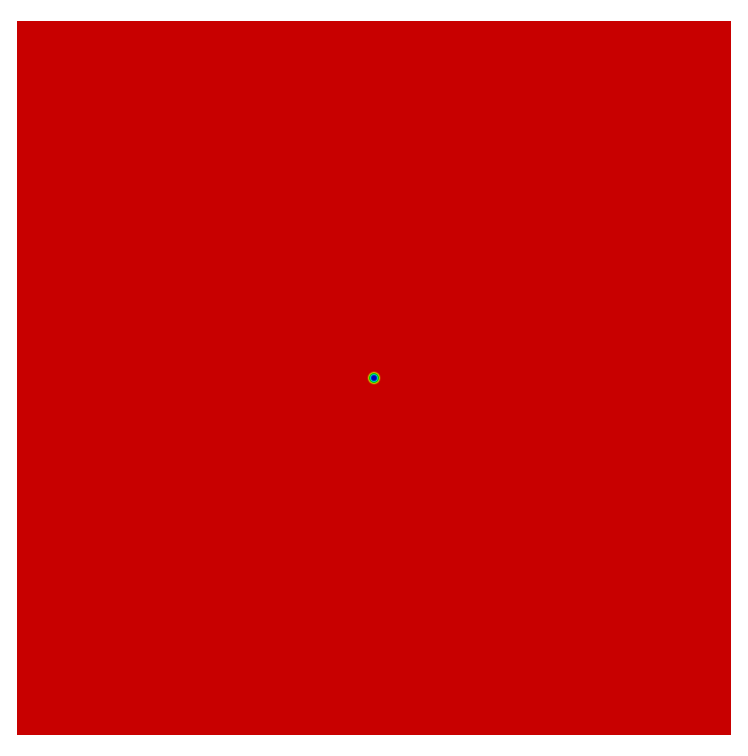}
\includegraphics[angle=-0,width=0.19\textwidth]{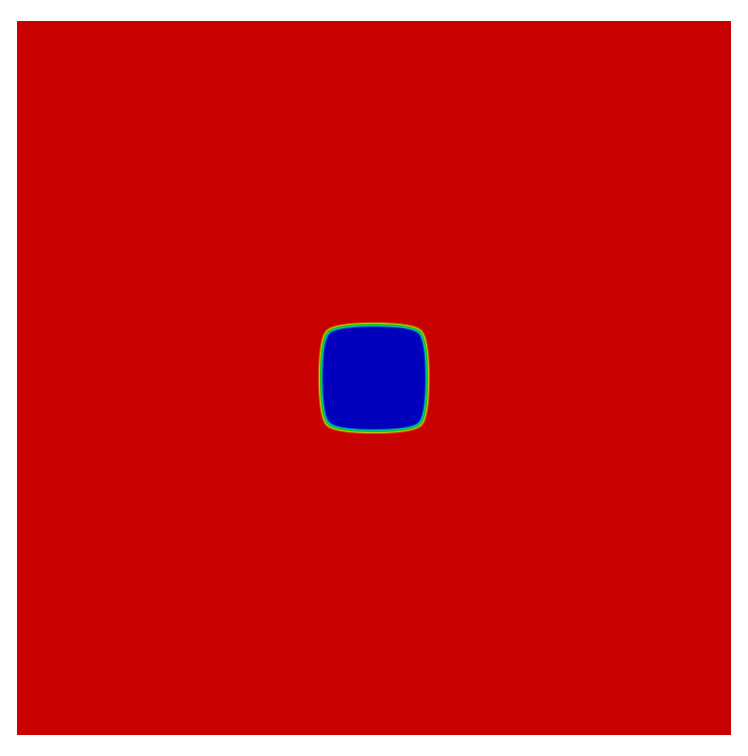}
\includegraphics[angle=-0,width=0.19\textwidth]{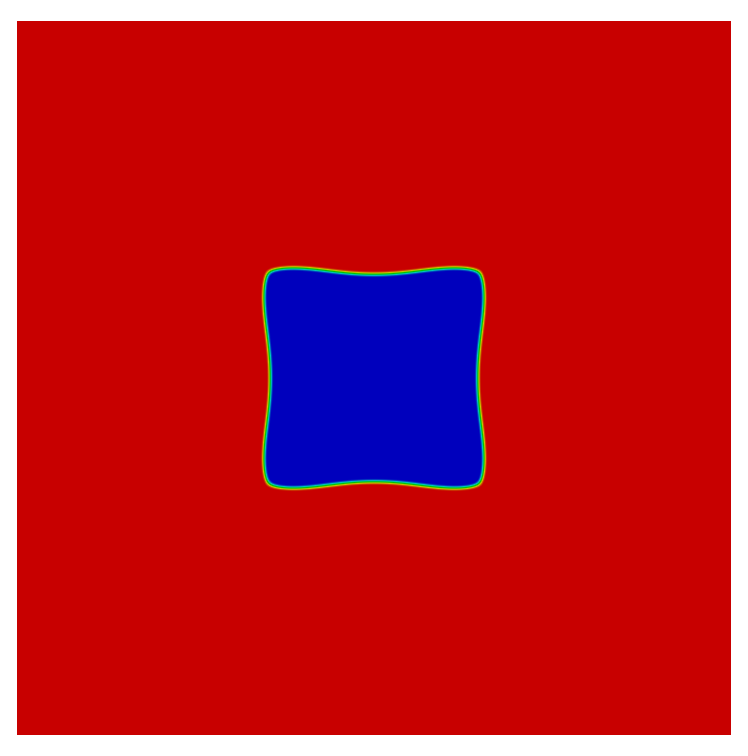}
\includegraphics[angle=-0,width=0.19\textwidth]{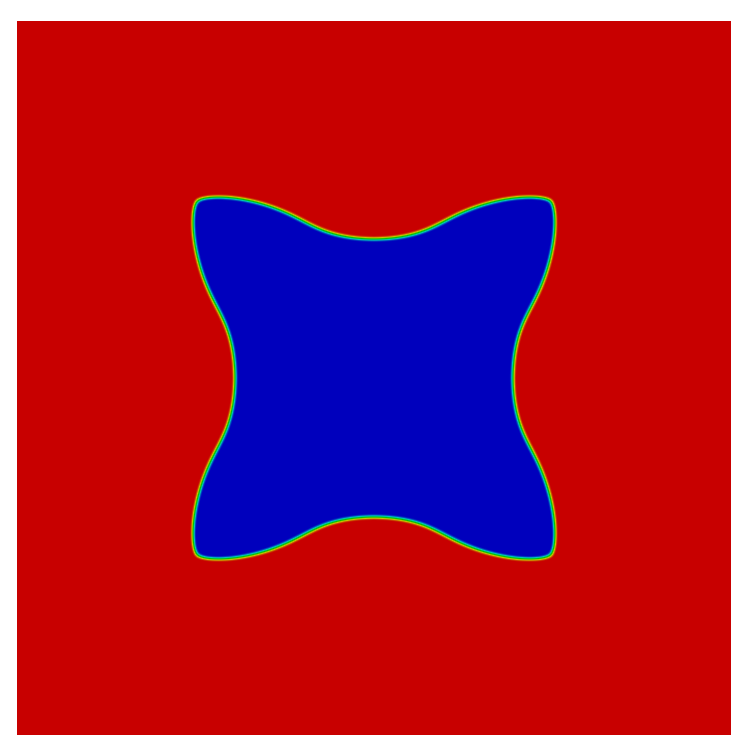}
\includegraphics[angle=-0,width=0.19\textwidth]{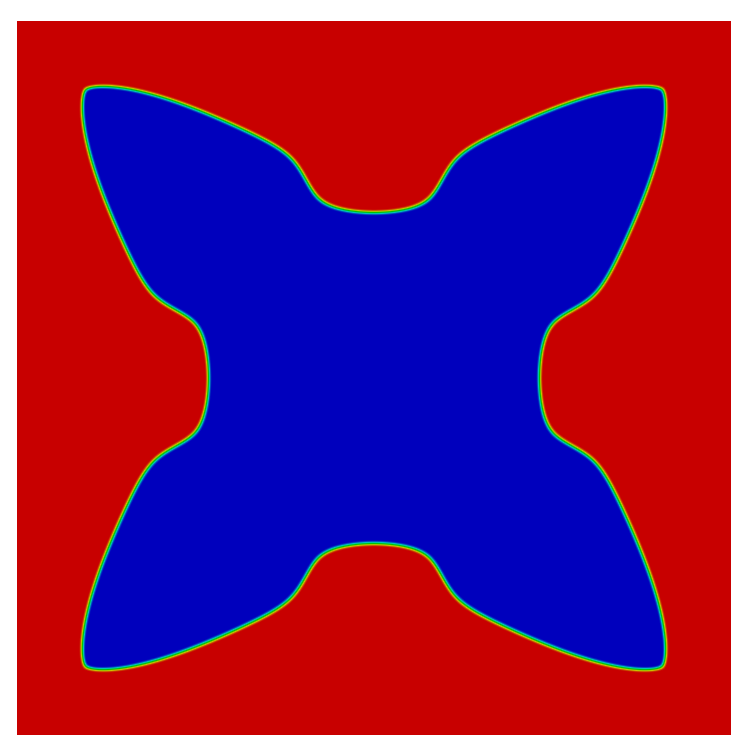}
\fi
\caption{($\epsilon^{-1} = 8\,\pi$, {\sc ani$_1^{(0.3)}$}, 
(\ref{eq:varrho})(ii), $\alpha=0.03$, $\rho = 0$, $\uD = -2$, 
$\Omega=(-8,8)^2$)
Snapshots of the solution at times $t=0,\,1,\,3,\,5,\,7$.
[This computation took $4.5$ hours.]
}
\label{fig:BSnewii}
\end{figure}%
\begin{figure}
\center
\ifpdf
\includegraphics[angle=-0,width=0.19\textwidth]{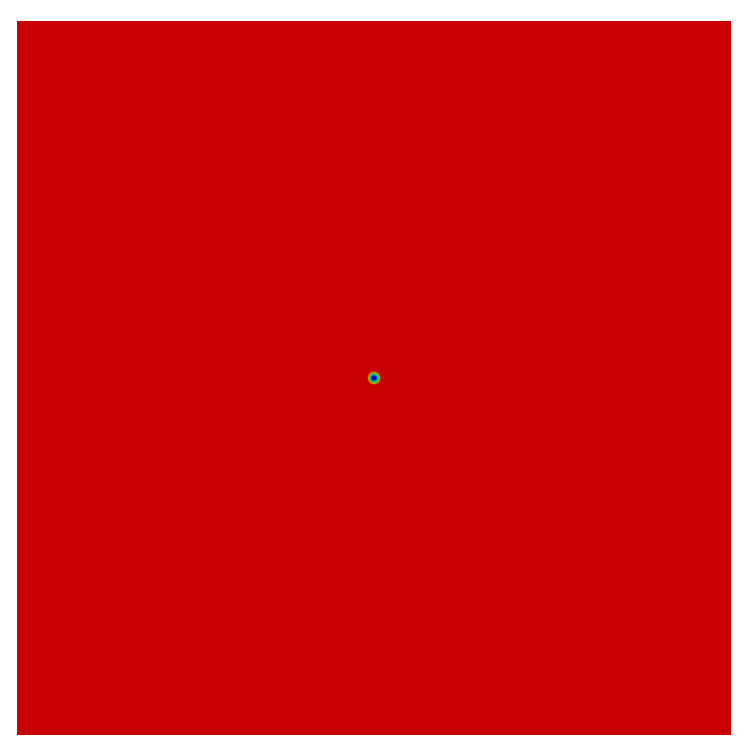}
\includegraphics[angle=-0,width=0.19\textwidth]{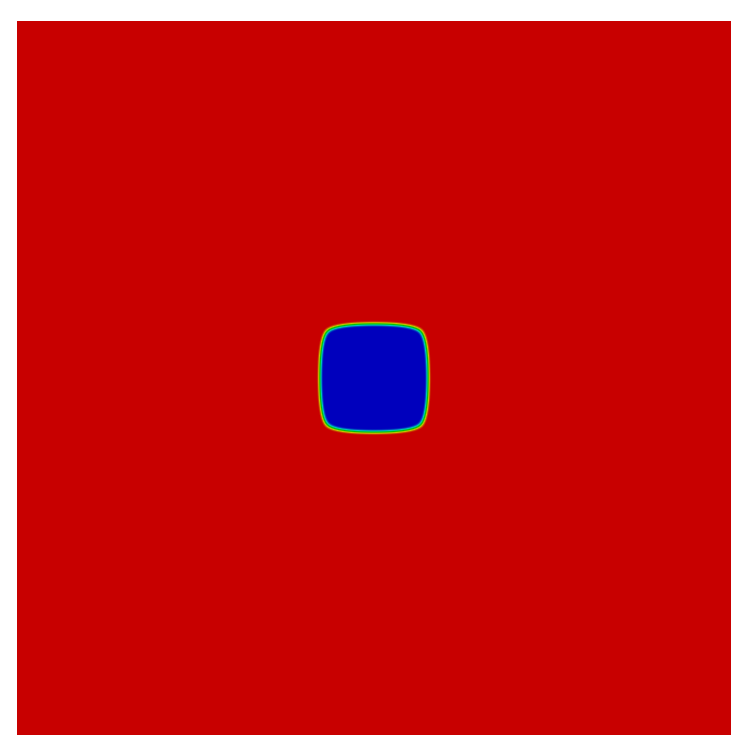}
\includegraphics[angle=-0,width=0.19\textwidth]{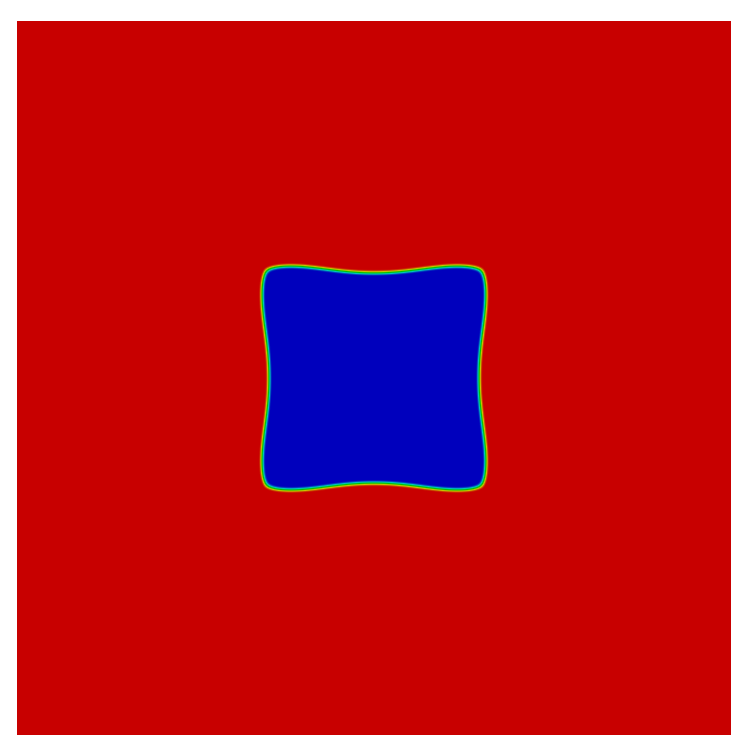}
\includegraphics[angle=-0,width=0.19\textwidth]{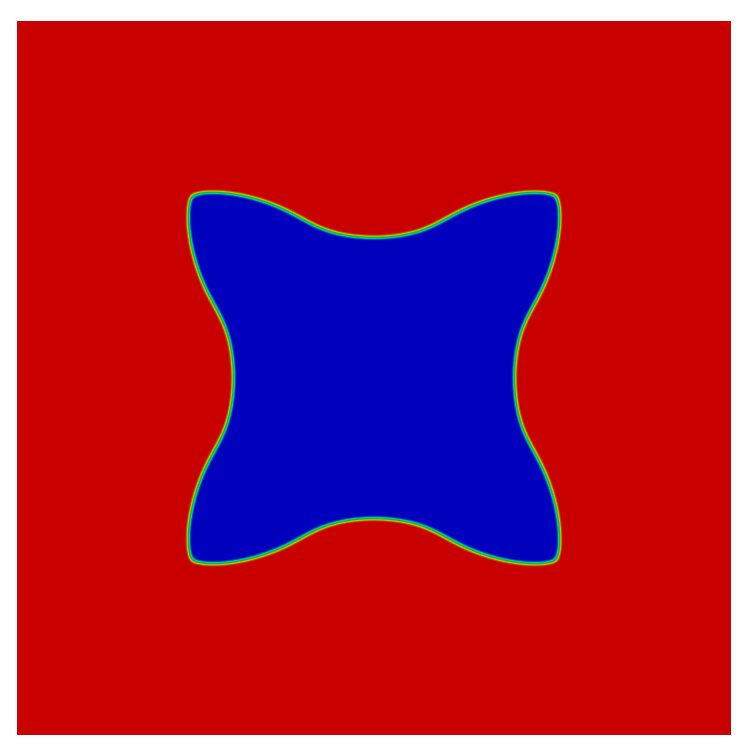}
\includegraphics[angle=-0,width=0.19\textwidth]{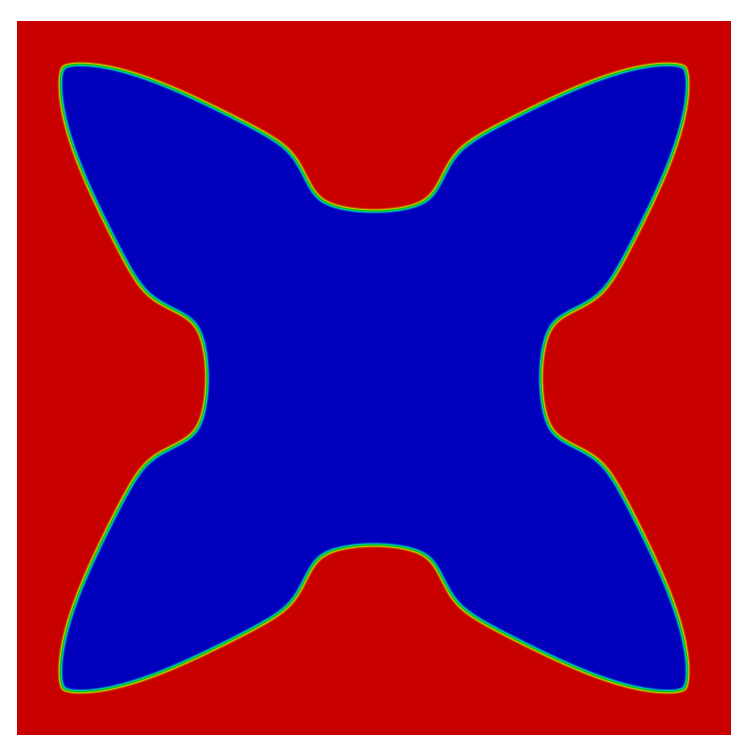}
\fi
\caption{($\epsilon^{-1} = 8\,\pi$, {\sc ani$_1^{(0.3)}$}, 
(\ref{eq:varrho})(iii), $\alpha=0.03$, $\rho = 0$, $\uD = -2$, 
$\Omega=(-8,8)^2$)
Snapshots of the solution at times $t=0,\,1,\,3,\,5,\,7$.
[This computation took $10.5$ hours.]
}
\label{fig:BSnewv}
\end{figure}%
While a further reduction in $\epsilon^{-1}$ and in the discretization
parameters $N_f$, $N_c$ and $\tau^{-1}$
leads to even bigger gains in computation times, the larger values
of $\epsilon$ soon lead to a loss of accuracy with respect to the approximation
of the underlying sharp interface problem (\ref{eq:1a}--e). We illustrate this
with an example for
$\epsilon^{-1} = 2\,\pi$ for the choice (\ref{eq:varrho})(ii)
together with $N_f = 256$, $N_c = 64$, 
$\tau = 10^{-2}$ and $T=6$. The results are shown in 
Figure~\ref{fig:BSnewii_2pi}, where we observe a qualitative difference to the
three previous simulations.
\begin{figure}
\center
\ifpdf
\includegraphics[angle=-0,width=0.19\textwidth]{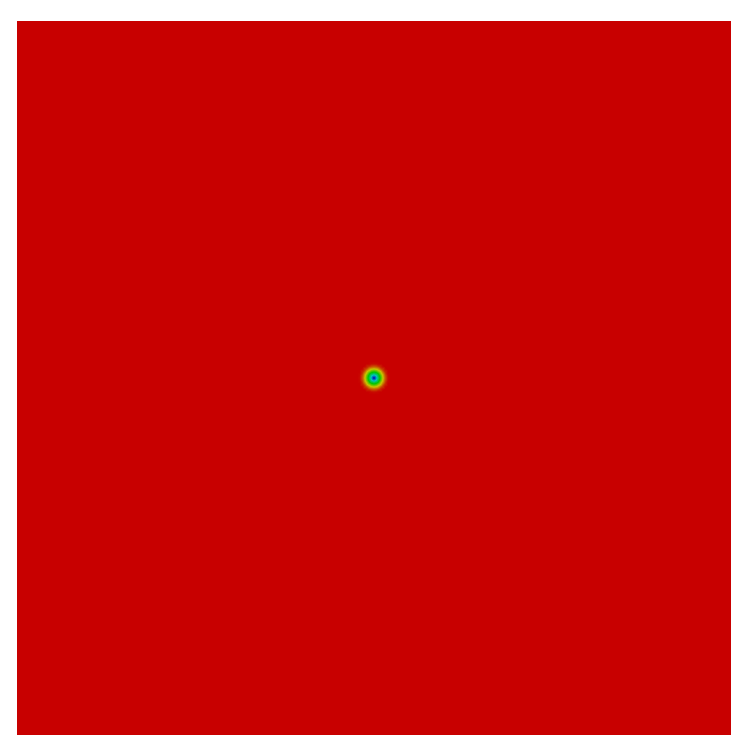}
\includegraphics[angle=-0,width=0.19\textwidth]{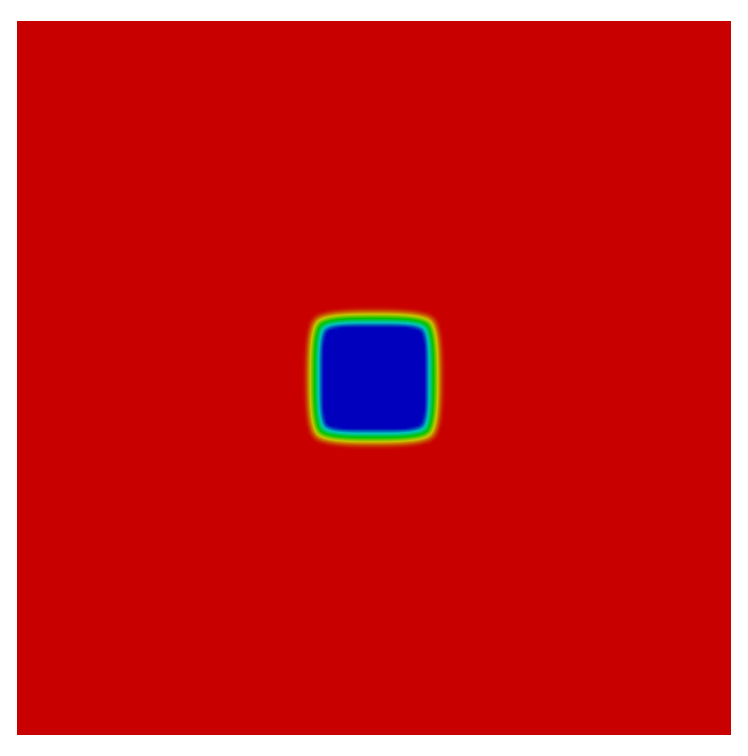}
\includegraphics[angle=-0,width=0.19\textwidth]{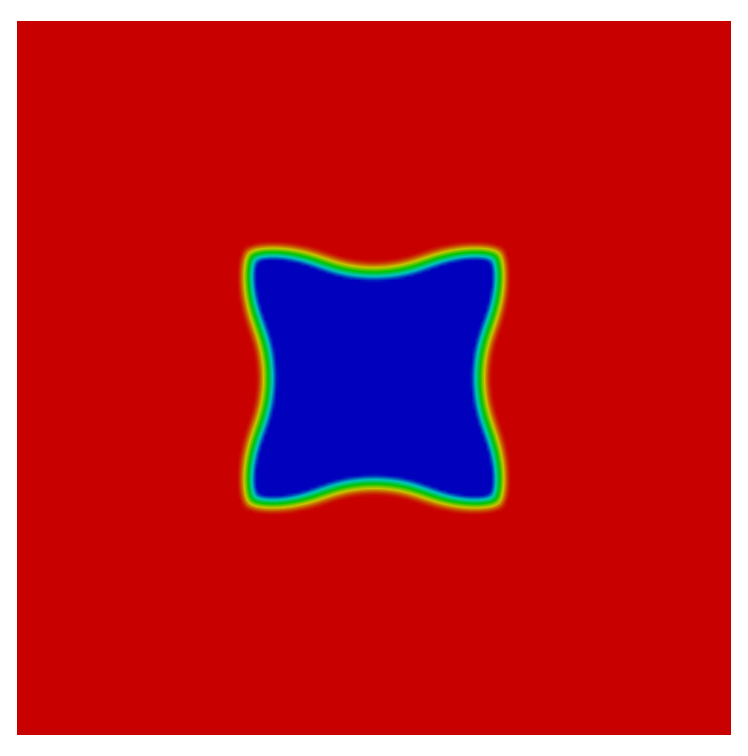}
\includegraphics[angle=-0,width=0.19\textwidth]{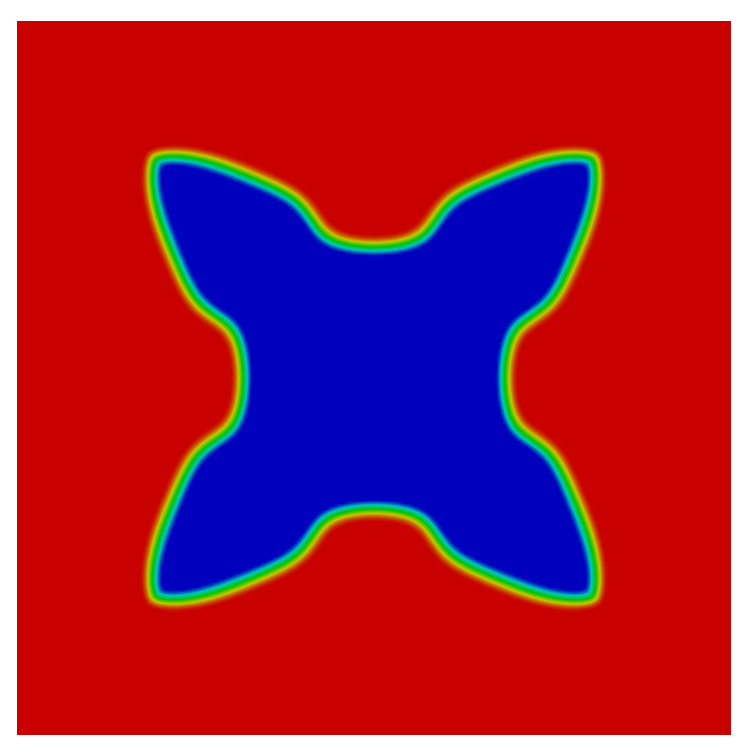}
\includegraphics[angle=-0,width=0.19\textwidth]{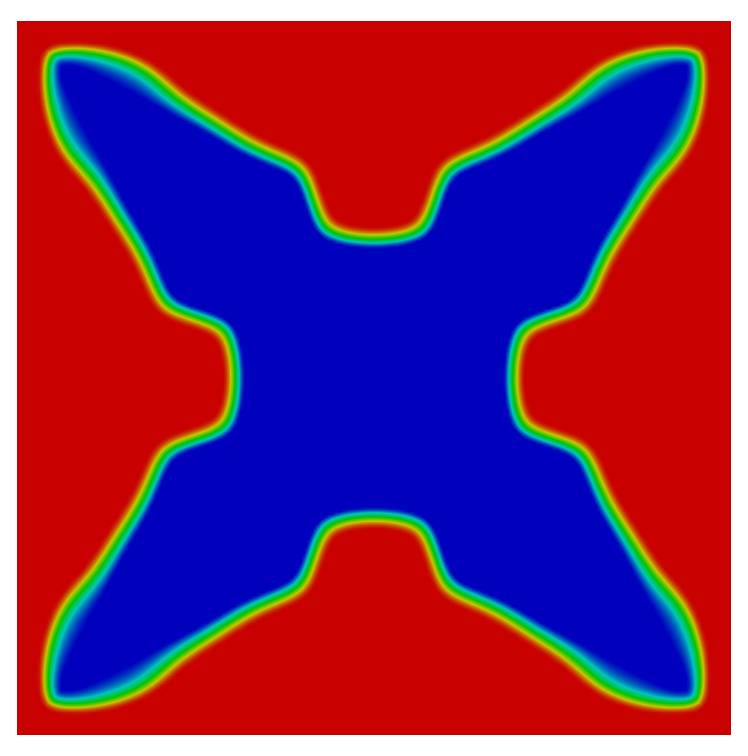}
\fi
\caption{($\epsilon^{-1} = 2\,\pi$, {\sc ani$_1^{(0.3)}$}, 
(\ref{eq:varrho})(ii), $\alpha=0.03$, $\rho = 0$, $\uD = -2$, 
$\Omega=(-8,8)^2$)
Snapshots of the solution at times $t=0,\,1,\,3,\,5,\,6$.
[This computation took $3$ minutes.]
}
\label{fig:BSnewii_2pi}
\end{figure}%

For the remainder of the simulations in this subsection we 
continue to employ (\ref{eq:varrho})(ii), but we now choose $\rho=0.01$
A simulation corresponding to Figure~\ref{fig:BSnewii_2pi} can be seen in
Figure~\ref{fig:BSrhoii_2pi}. We observe that in this example, the presence
of kinetic undercooling ($\rho>0$) only has a small influence on the overall
evolution.
\begin{figure}
\center
\ifpdf
\includegraphics[angle=-0,width=0.19\textwidth]{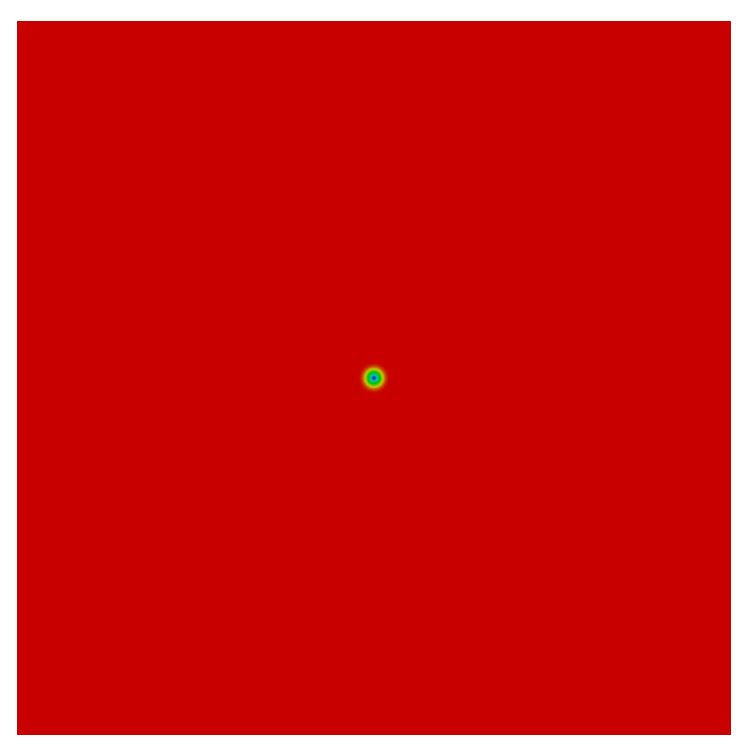}
\includegraphics[angle=-0,width=0.19\textwidth]{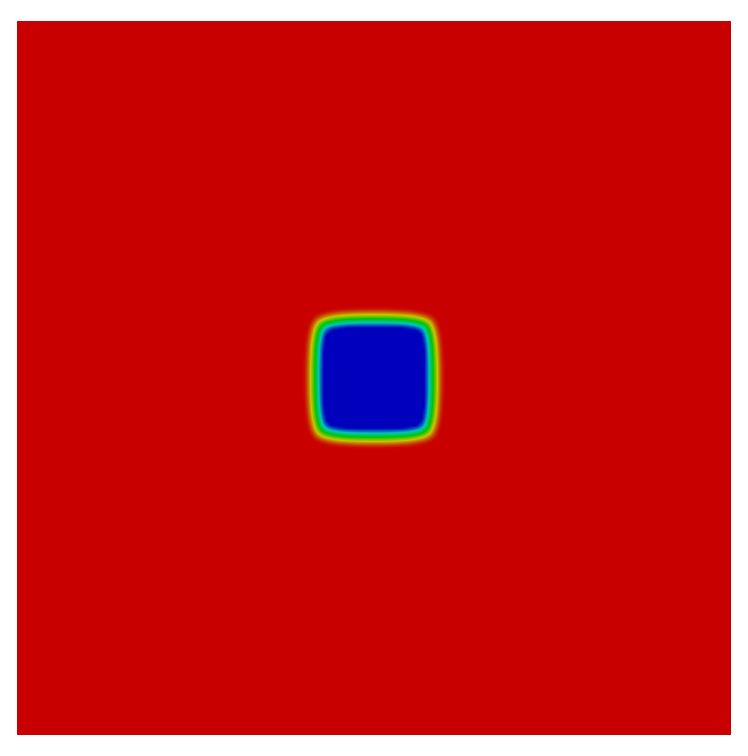}
\includegraphics[angle=-0,width=0.19\textwidth]{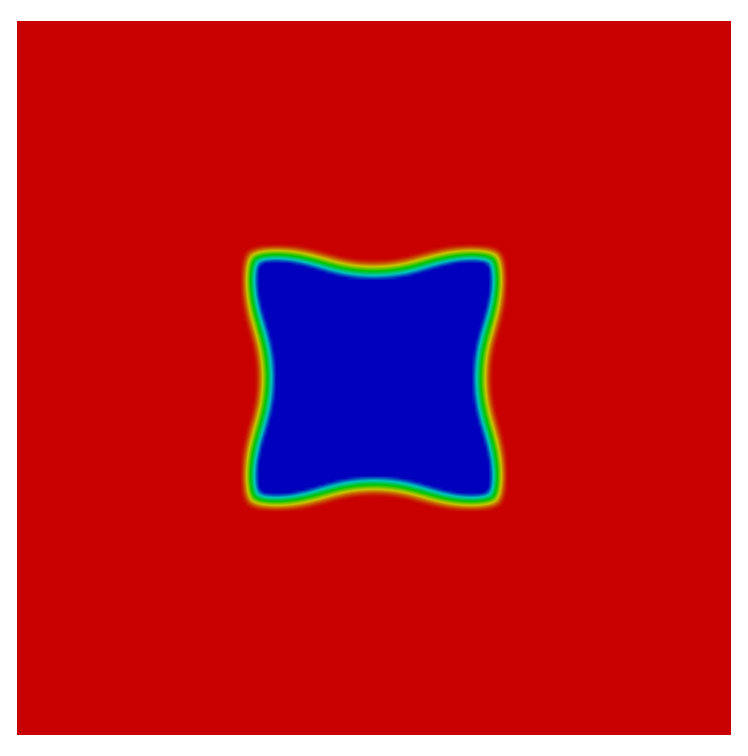}
\includegraphics[angle=-0,width=0.19\textwidth]{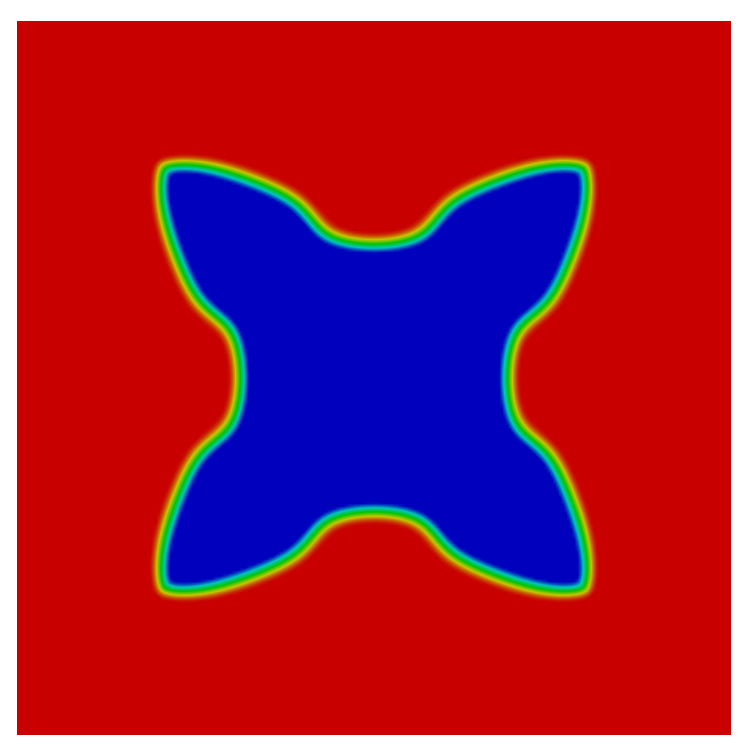}
\includegraphics[angle=-0,width=0.19\textwidth]{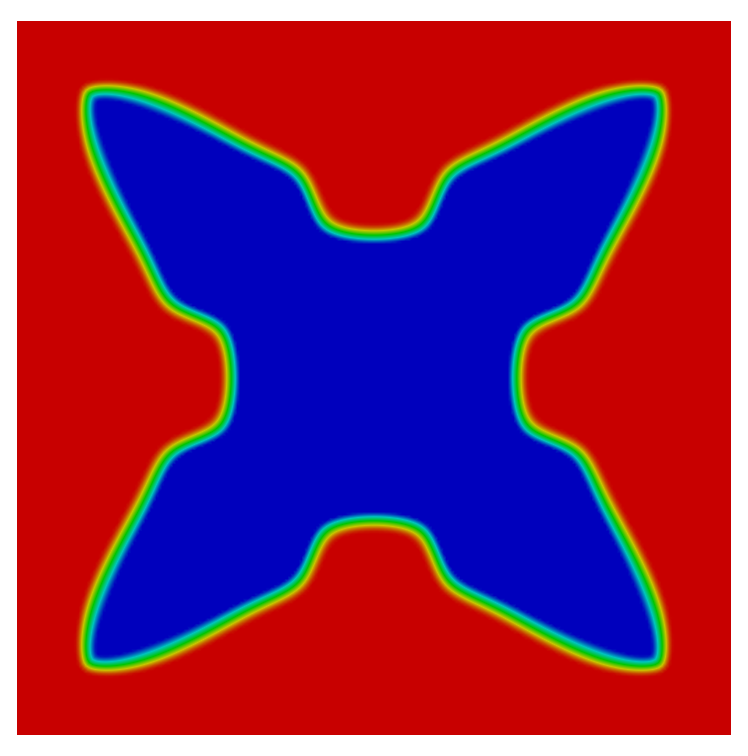}
\fi
\caption{($\epsilon^{-1} = 2\,\pi$, {\sc ani$_1^{(0.3)}$}, 
(\ref{eq:varrho})(ii), $\alpha=0.03$, $\rho = 0.01$, $\uD = -2$, 
$\Omega=(-8,8)^2$)
Snapshots of the solution at times $t=0,\,1,\,3,\,5,\,6$.
[This computation took $4$ minutes.]
}
\label{fig:BSrhoii_2pi}
\end{figure}%

The remaining computations in this subsection are for the rotated 
hexagonal anisotropy {\sc ani$_3^\star$}. The first simulation is analogous to
Figure~\ref{fig:BSrhoii_2pi}, but now on the larger domain $\Omega=(-16,16)^2$.
In particular, we keep all the parameters as before, apart from $\gamma$ and
apart from $N_f= 512$, $N_c=128$ due to the increased value of $H$.
The results are shown in Figure~\ref{fig:hexBSii_2pi_large}.
\begin{figure}
\center
\ifpdf
\includegraphics[angle=-0,width=0.19\textwidth]{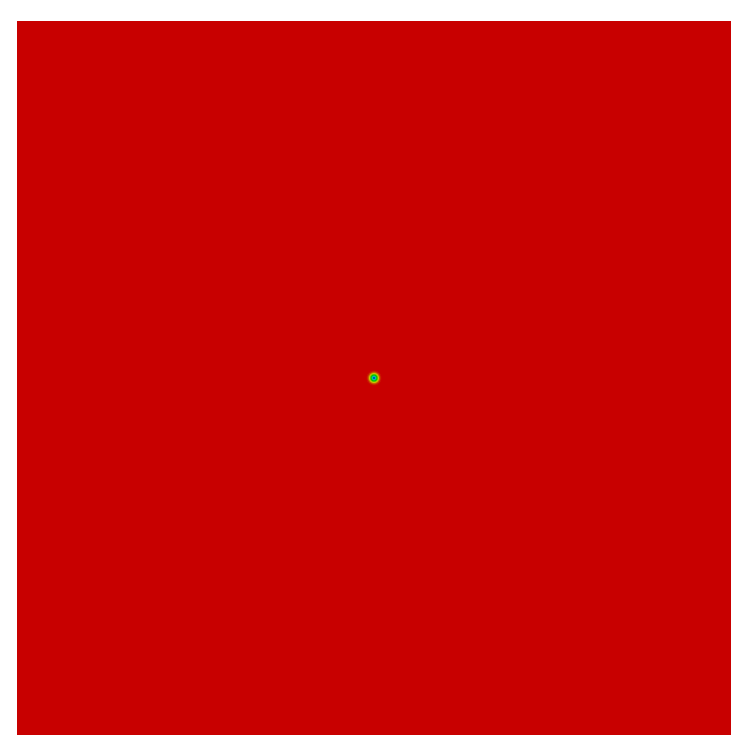}
\includegraphics[angle=-0,width=0.19\textwidth]{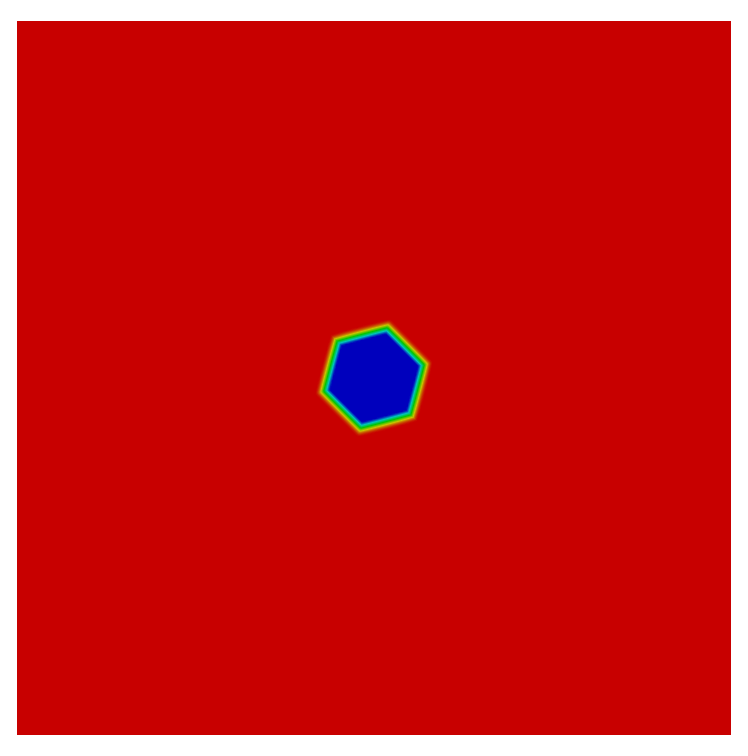}
\includegraphics[angle=-0,width=0.19\textwidth]{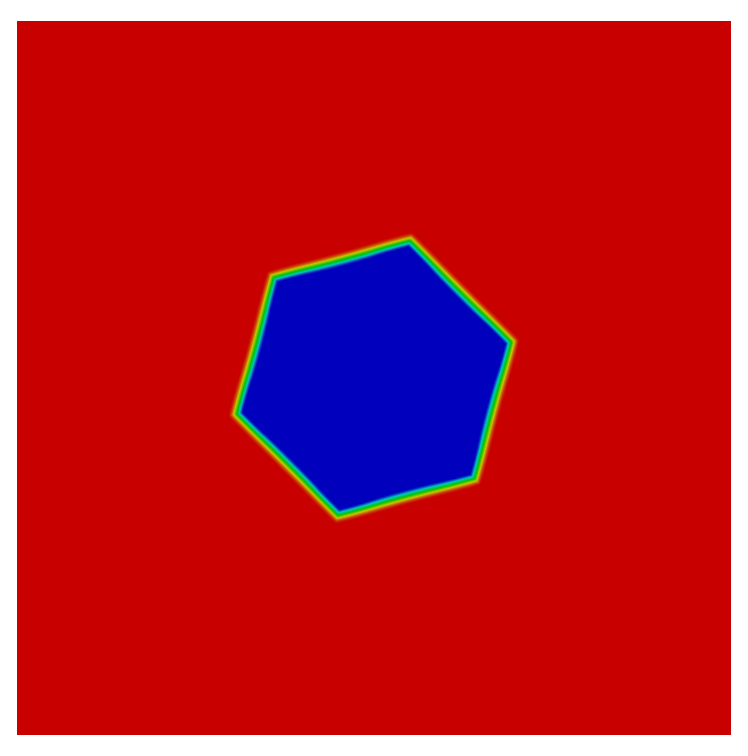}
\includegraphics[angle=-0,width=0.19\textwidth]{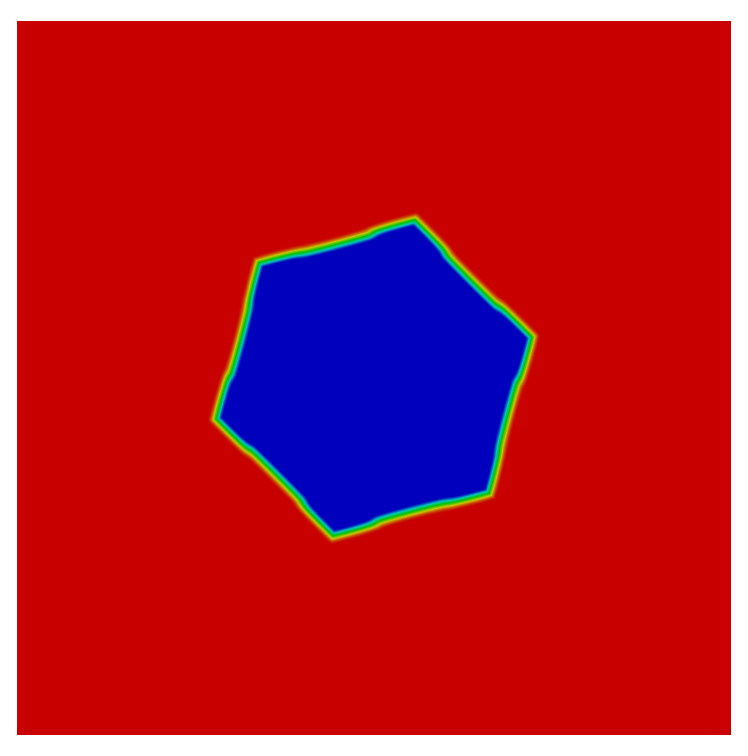}
\includegraphics[angle=-0,width=0.19\textwidth]{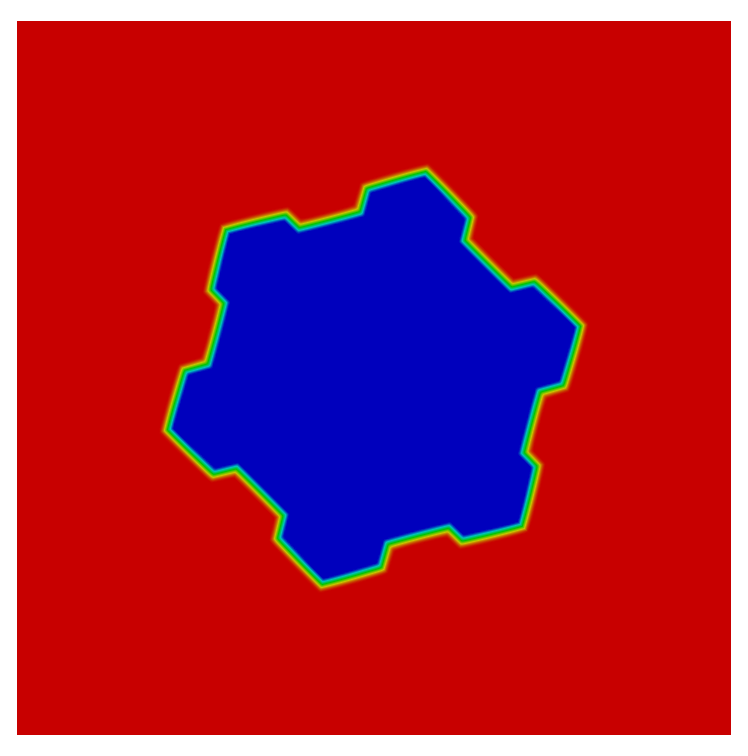}
\fi
\caption{($\epsilon^{-1} = 2\,\pi$, {\sc ani$_3^\star$}, (\ref{eq:varrho})(ii),
$\alpha=0.03$, $\rho = 0.01$, $\uD = -2$, $\Omega=(-16,16)^2$)
Snapshots of the solution at times $t=0,\,1,\,5,\,6,\,8$.
[This computation took $25$ minutes.]
}
\label{fig:hexBSii_2pi_large}
\end{figure}%
We have seen in previous simulations that the value of $\epsilon$ can have a
large influence on the evolution of the phase field approximation.
Reassuringly, in this example the evolution remains qualitatively unchanged if
we repeat the simulation for $\epsilon^{-1} = 4\,\pi$.
A run with $N_f= 1024$, $N_c=256$, $\tau=10^{-3}$ and $T=8$ 
is shown in Figure~\ref{fig:hexBSii_4pi_large}.
\begin{figure}
\center
\ifpdf
\includegraphics[angle=-0,width=0.19\textwidth]{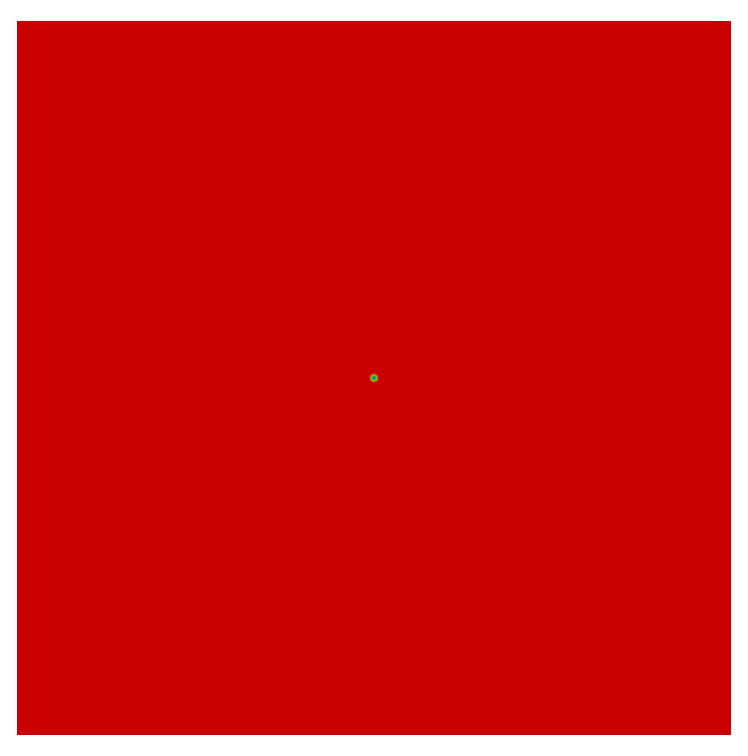}
\includegraphics[angle=-0,width=0.19\textwidth]{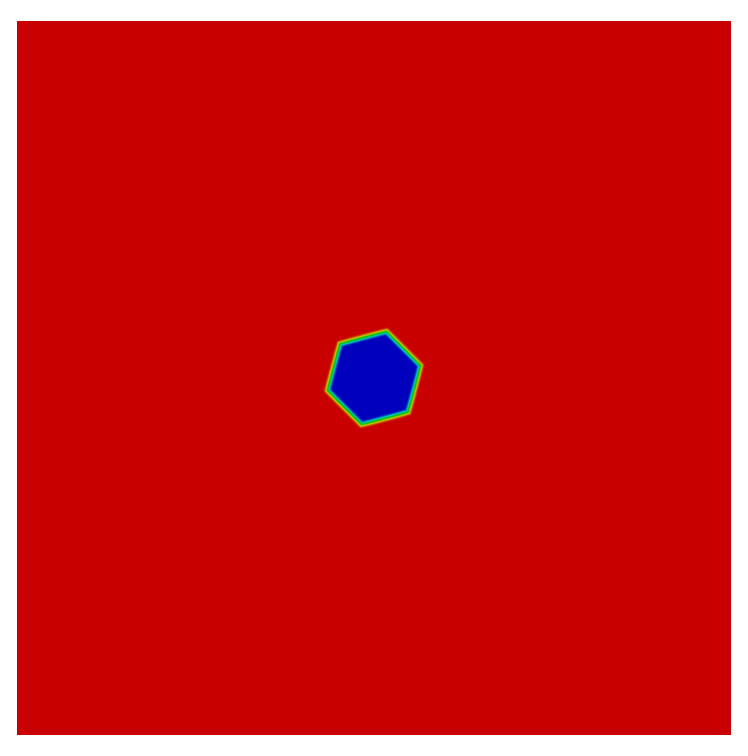}
\includegraphics[angle=-0,width=0.19\textwidth]{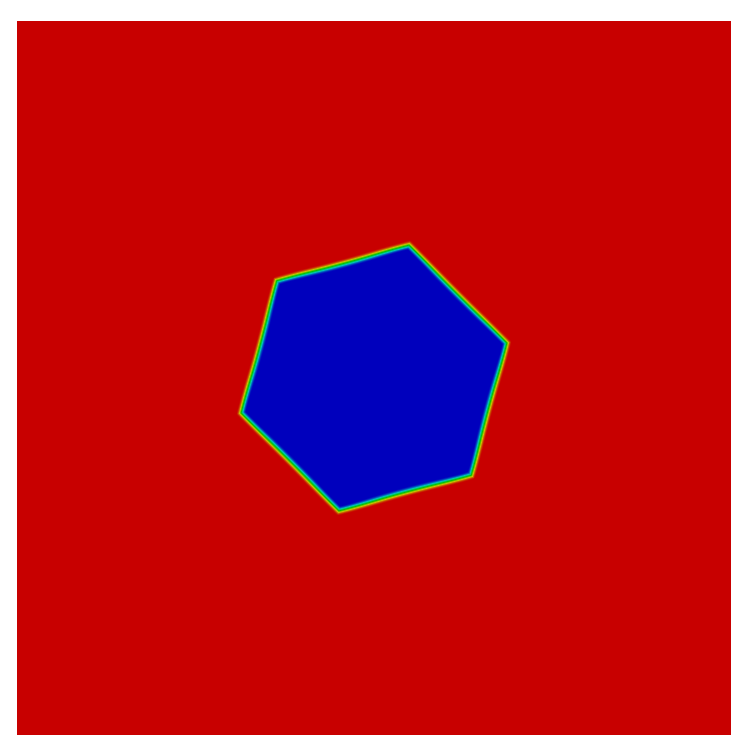}
\includegraphics[angle=-0,width=0.19\textwidth]{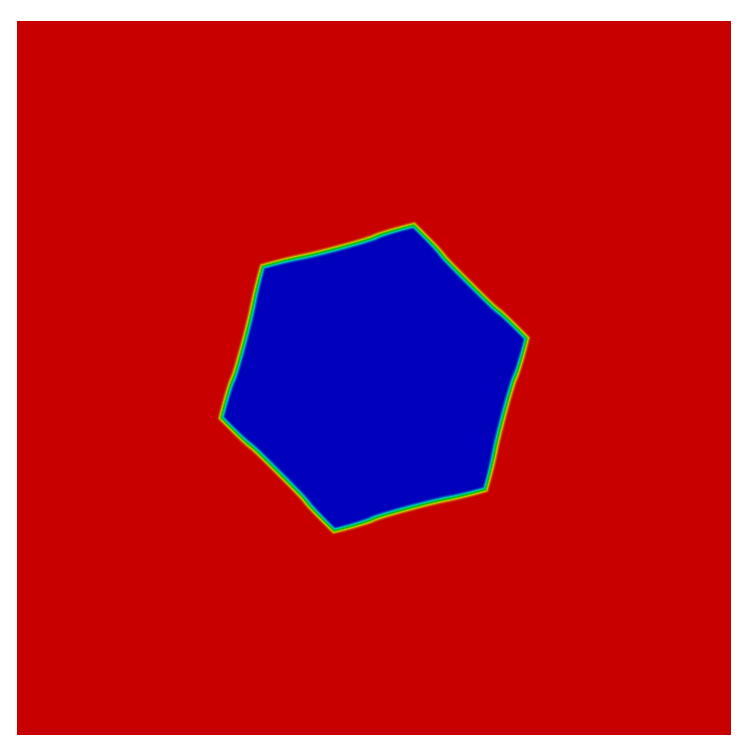}
\includegraphics[angle=-0,width=0.19\textwidth]{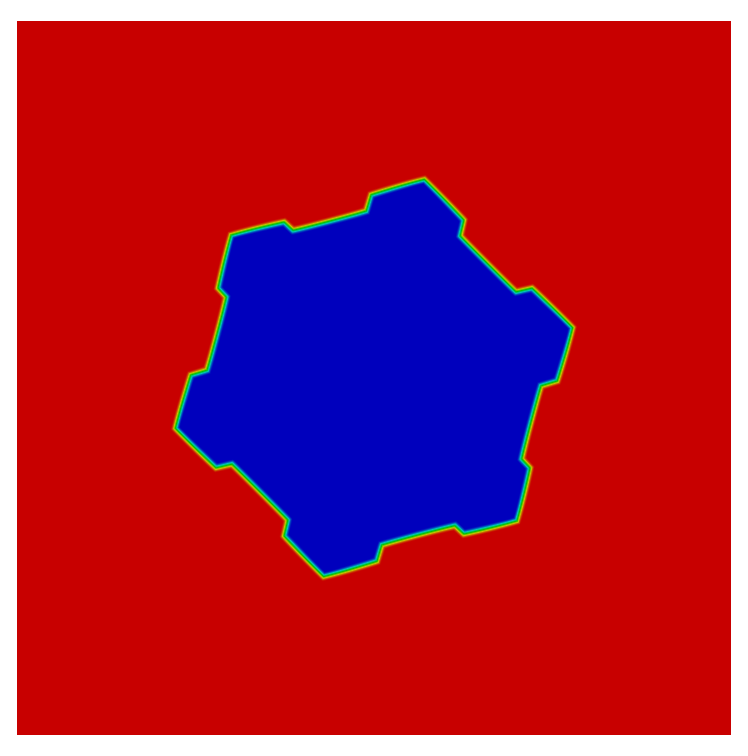}
\fi
\caption{($\epsilon^{-1} = 4\,\pi$, {\sc ani$_3^\star$}, (\ref{eq:varrho})(ii),
$\alpha=0.03$, $\rho = 0.01$, $\uD = -2$, $\Omega=(-16,16)^2$)
Snapshots of the solution at times $t=0,\,1,\,5,\,6,\,8$.
[This computation took $5$ hours, $17$ minutes.]
}
\label{fig:hexBSii_4pi_large}
\end{figure}%
We end this subsection with a repeat of the last computation, but now for the
stronger supercooling $\uD = -4$. The evolution now exhibits six distinct side
arms, as can be seen in Figure~\ref{fig:hexBSii_4pi_4large}.
\begin{figure}
\center
\ifpdf
\includegraphics[angle=-0,width=0.19\textwidth]{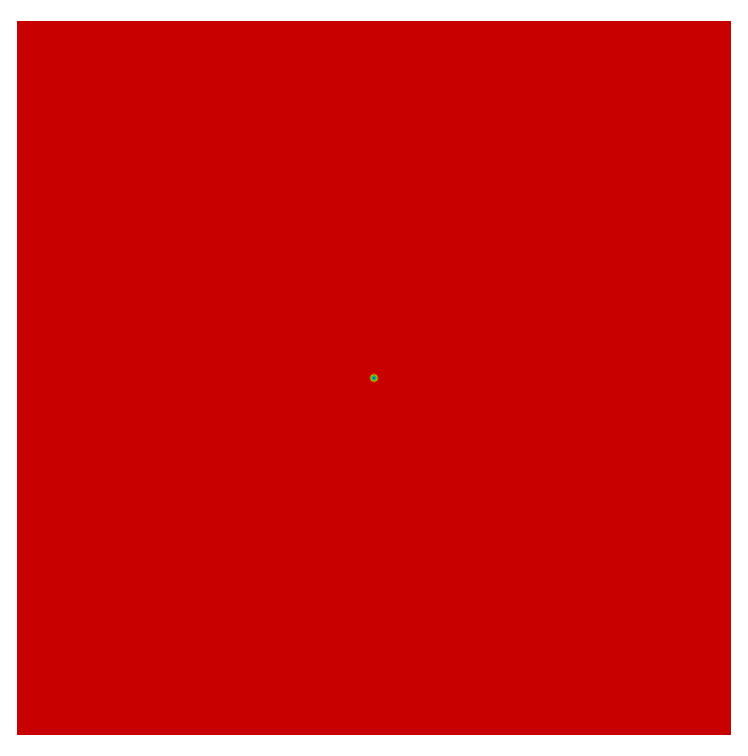}
\includegraphics[angle=-0,width=0.19\textwidth]{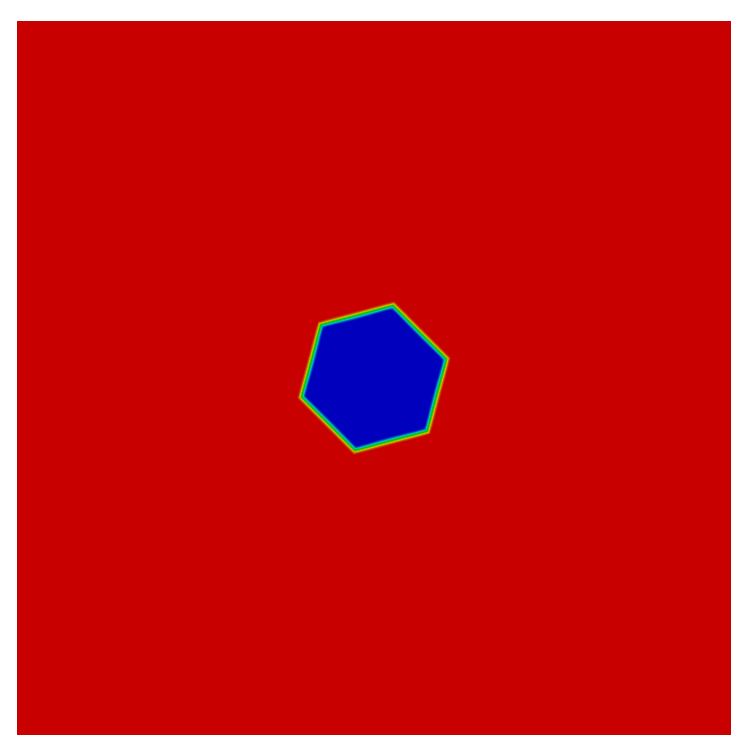}
\includegraphics[angle=-0,width=0.19\textwidth]{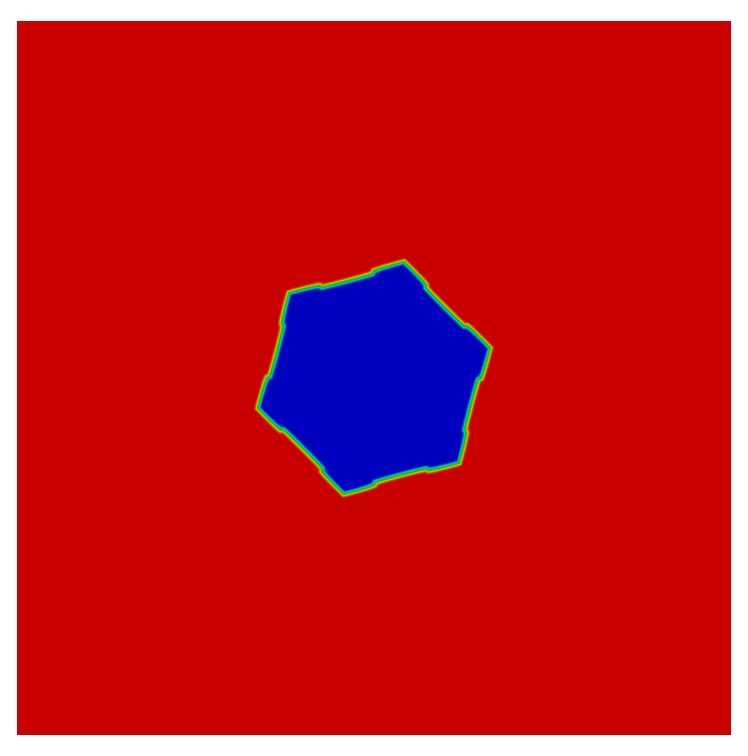}
\includegraphics[angle=-0,width=0.19\textwidth]{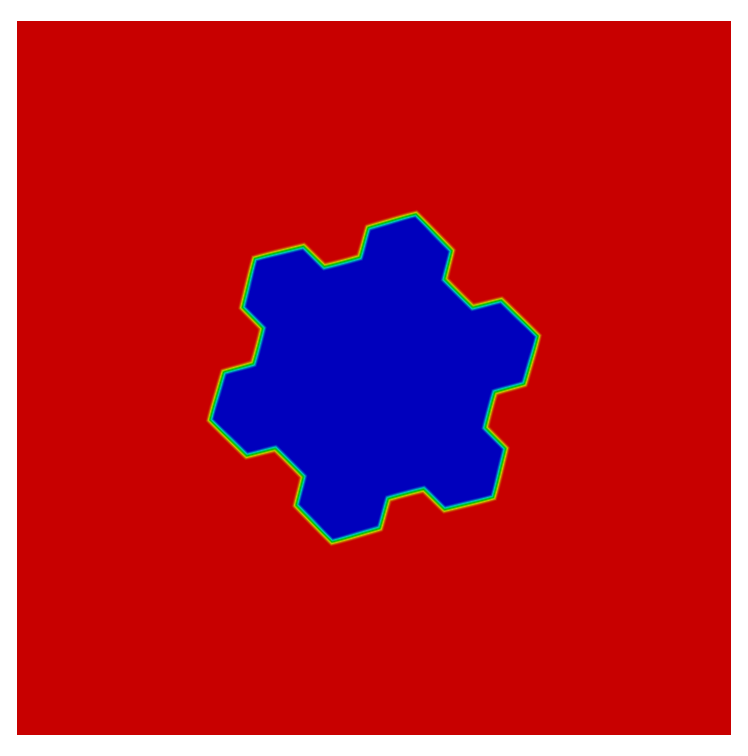}
\includegraphics[angle=-0,width=0.19\textwidth]{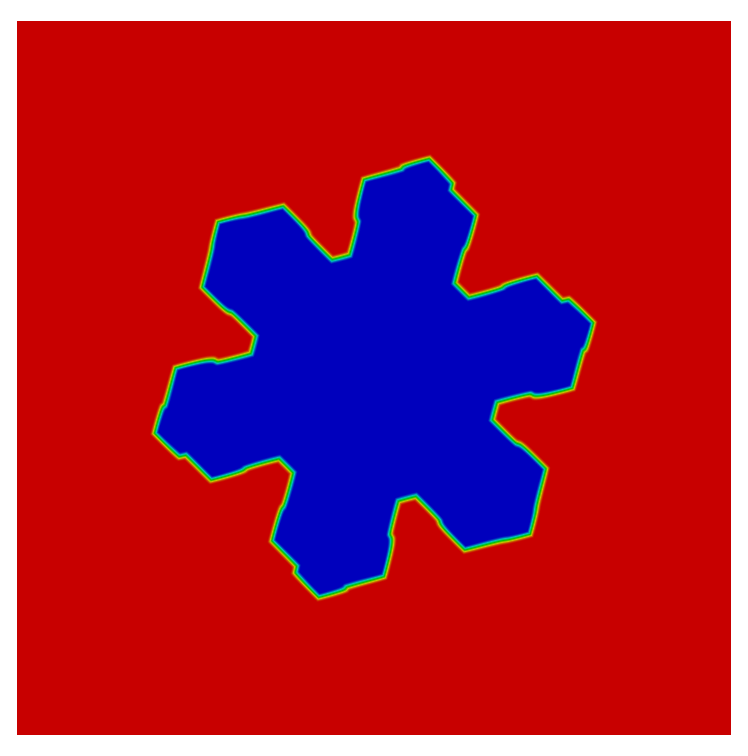}
\fi
\caption{($\epsilon^{-1} = 4\,\pi$, {\sc ani$_3^\star$}, (\ref{eq:varrho})(ii),
$\alpha=0.03$, $\rho = 0.01$, $\uD = -4$, $\Omega=(-16,16)^2$)
Snapshots of the solution at times $t=0,\,1,\,2,\,3,\,4$.
[This computation took $3$ hours, $5$ minutes.]
}
\label{fig:hexBSii_4pi_4large}
\end{figure}%

\subsection{Stefan problem in two space dimensions}

In a first simulation for the full Stefan problem, i.e.\ with $\vartheta > 0$, 
we take parameters that are close to the ones used in 
\citet[Fig.\ 10]{dendritic}. In particular, we have
$\vartheta=1$, $\alpha=5\times10^{-4}$, $\rho=0.01$, $\uD = -0.5$ and 
$R_0=0.2$, $H=8$.
An experiment with $\epsilon^{-1} =
4\,\pi$ together with $N_f = 512$, $N_c = 64$, $\tau = 10^{-3}$
and $T=1$ is shown in Figure~\ref{fig:Stefanii_4pi},
where we employ (\ref{eq:varrho})(ii).
We observe a very large interfacial region, which indicates that $\epsilon$
was not chosen small enough. 
\begin{figure}
\center
\ifpdf
\includegraphics[angle=-0,width=0.19\textwidth]{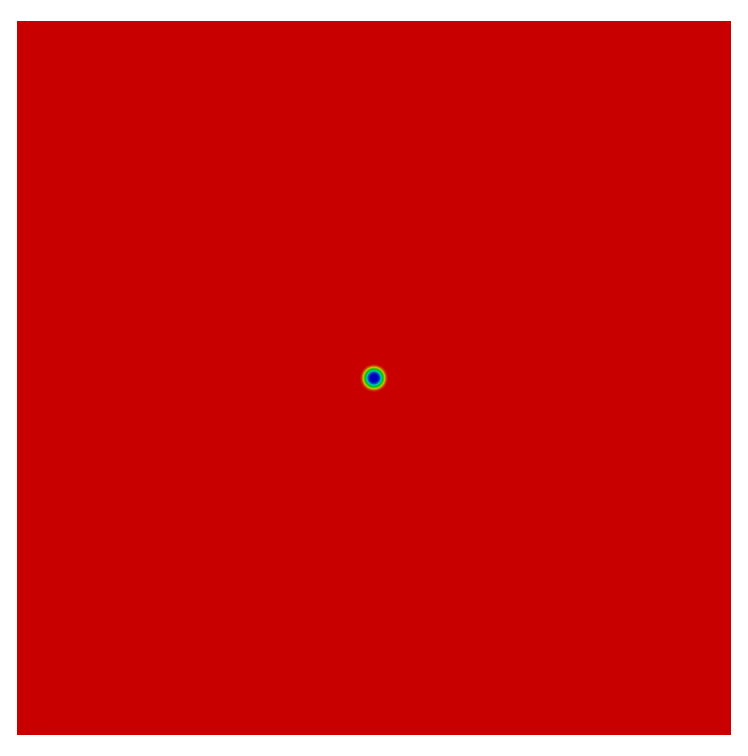}
\includegraphics[angle=-0,width=0.19\textwidth]{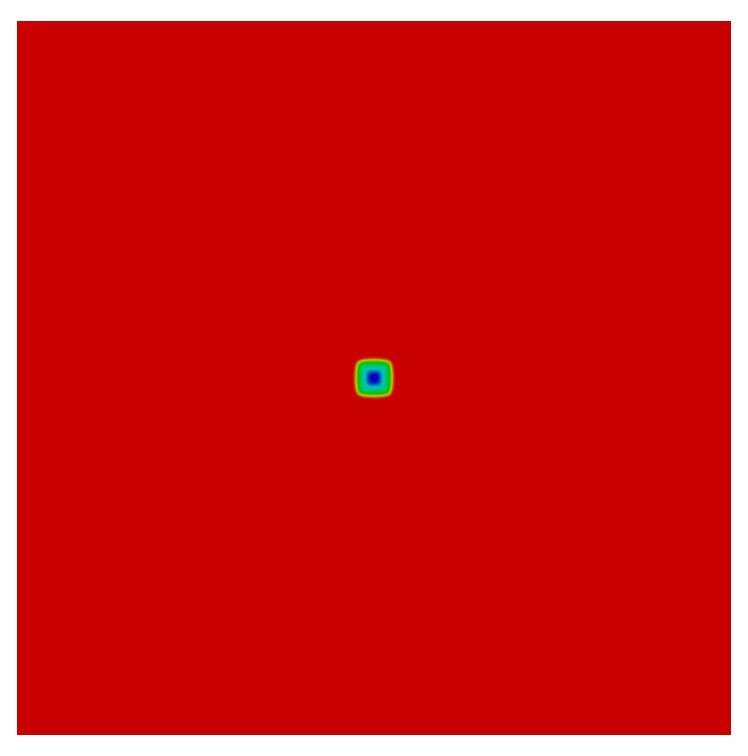}
\includegraphics[angle=-0,width=0.19\textwidth]{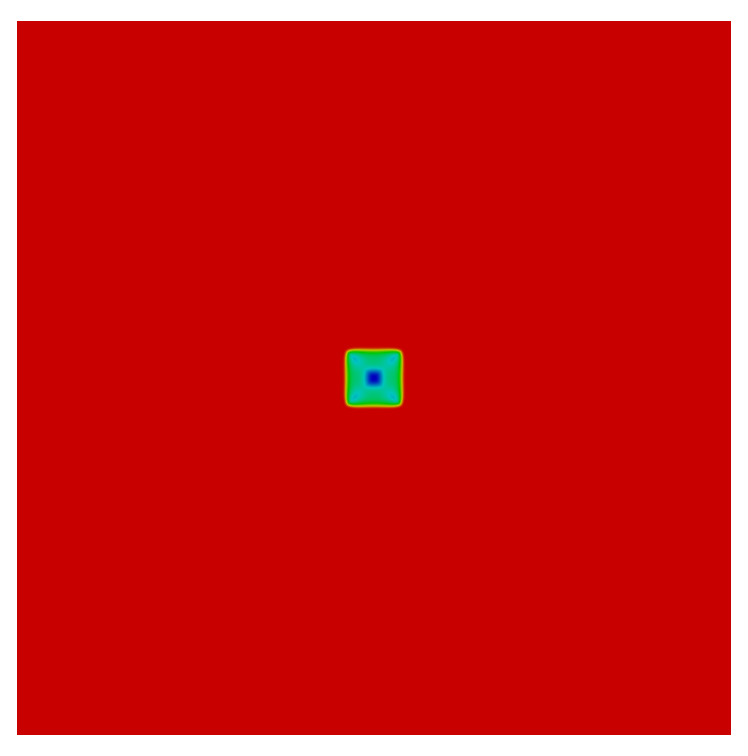}
\includegraphics[angle=-0,width=0.19\textwidth]{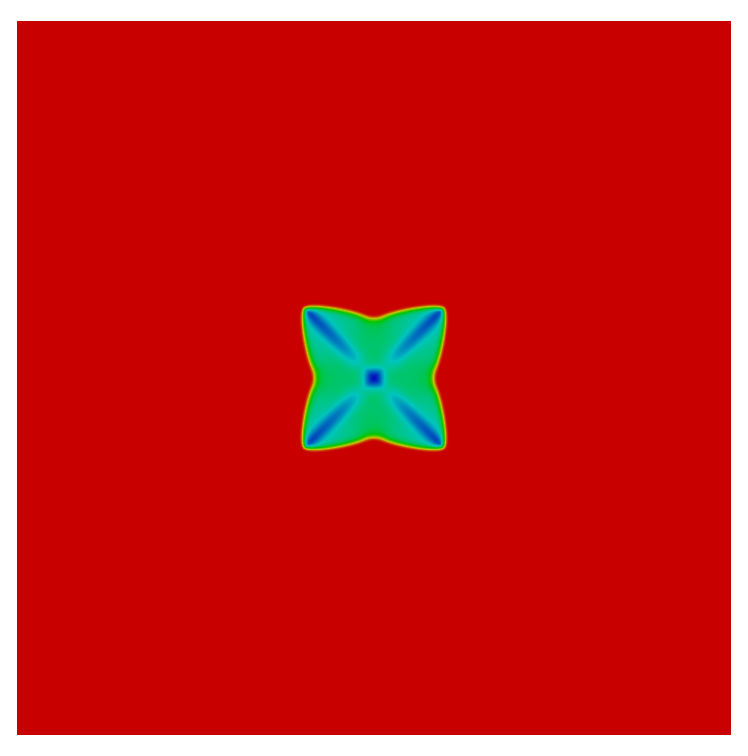}
\includegraphics[angle=-0,width=0.19\textwidth]{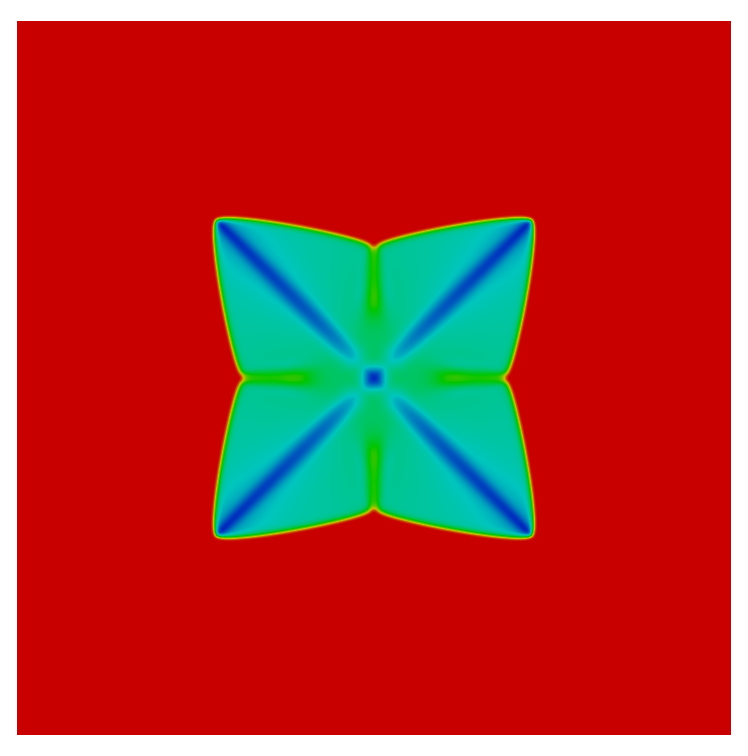}
\fi
\caption{($\epsilon^{-1} = 4\,\pi$, {\sc ani$_1^{(0.3)}$}, 
(\ref{eq:varrho})(ii), $\alpha=5\times10^{-4}$, $\rho=0.01$, $\uD = -0.5$, 
$\Omega=(-8,8)^2$)
Snapshots of the solution at times $t=0,\,0.1,\,0.2,\,0.5,\,1$.
[This computation took $76$ minutes.]
}
\label{fig:Stefanii_4pi}
\end{figure}%
A similar behaviour can be observed for the choice (\ref{eq:varrho})(iii),
see Figure~\ref{fig:Stefaniv_4pi}.
Here we note that in this example, %, in contrast to Remark~\ref{rem:varrho},
in line with the analysis in (\ref{eq:obstcond}), 
there appears to be no benefit in using (\ref{eq:varrho})(iii) over
(\ref{eq:varrho})(ii). 
\begin{figure}
\center
\ifpdf
\includegraphics[angle=-0,width=0.19\textwidth]{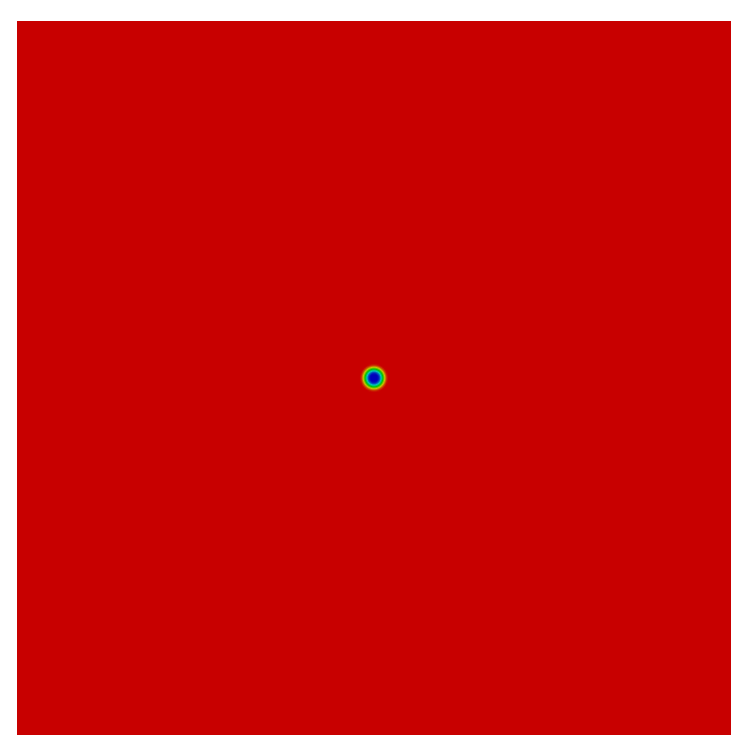}
\includegraphics[angle=-0,width=0.19\textwidth]{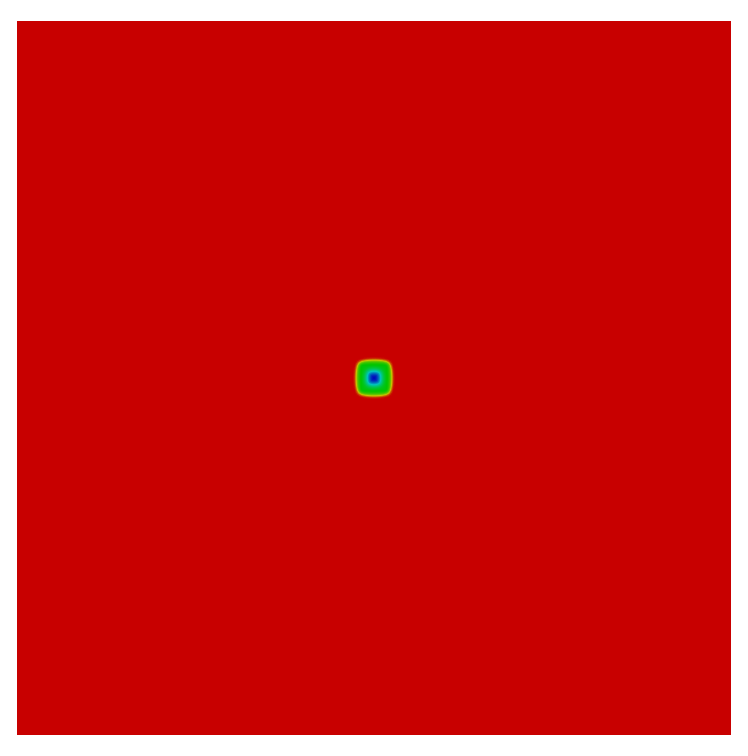}
\includegraphics[angle=-0,width=0.19\textwidth]{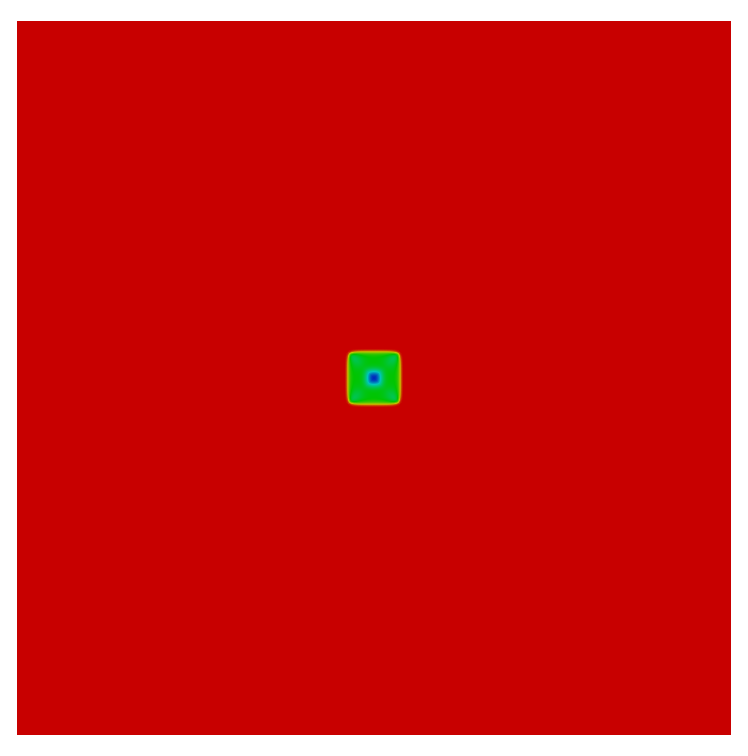}
\includegraphics[angle=-0,width=0.19\textwidth]{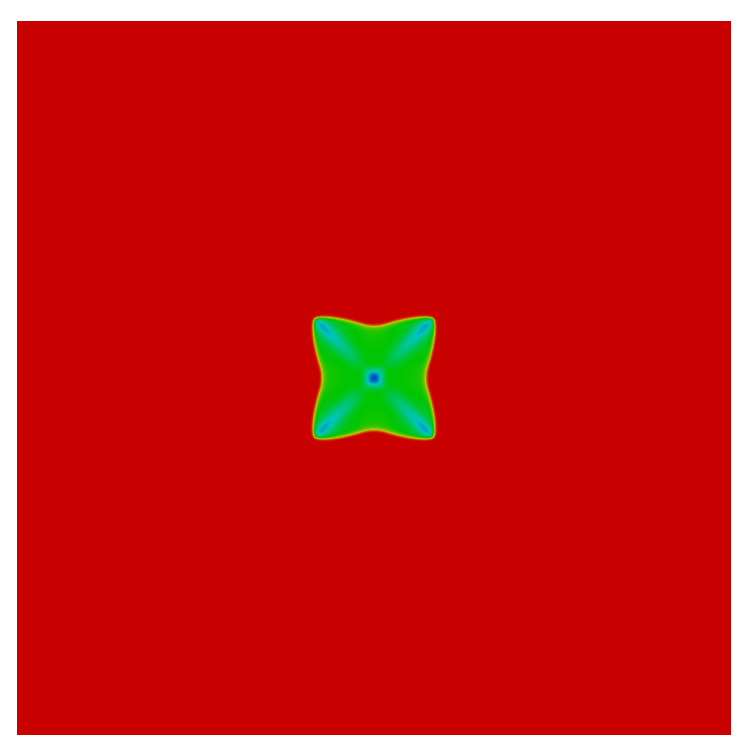}
\includegraphics[angle=-0,width=0.19\textwidth]{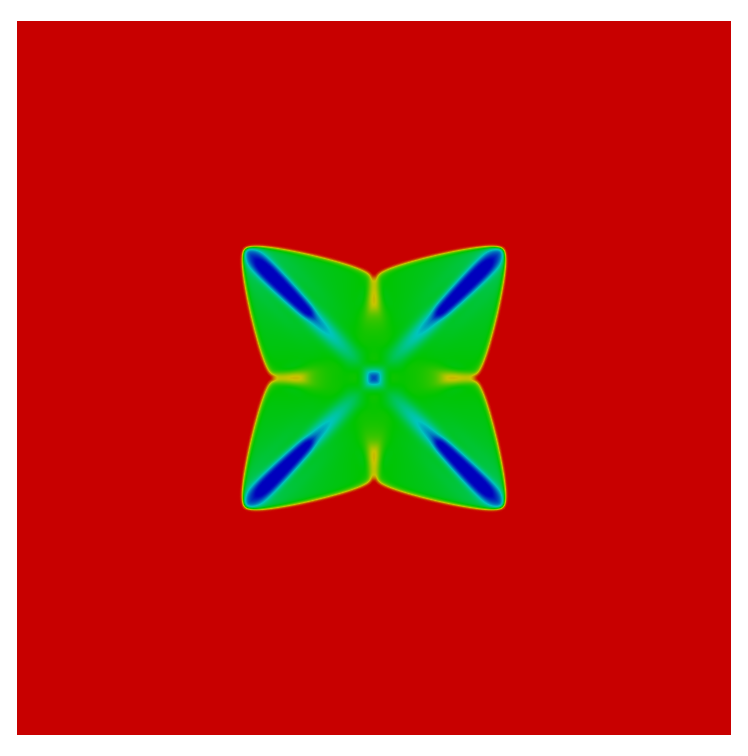}
\fi
\caption{($\epsilon^{-1} = 4\,\pi$, {\sc ani$_1^{(0.3)}$}, 
(\ref{eq:varrho})(iii), $\alpha=5\times10^{-4}$, $\rho=0.01$, $\uD = -0.5$, 
$\Omega=(-8,8)^2$)
Snapshots of the solution at times $t=0,\,0.1,\,0.2,\,0.5,\,1$.
[This computation took $165$ minutes.]
}
\label{fig:Stefaniv_4pi}
\end{figure}%
On reducing the size of the interfacial parameter $\epsilon$, the
phase field again assumes its expected profile across the interface, and we
obtain the following numerical results. 
If $\epsilon^{-1} = 16\,\pi$ together with $N_f = 2048$, $N_c = 128$, 
$\tau = 10^{-4}$ and $T=2$ we obtain the results shown in
Figure~\ref{fig:Stefanii_16pi}.
\begin{figure}
\center
\ifpdf
\includegraphics[angle=-0,width=0.19\textwidth]{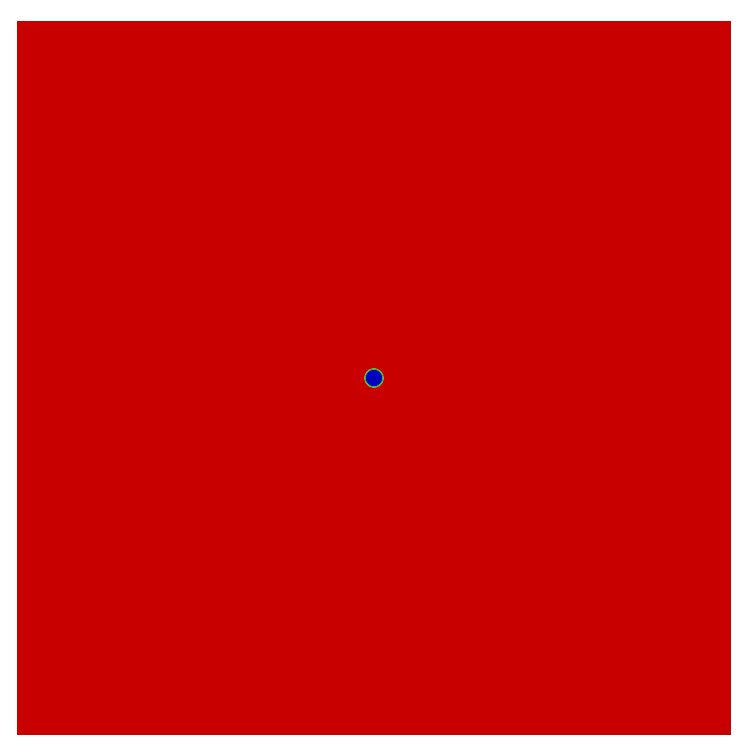}
\includegraphics[angle=-0,width=0.19\textwidth]{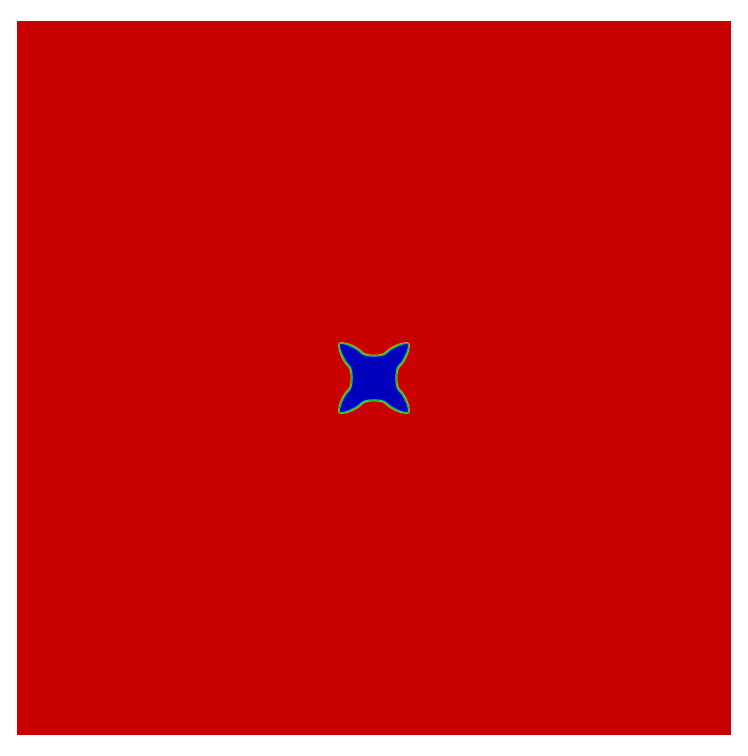}
\includegraphics[angle=-0,width=0.19\textwidth]{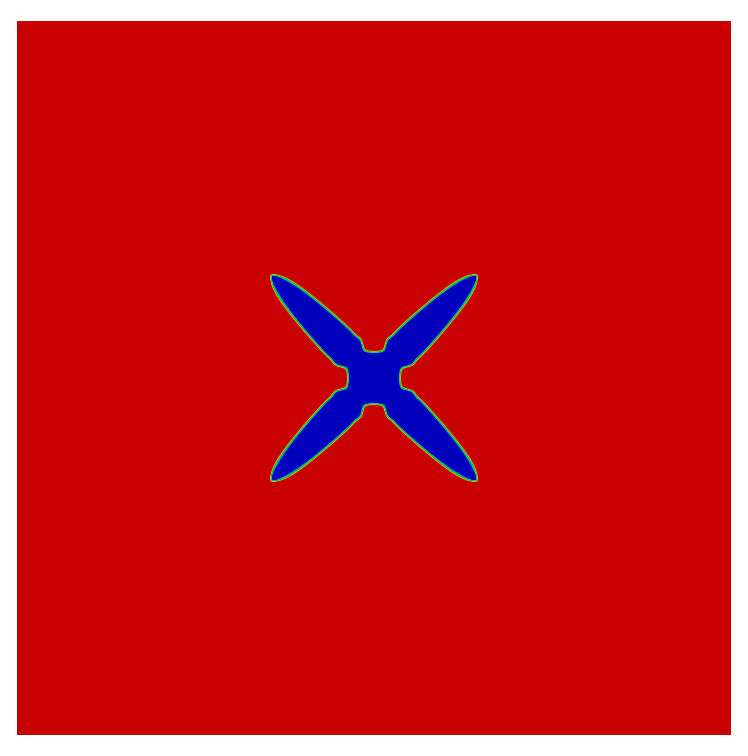}
\includegraphics[angle=-0,width=0.19\textwidth]{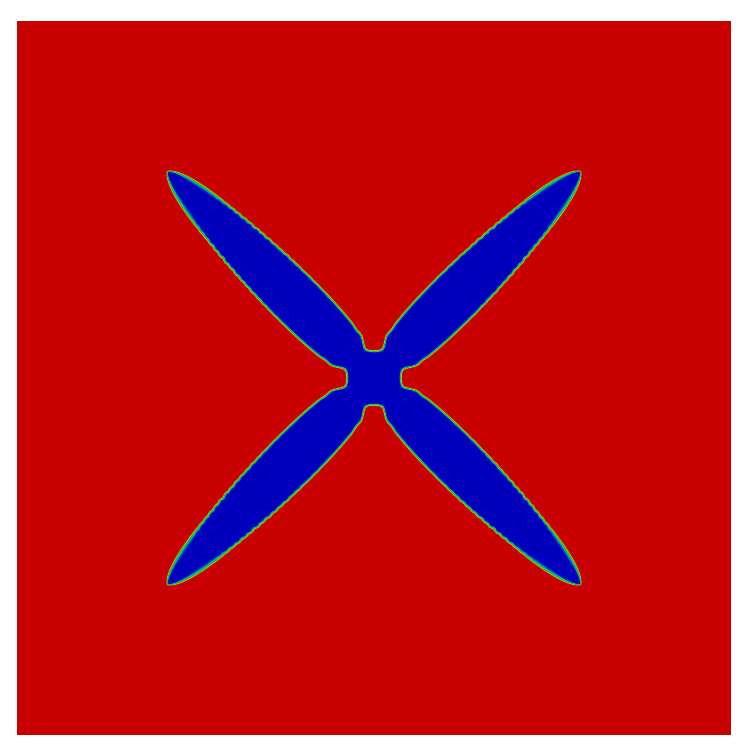}
\includegraphics[angle=-0,width=0.19\textwidth]{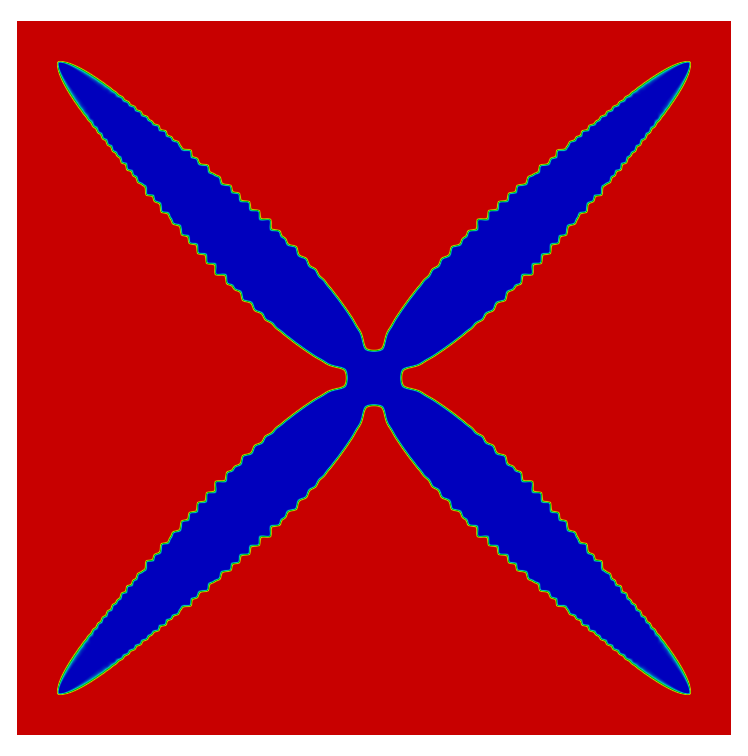}
\fi
\caption{($\epsilon^{-1} = 16\,\pi$, {\sc ani$_1^{(0.3)}$}, 
(\ref{eq:varrho})(ii), $\alpha=5\times10^{-4}$, $\rho=0.01$, $\uD = -0.5$, 
$\Omega=(-8,8)^2$)
Snapshots of the solution at times $t=0,\,0.5,\,1,\,1.5,\,2$.
[This computation took $15.5$ hours.]
}
\label{fig:Stefanii_16pi}
\end{figure}%
If $\epsilon^{-1} = 32\,\pi$ together with $N_f = 4096$, $N_c = 128$, 
$\tau = 10^{-4}$ and $T=2$ we obtain the results shown in
Figure~\ref{fig:Stefanii_32pi}. We can see that the % magnitude of the 
small oscillations present
in the final snapshot in Figure~\ref{fig:Stefanii_16pi} have vanished in the
corresponding plot in Figure~\ref{fig:Stefanii_32pi}. 
\begin{figure}
\center
\ifpdf
\includegraphics[angle=-0,width=0.19\textwidth]{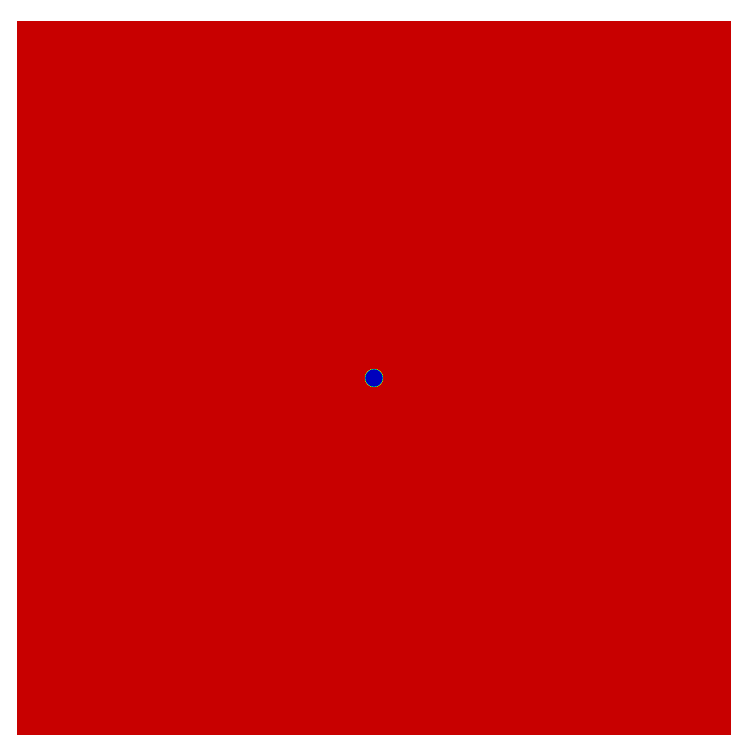}
\includegraphics[angle=-0,width=0.19\textwidth]{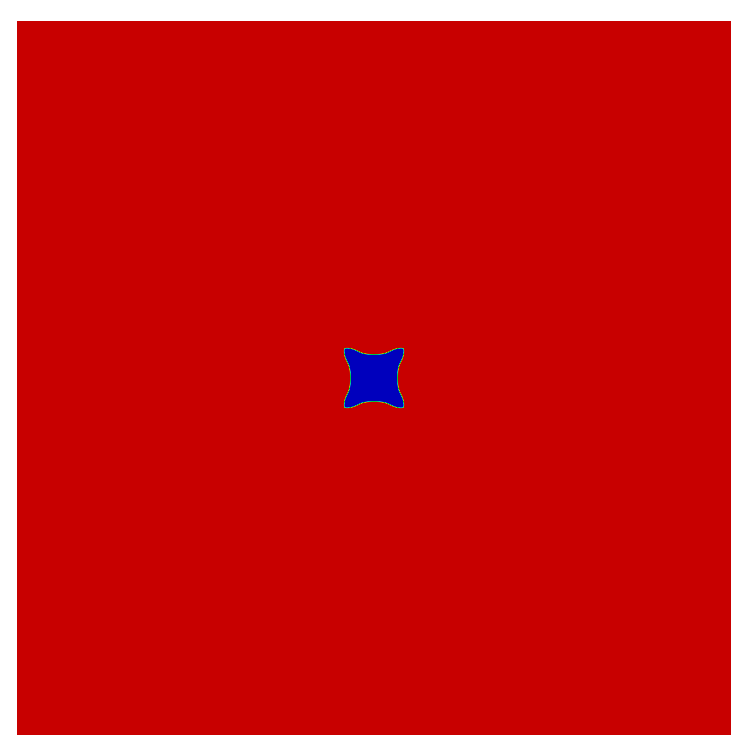}
\includegraphics[angle=-0,width=0.19\textwidth]{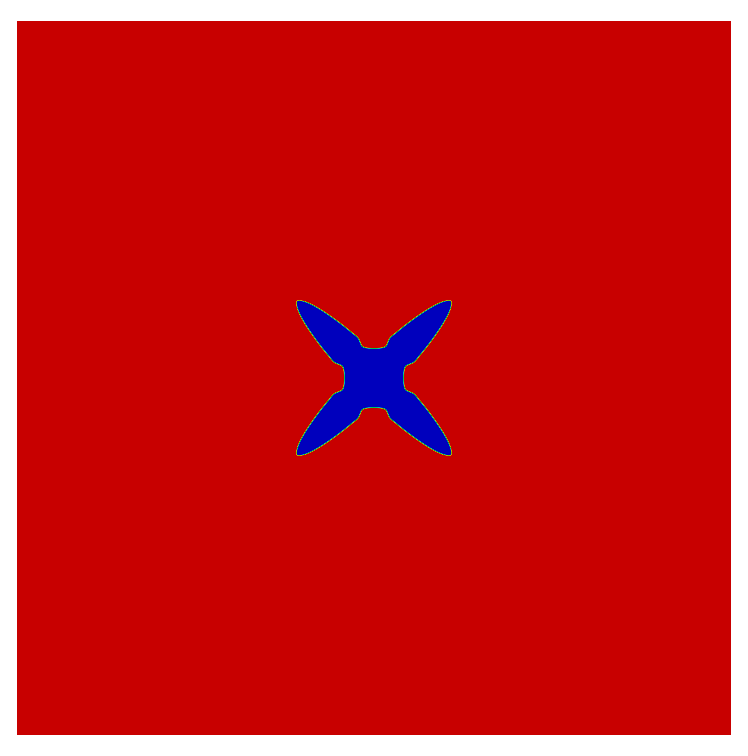}
\includegraphics[angle=-0,width=0.19\textwidth]{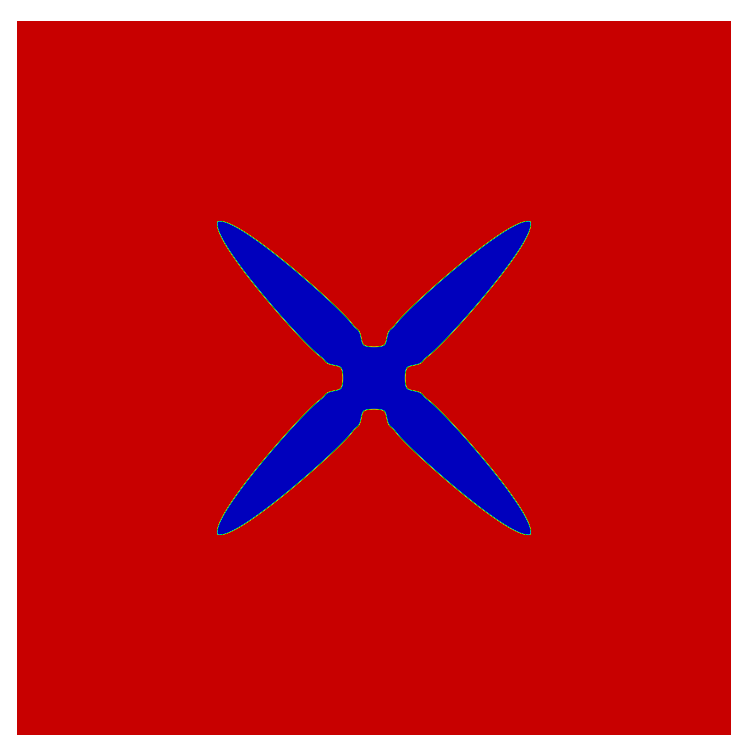}
\includegraphics[angle=-0,width=0.19\textwidth]{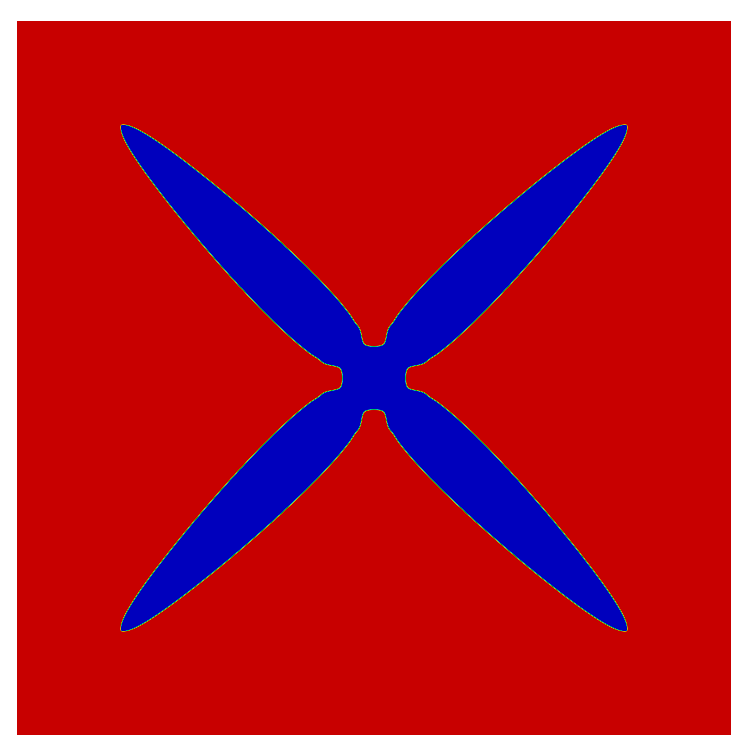}
\fi
\caption{($\epsilon^{-1} = 32\,\pi$, {\sc ani$_1^{(0.3)}$}, 
(\ref{eq:varrho})(ii), $\alpha=5\times10^{-4}$, $\rho=0.01$, $\uD = -0.5$, 
$\Omega=(-8,8)^2$)
Snapshots of the solution at times $t=0,\,0.5,\,1,\,1.5,\,2$.
[This computation took $19.5$ hours.]
}
\label{fig:Stefanii_32pi}
\end{figure}%
A closer comparison of the two solutions at time $t=2$ can be seen in
Figure~\ref{fig:Stefanii_16pi_detail}.
\begin{figure}
\center
\ifpdf
\includegraphics[angle=-0,width=0.4\textwidth]{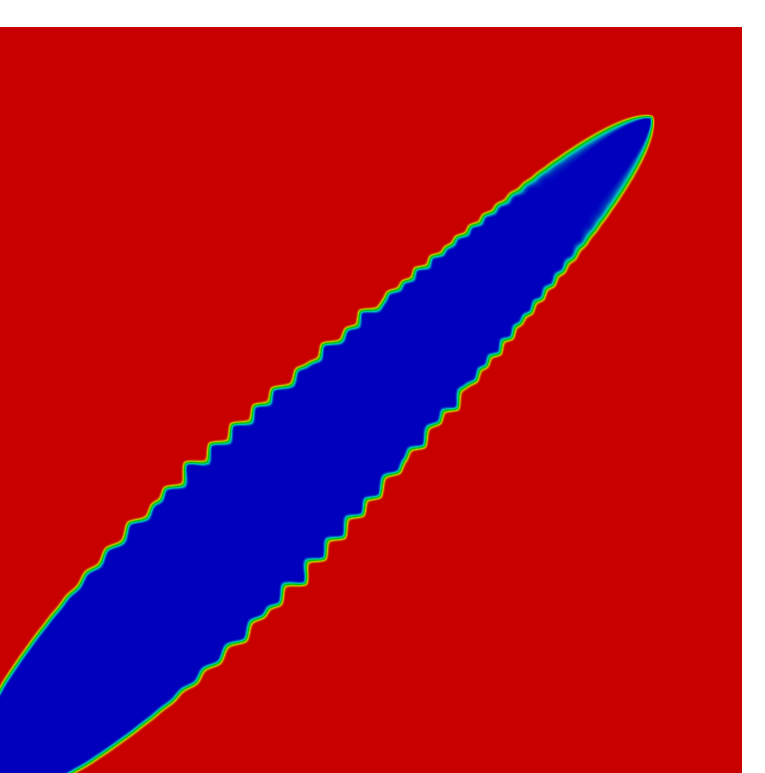}
\qquad
\includegraphics[angle=-0,width=0.4\textwidth]{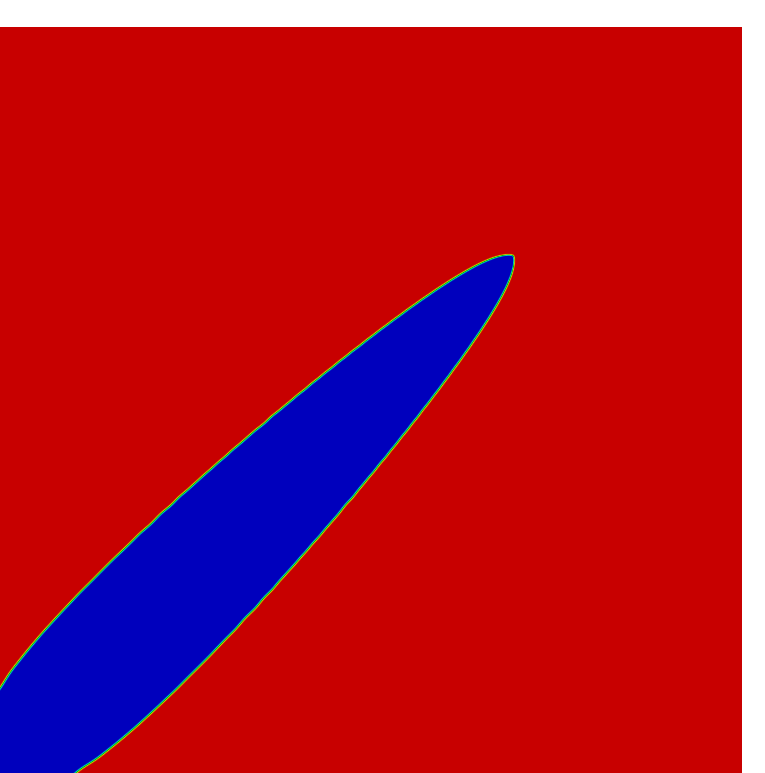}
\fi
\caption{A close comparison of the final solutions from
Figures~\ref{fig:Stefanii_16pi} and \ref{fig:Stefanii_32pi}. On the left 
$\epsilon^{-1} = 16\,\pi$, and on the right $\epsilon^{-1} = 32\,\pi$.
}
\label{fig:Stefanii_16pi_detail}
\end{figure}%

The next simulation is for the rotated hexagonal anisotropy 
{\sc ani$_3^\star$}. All the remaining parameters are as in
Figure~\ref{fig:Stefanii_16pi}, i.e.\
$\epsilon^{-1} = 16\,\pi$ together with $N_f = 2048$, $N_c = 128$, 
$\tau = 10^{-4}$ and $T=2$. See Figure~\ref{fig:hexStefanii_16pi} for the
numerical results. 
\begin{figure}
\center
\ifpdf
\includegraphics[angle=-0,width=0.19\textwidth]{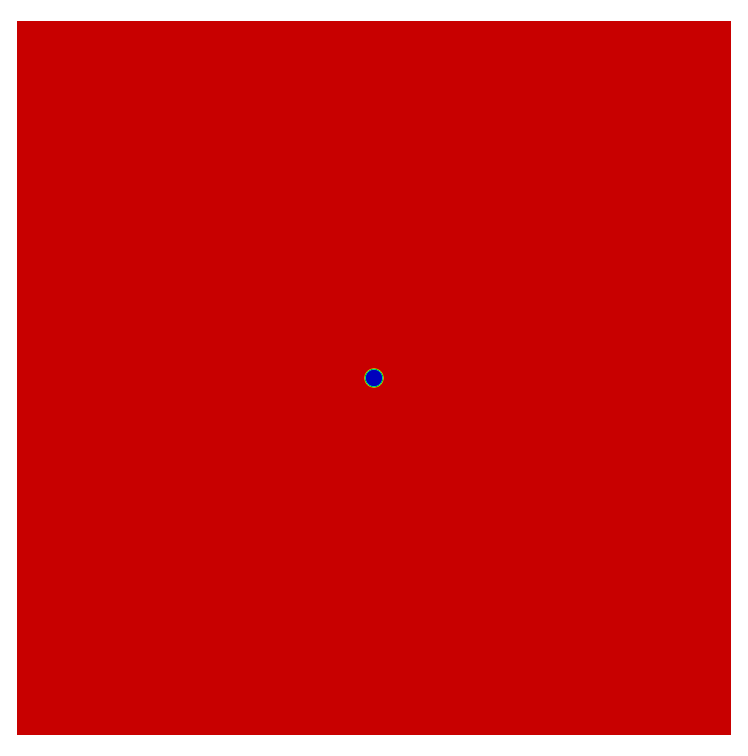}
\includegraphics[angle=-0,width=0.19\textwidth]{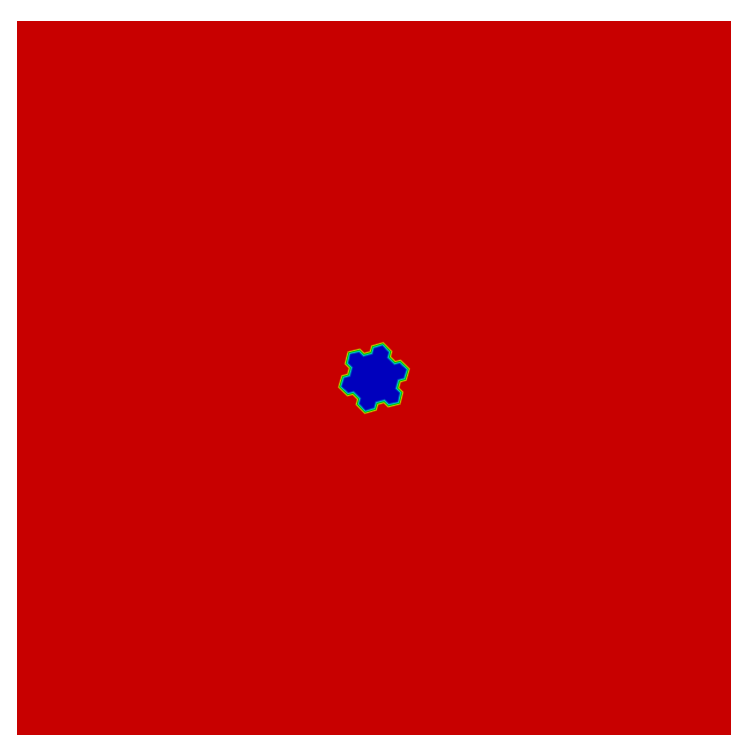}
\includegraphics[angle=-0,width=0.19\textwidth]{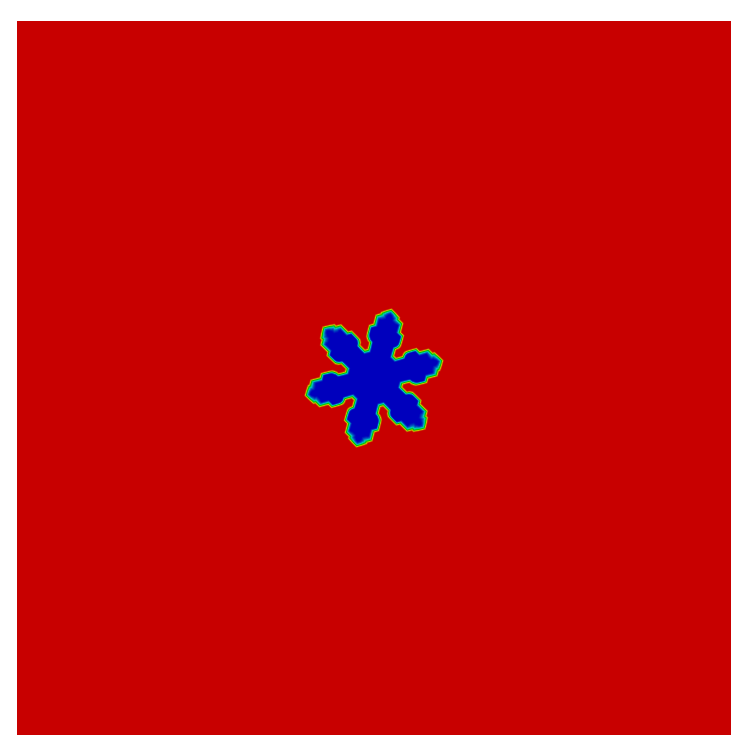}
\includegraphics[angle=-0,width=0.19\textwidth]{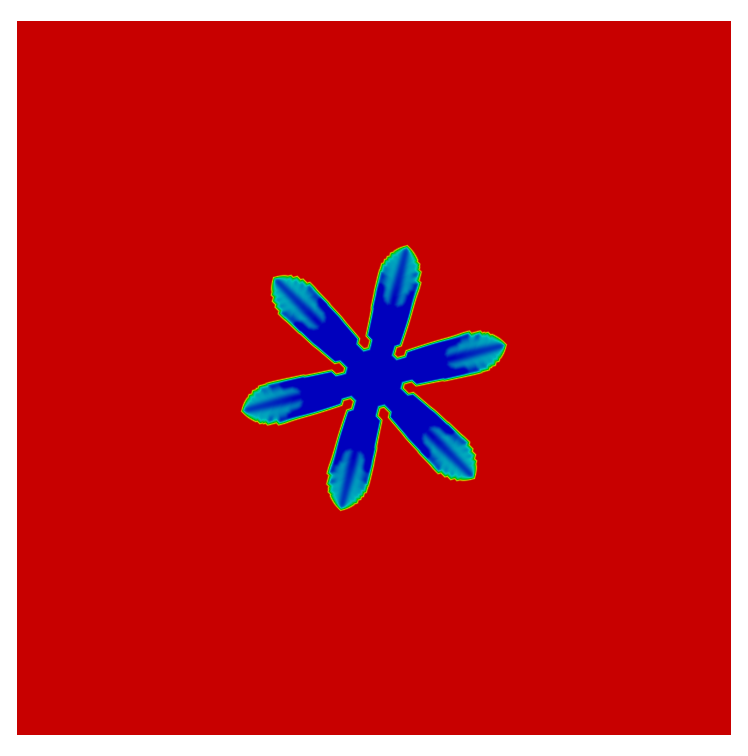}
\includegraphics[angle=-0,width=0.19\textwidth]{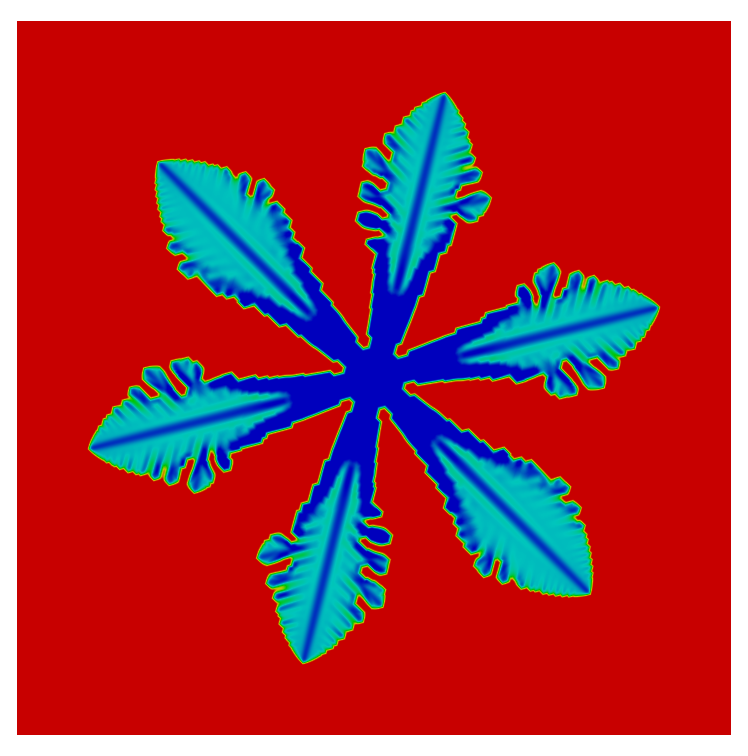}
\fi
\caption{($\epsilon^{-1} = 16\,\pi$, {\sc ani$_3^\star$}, 
(\ref{eq:varrho})(ii), $\alpha=5\times10^{-4}$, $\rho=0.01$, $\uD = -0.5$, 
$\Omega=(-8,8)^2$)
Snapshots of the solution at times $t=0,\,0.5,\,1,\,1.5,\,2$.
[This computation took $51$ hours.]
}
\label{fig:hexStefanii_16pi}
\end{figure}%
The large mushy regions in the final plot in Figure~\ref{fig:hexStefanii_16pi}
indicate once again that $\epsilon$ needs to be chosen smaller. 
Hence we repeat this experiment and now choose $\epsilon^{-1} = 32\,\pi$ 
together with $N_f = 4096$, $N_c = 128$, $\tau = 10^{-4}$ and $T=2$. 
See Figure~\ref{fig:hexStefanii_32pi} for the numerical results. We observe
that the interfacial region is now well defined and that the evolution 
exhibits six distinct side arms.
\begin{figure}
\center
\ifpdf
\includegraphics[angle=-0,width=0.19\textwidth]{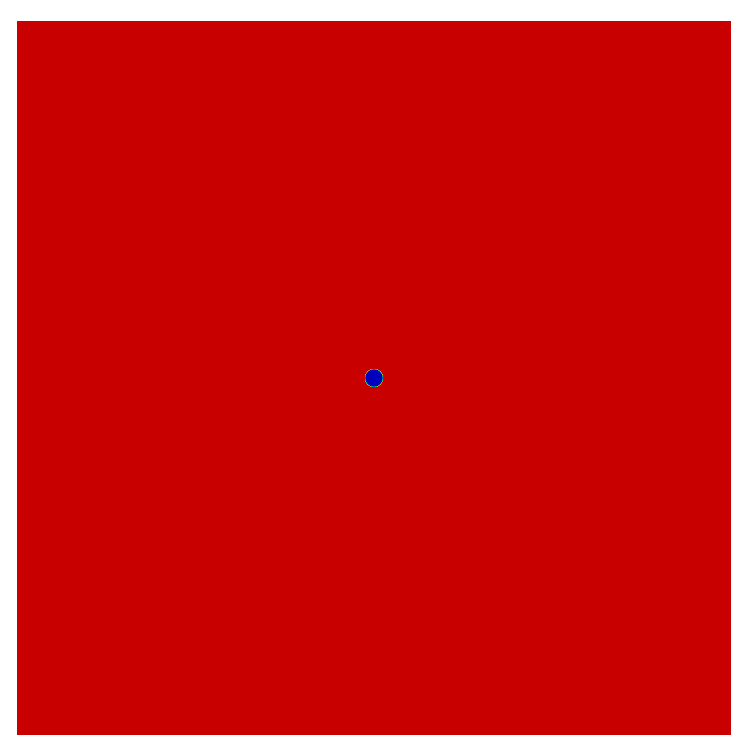}
\includegraphics[angle=-0,width=0.19\textwidth]{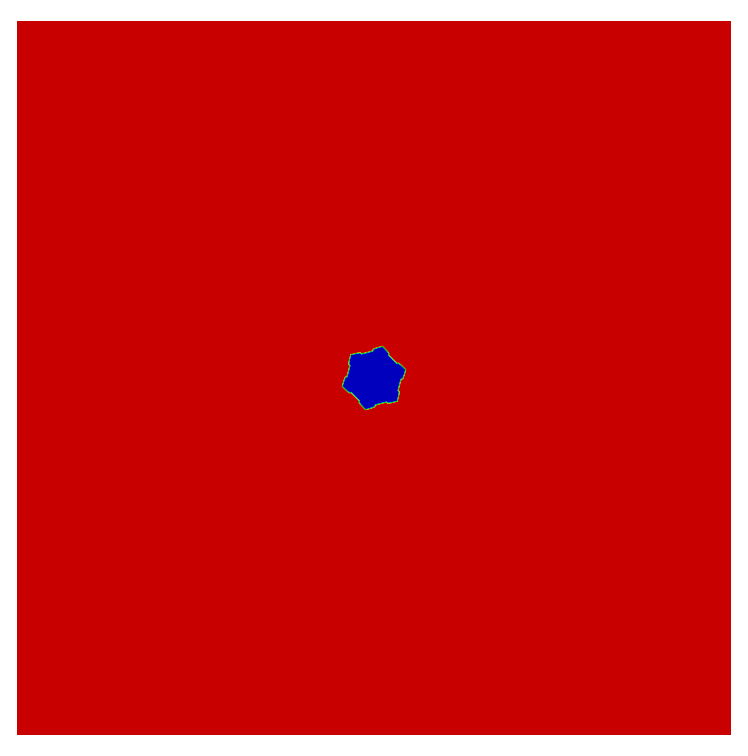}
\includegraphics[angle=-0,width=0.19\textwidth]{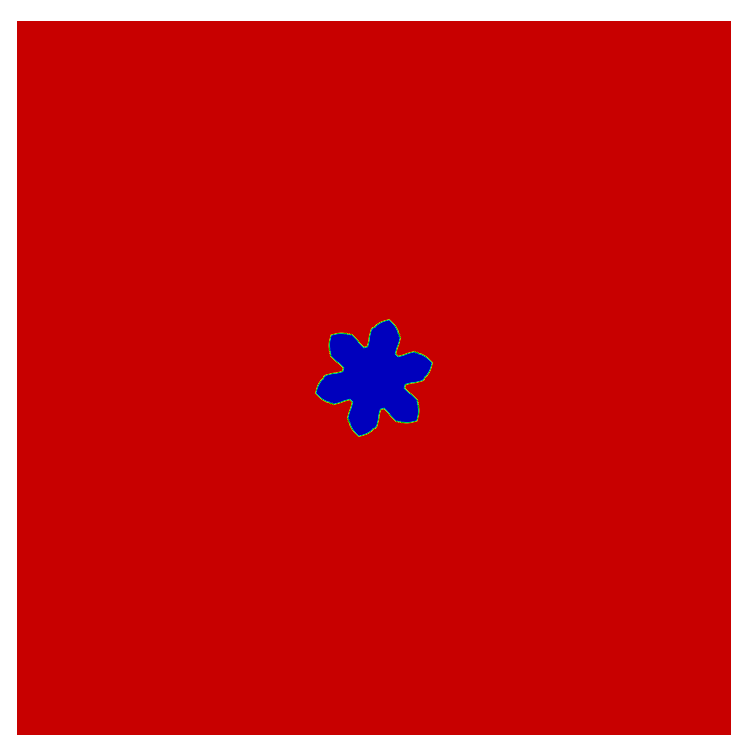}
\includegraphics[angle=-0,width=0.19\textwidth]{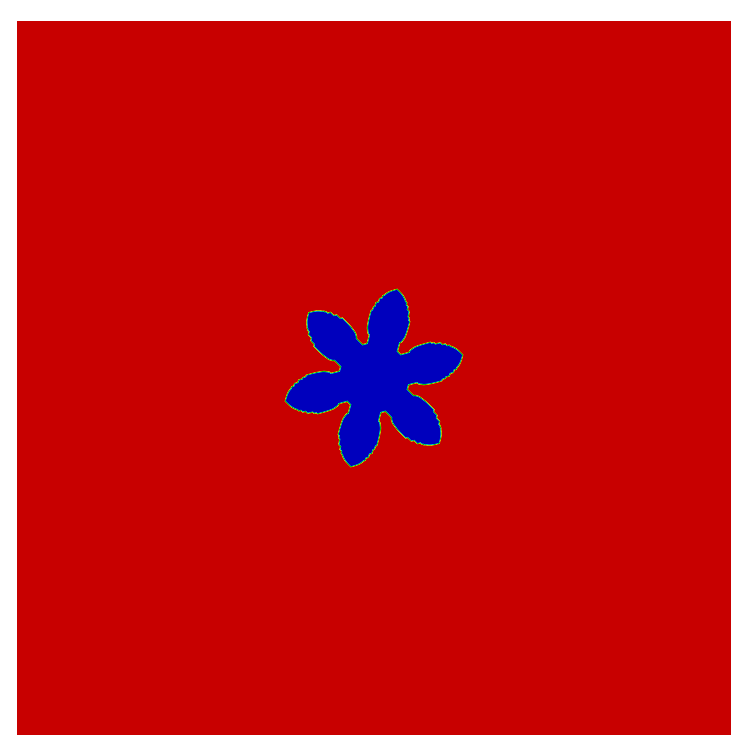}
\includegraphics[angle=-0,width=0.19\textwidth]{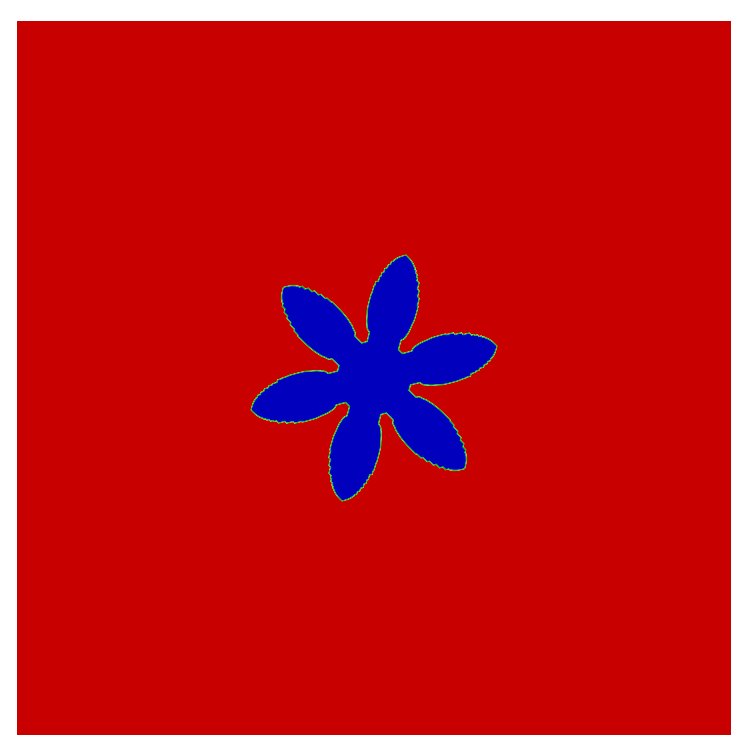}
\fi
\caption{($\epsilon^{-1} = 32\,\pi$, {\sc ani$_3^\star$}, 
(\ref{eq:varrho})(ii), $\alpha=5\times10^{-4}$, $\rho=0.01$, $\uD = -0.5$, 
$\Omega=(-8,8)^2$)
Snapshots of the solution at times $t=0,\,0.5,\,1,\,1.5,\,2$.
[This computation took $19$ hours.]
}
\label{fig:hexStefanii_32pi}
\end{figure}%

\subsection{Mullins--Sekerka in three space dimensions} \label{sec:53}
In this subsection we always employ (\ref{eq:varrho})(ii), and we always
let $\vartheta = 0$. At first we also choose
$\rho=0$, so that we approximate a Mullins--Sekerka problem without kinetic
undercooling.
A simulation for the cubic anisotropy {\sc ani$_9$} with 
$\epsilon^{-1} = 2\,\pi$, and for the physical parameters
$\alpha = 0.03$ and $\uD = -2$,
can be seen in Figure~\ref{fig:3dBSnewii_2pi}. 
The discretization parameters are
$N_f = 256$, $N_c = 64$, $\tau = 10^{-3}$ and $T=1$.
We observe that the cubic anisotropy induces the growth of the typical six
symmetric side arms.
\begin{figure}
\center
\ifpdf
\includegraphics[angle=-0,width=0.19\textwidth]{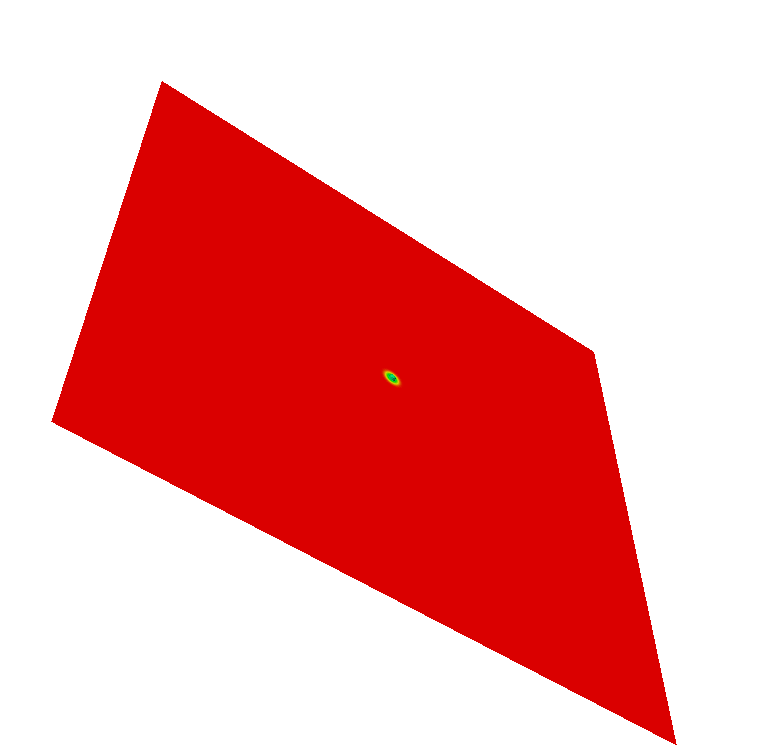}
\includegraphics[angle=-0,width=0.19\textwidth]{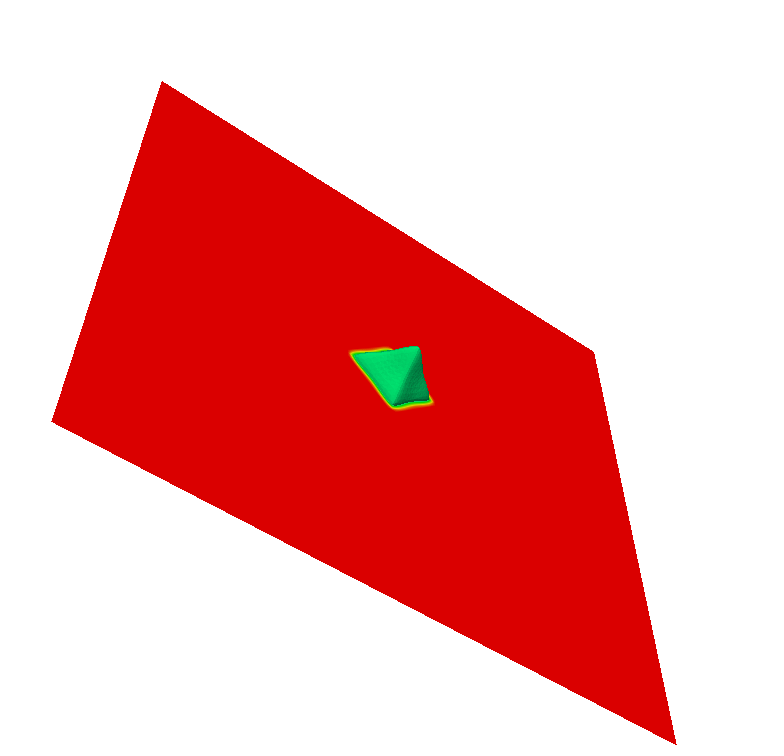} 
\includegraphics[angle=-0,width=0.19\textwidth]{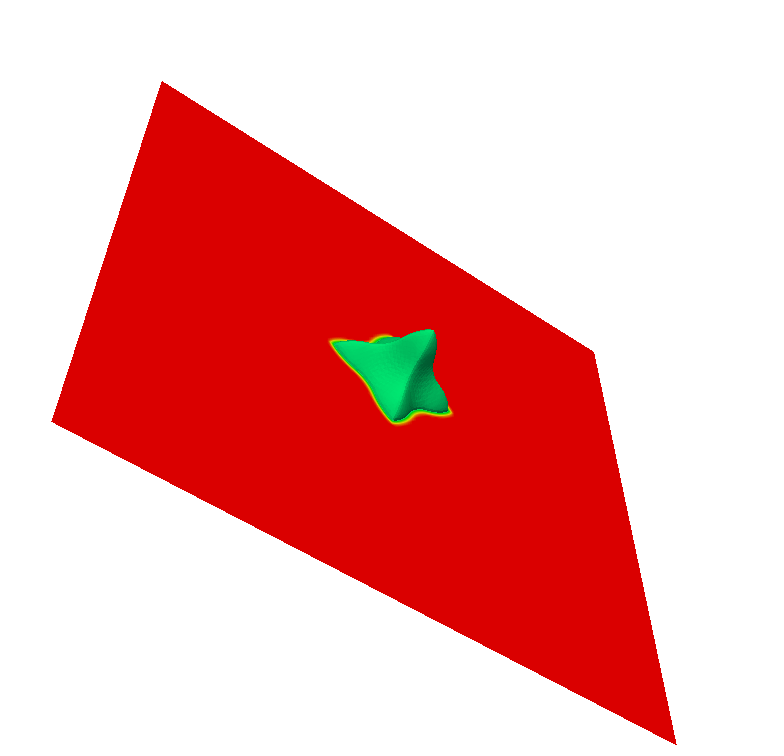}
\includegraphics[angle=-0,width=0.19\textwidth]{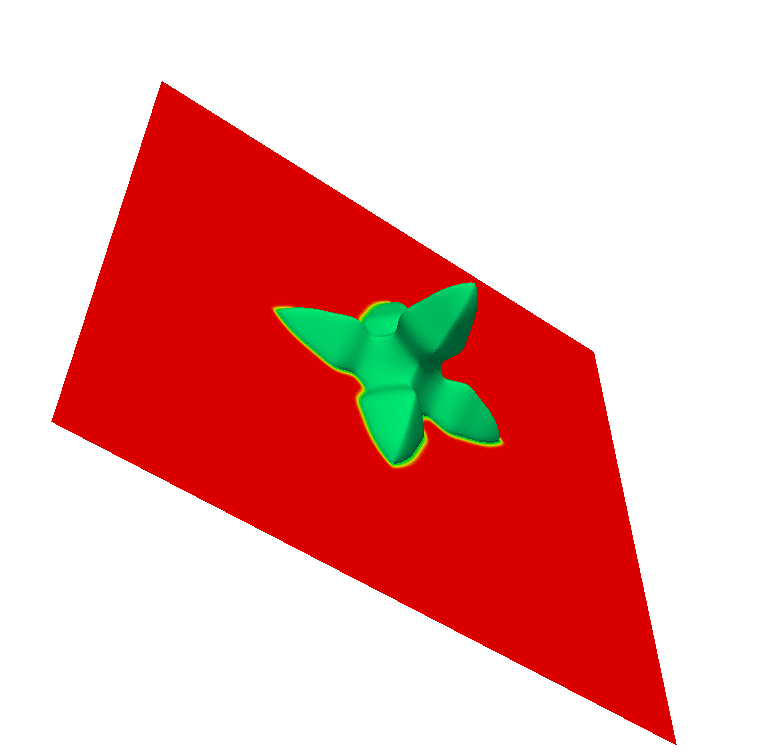}
\includegraphics[angle=-0,width=0.19\textwidth]{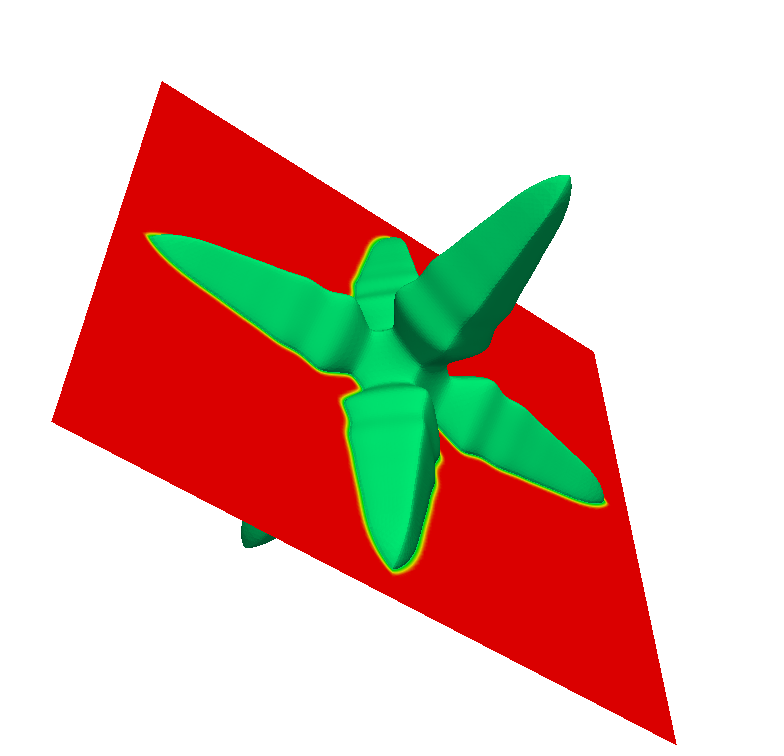}
\includegraphics[angle=-0,width=0.19\textwidth]{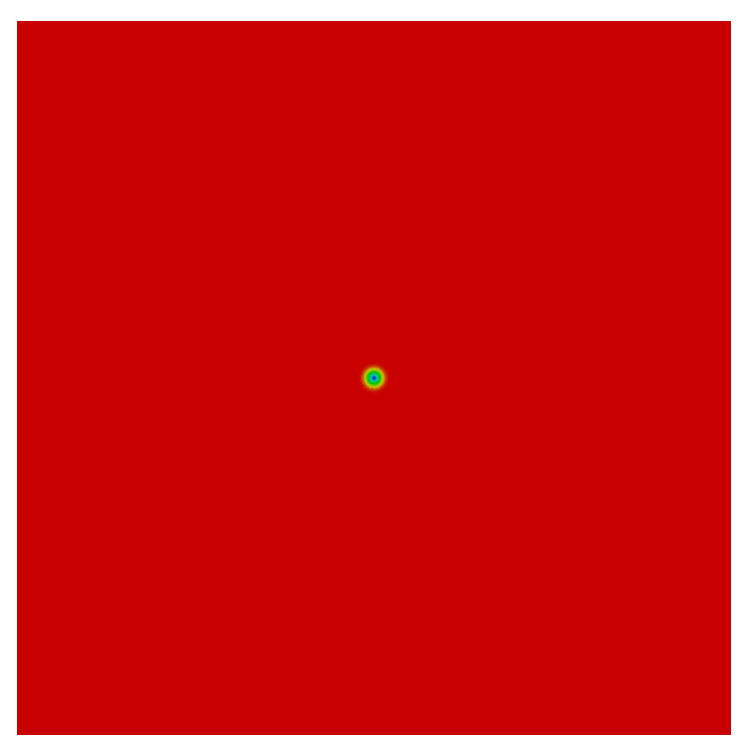}
\includegraphics[angle=-0,width=0.19\textwidth]{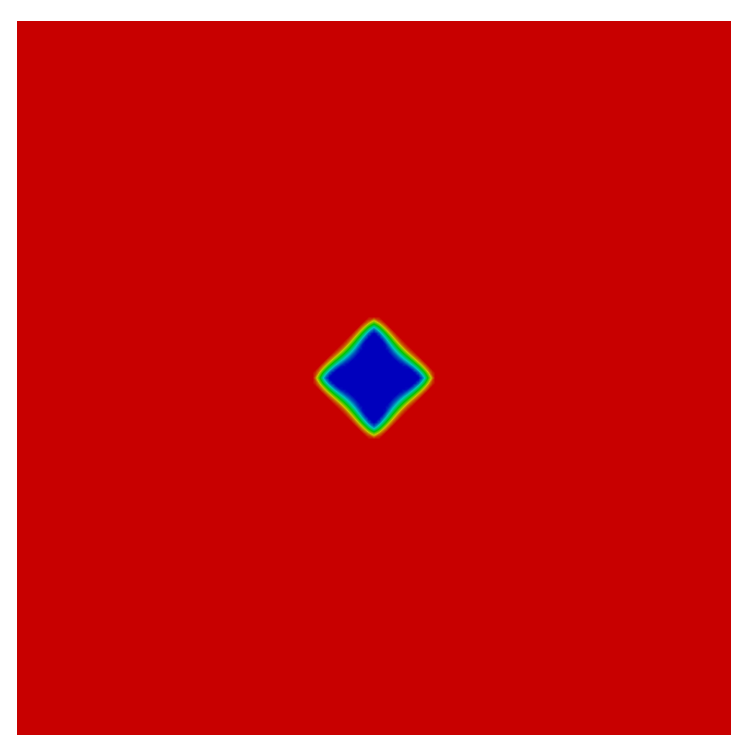}
\includegraphics[angle=-0,width=0.19\textwidth]{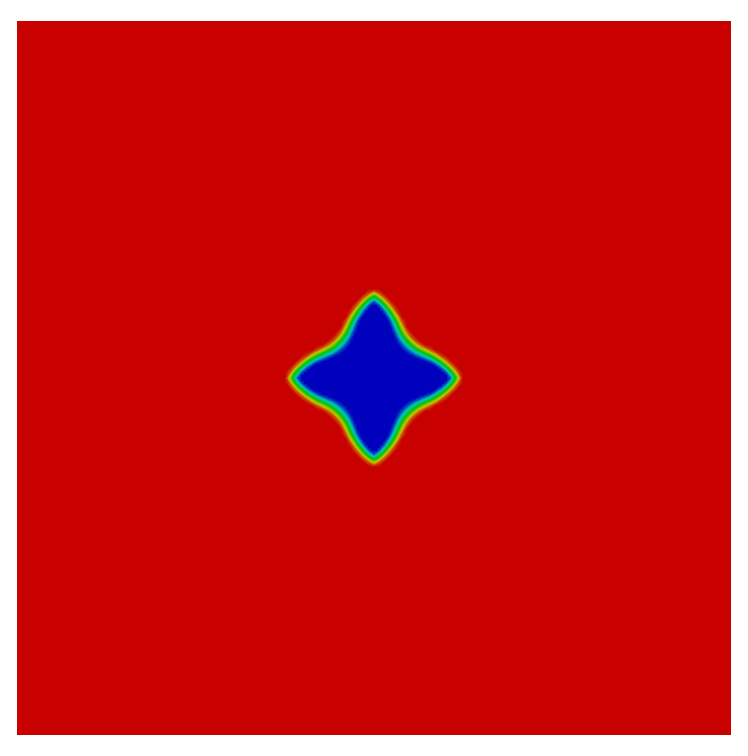}
\includegraphics[angle=-0,width=0.19\textwidth]{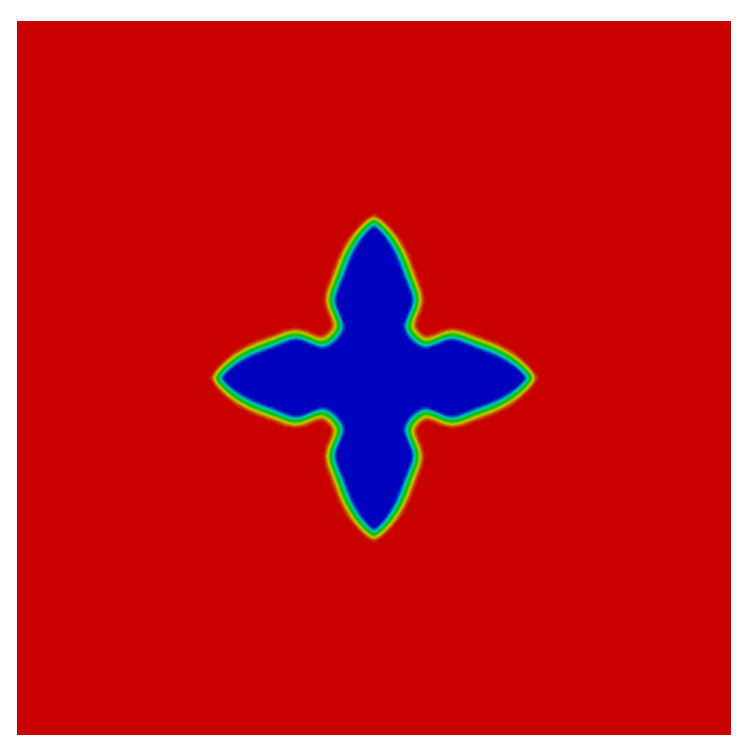}
\includegraphics[angle=-0,width=0.19\textwidth]{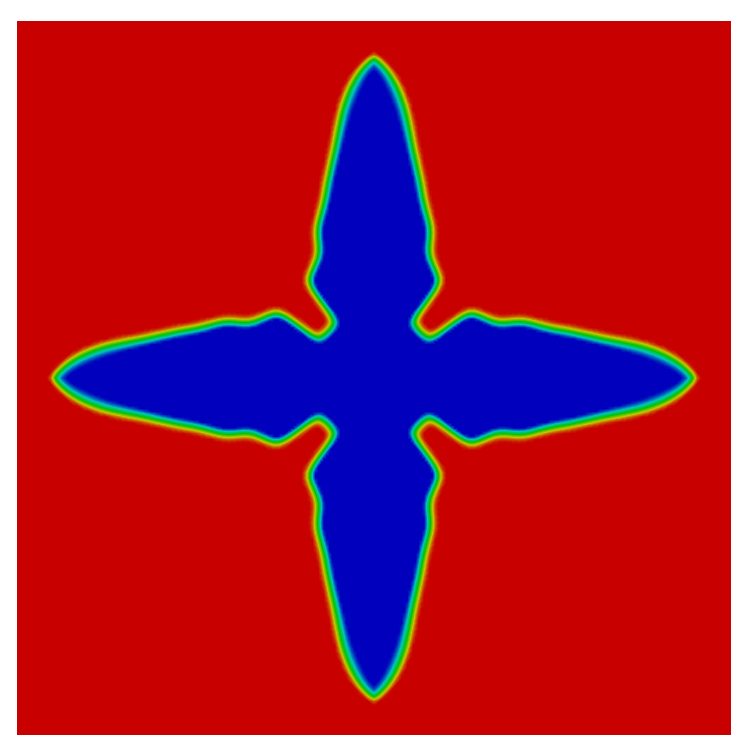}
\fi
\caption{($\epsilon^{-1} = 2\,\pi$, {\sc ani$_9$}, (\ref{eq:varrho})(ii),
$\alpha = 0.03$, $\rho=0$, $\uD = -2$, $\Omega=(-8,8)^3$)
Snapshots of the solution at times $t=0,\,0.1,\,0.2,\,0.5,\,1$.
[This computation took $13$ days.]
}
\label{fig:3dBSnewii_2pi}
\end{figure}%
If we repeat the simulation with $\rho=0.01$, which models the presence of
kinetic undercooling, the shape of the phase field approximation of the growing
crystal changes significantly. We present a run for the 
discretization parameters $N_f = 256$, $N_c = 64$, $\tau = 10^{-2}$ and $T=2$
in Figure~\ref{fig:3dBSrhoii_2pi}. Note that the larger time step size used
here yields a large reduction in the overall CPU time.
\begin{figure}
\center
\ifpdf
\includegraphics[angle=-0,width=0.19\textwidth]{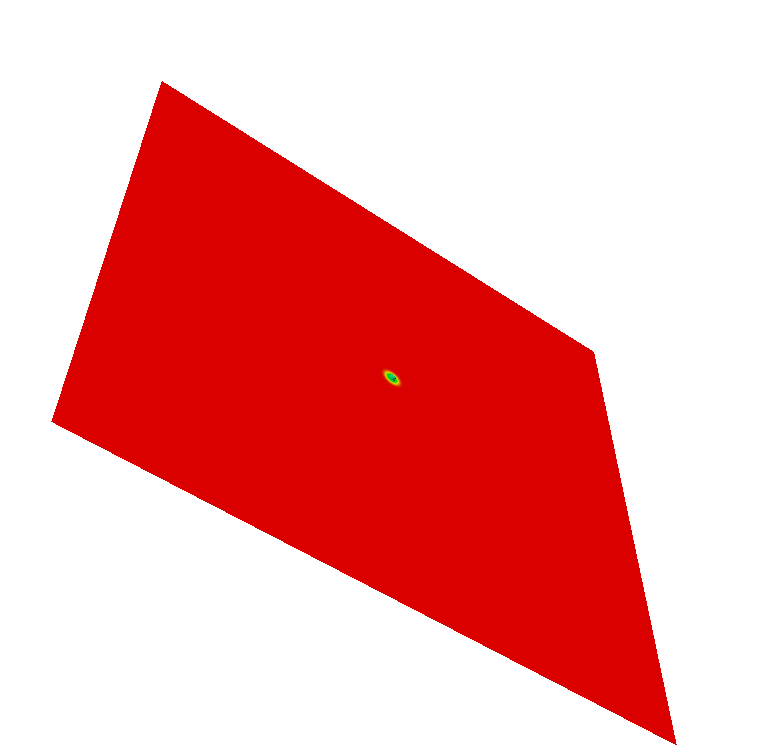}
\includegraphics[angle=-0,width=0.19\textwidth]{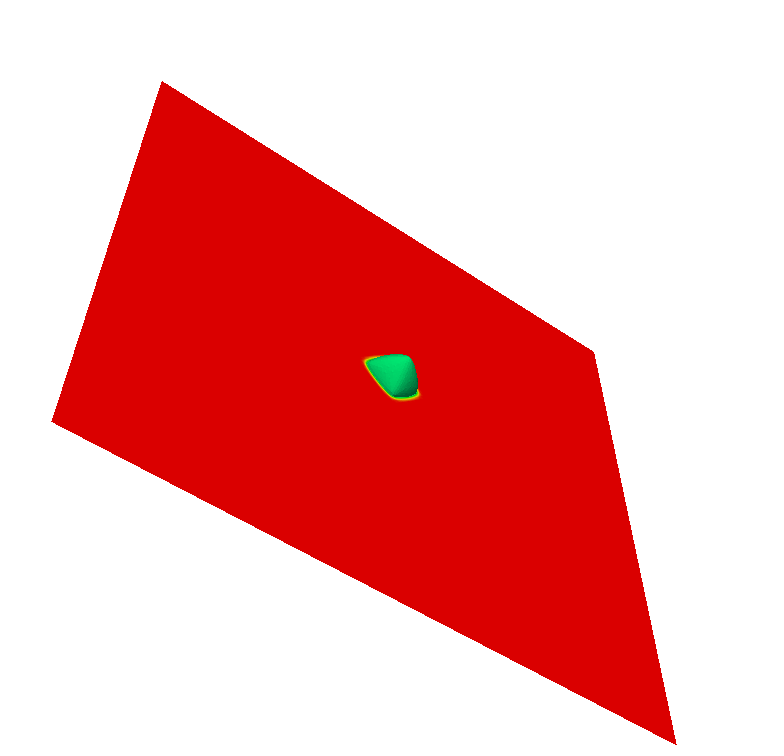}
\includegraphics[angle=-0,width=0.19\textwidth]{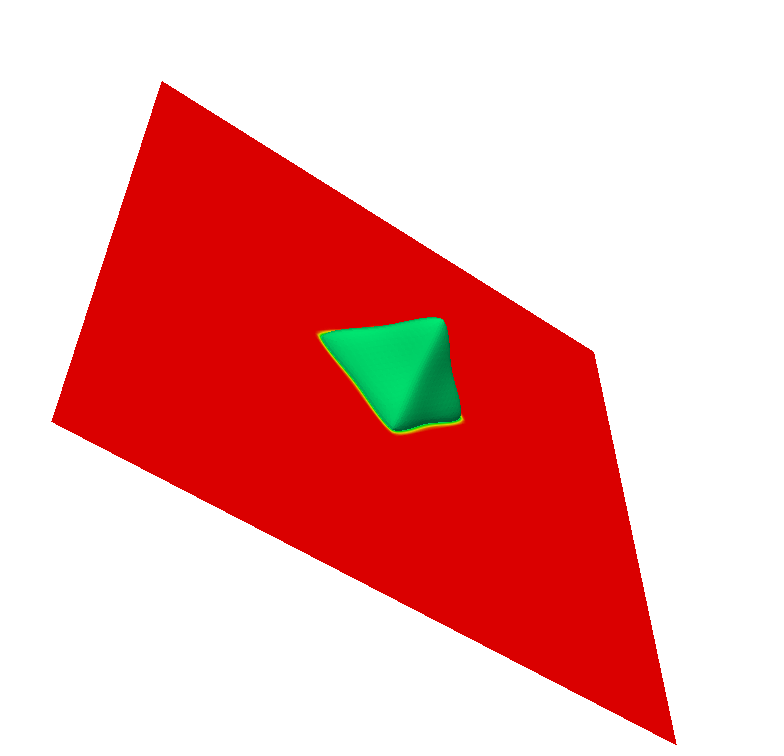}
\includegraphics[angle=-0,width=0.19\textwidth]{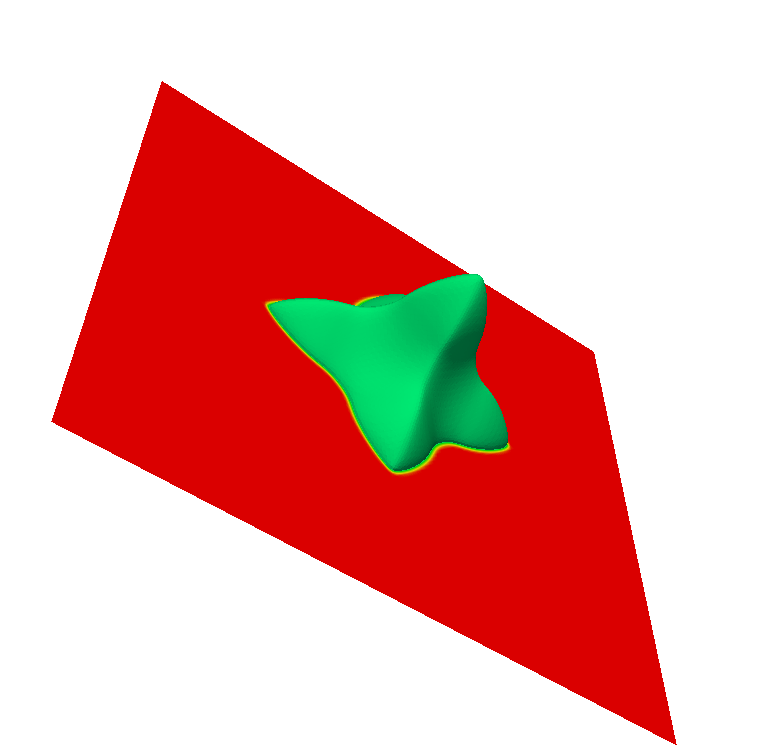}
\includegraphics[angle=-0,width=0.19\textwidth]{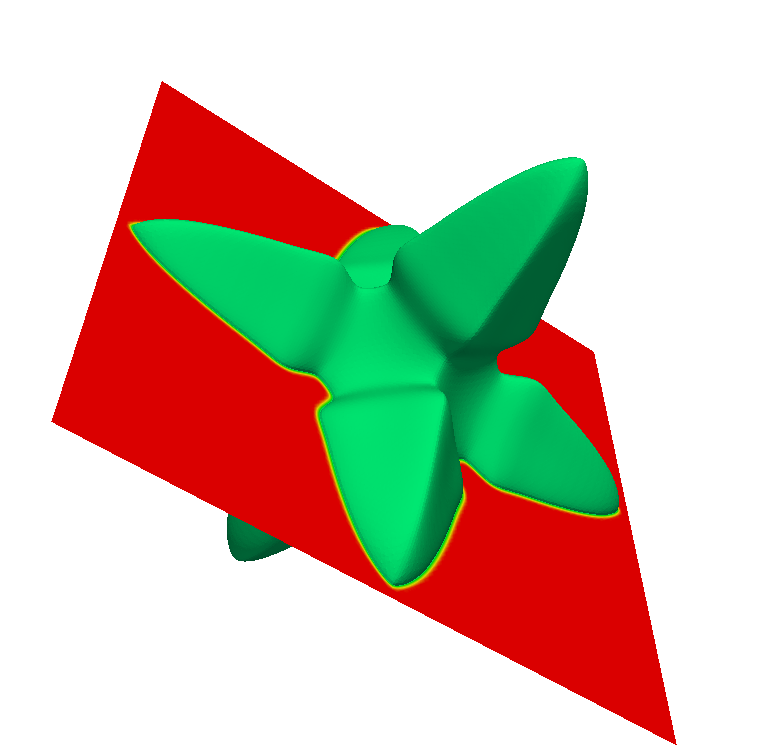}
\includegraphics[angle=-0,width=0.19\textwidth]{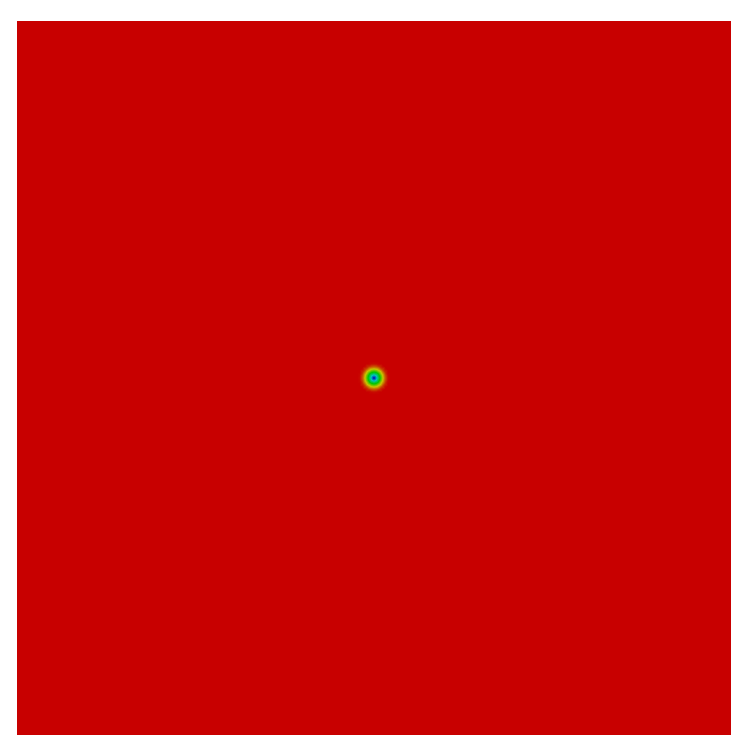}
\includegraphics[angle=-0,width=0.19\textwidth]{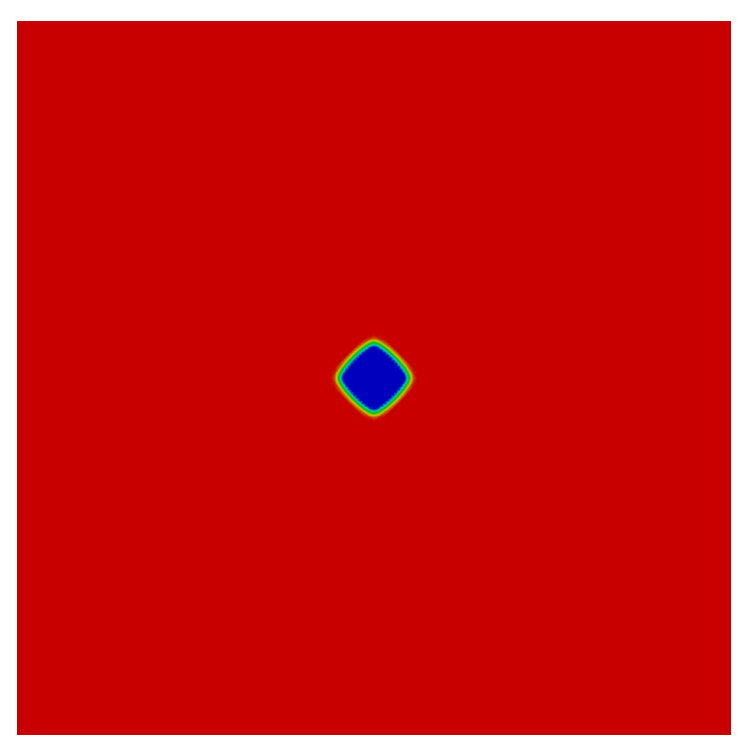}
\includegraphics[angle=-0,width=0.19\textwidth]{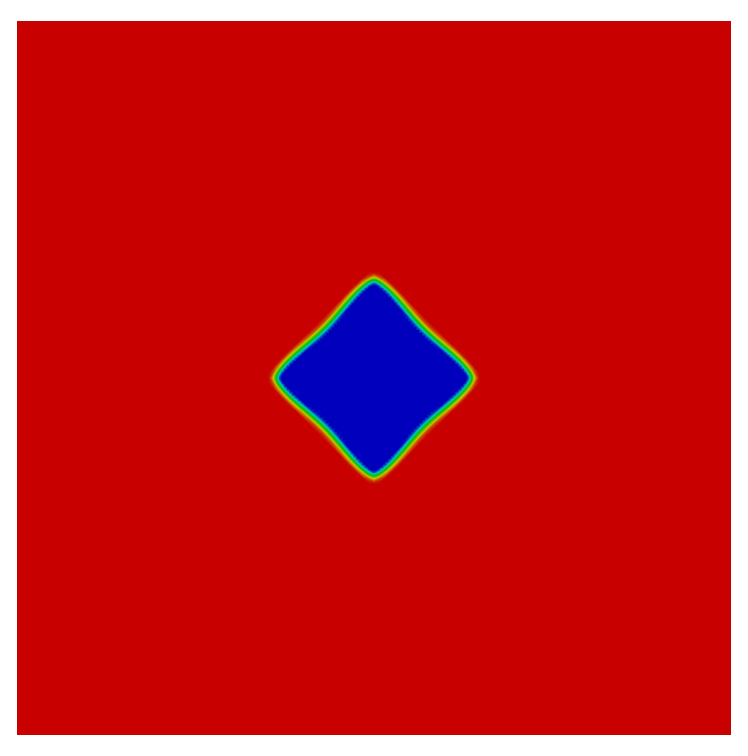}
\includegraphics[angle=-0,width=0.19\textwidth]{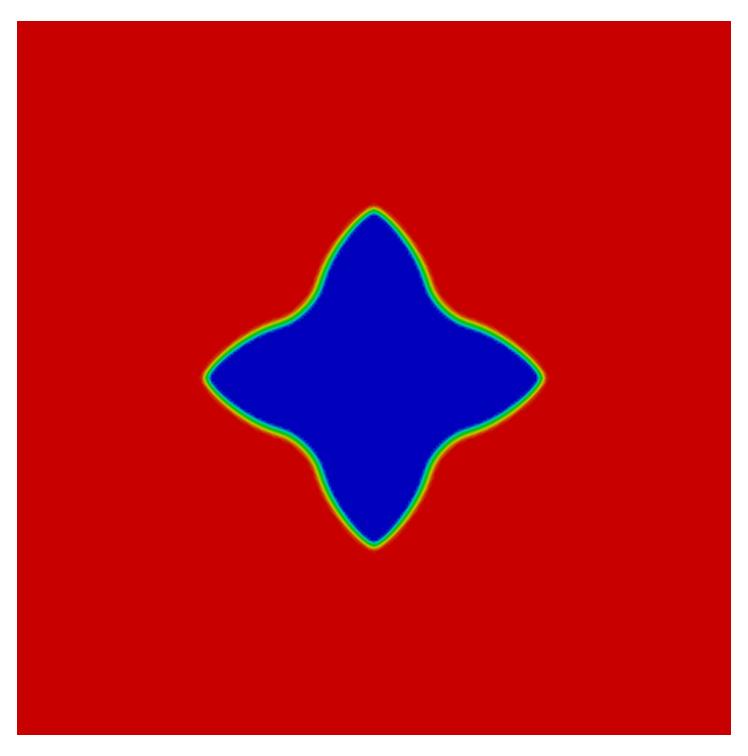}
\includegraphics[angle=-0,width=0.19\textwidth]{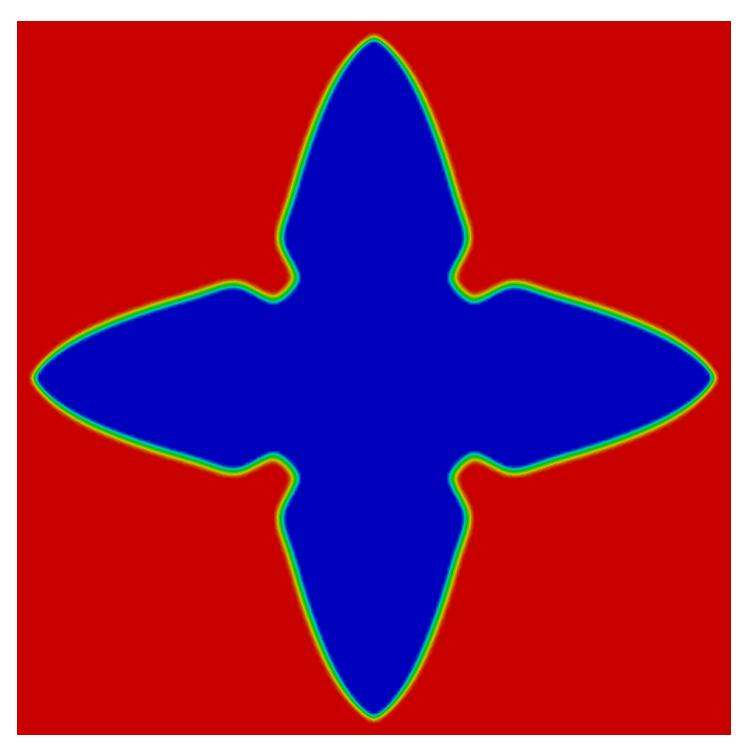}
\fi
\caption{($\epsilon^{-1} = 2\,\pi$, {\sc ani$_9$}, (\ref{eq:varrho})(ii),
$\alpha = 0.03$, $\rho=0.01$, $\uD = -2$, $\Omega=(-8,8)^3$)
Snapshots of the solution at times $t=0,\,0.1,\,0.5,\,1,\,2$.
[This computation took $4$ days.]
}
\label{fig:3dBSrhoii_2pi}
\end{figure}%

A repeat of the simulation in Figure~\ref{fig:3dBSrhoii_2pi}, but now for the
rotated hexagonal anisotropy {\sc ani$_4^\star$} can be seen in 
Figure~\ref{fig:3dxBSrhoii_2pi}. 
In this simulation we can
observe facet breaking, both in the basal and in the prismal directions,
similarly to the sharp interface computation shown in \citet[Fig.\ 18]{jcg}. 
\begin{figure}
\center
\ifpdf
\includegraphics[angle=-0,width=0.19\textwidth]{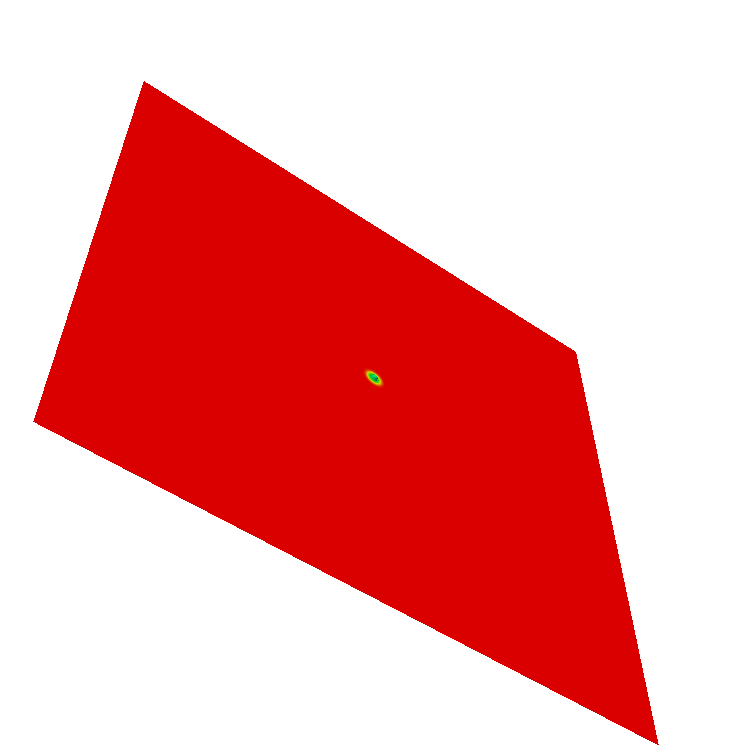}
\includegraphics[angle=-0,width=0.19\textwidth]{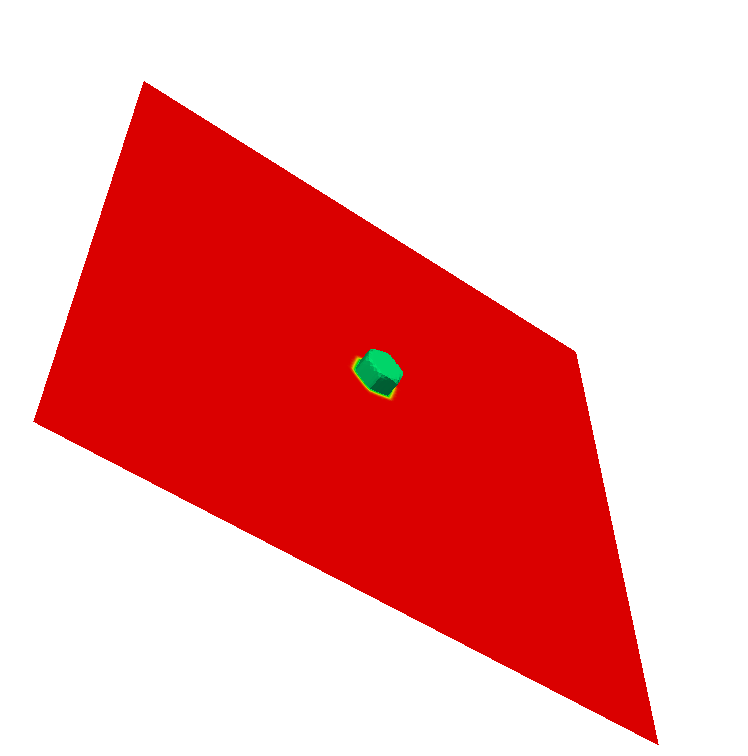}
\includegraphics[angle=-0,width=0.19\textwidth]{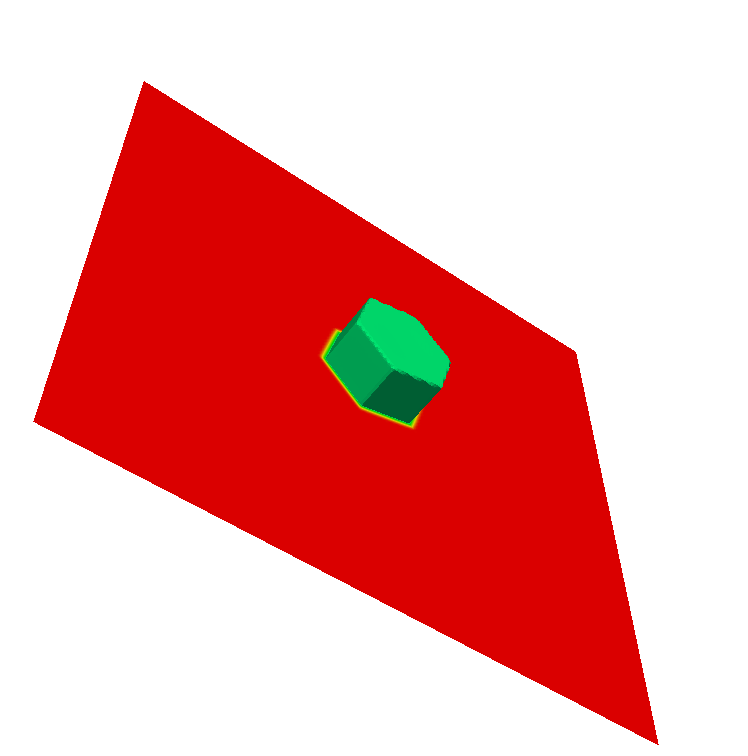}
\includegraphics[angle=-0,width=0.19\textwidth]{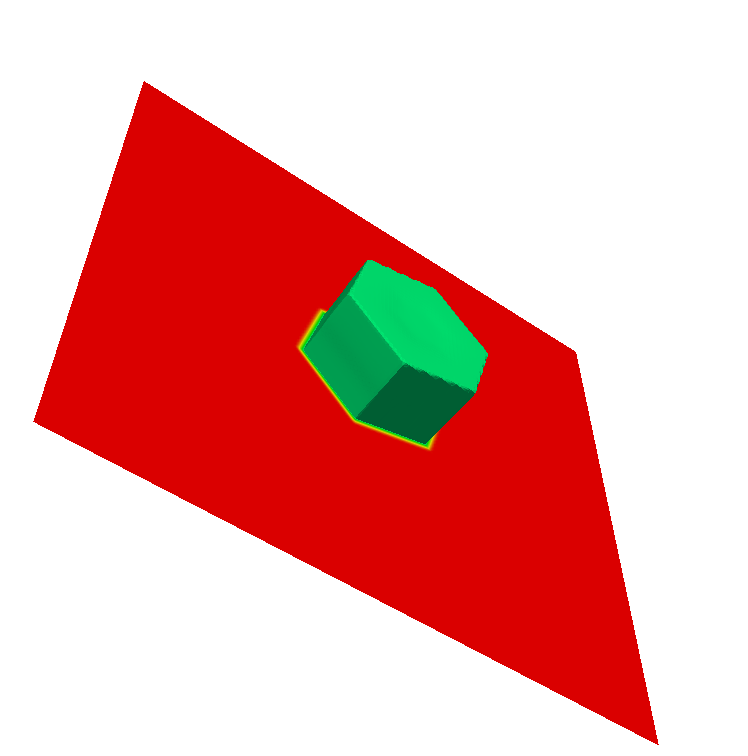}
\includegraphics[angle=-0,width=0.19\textwidth]{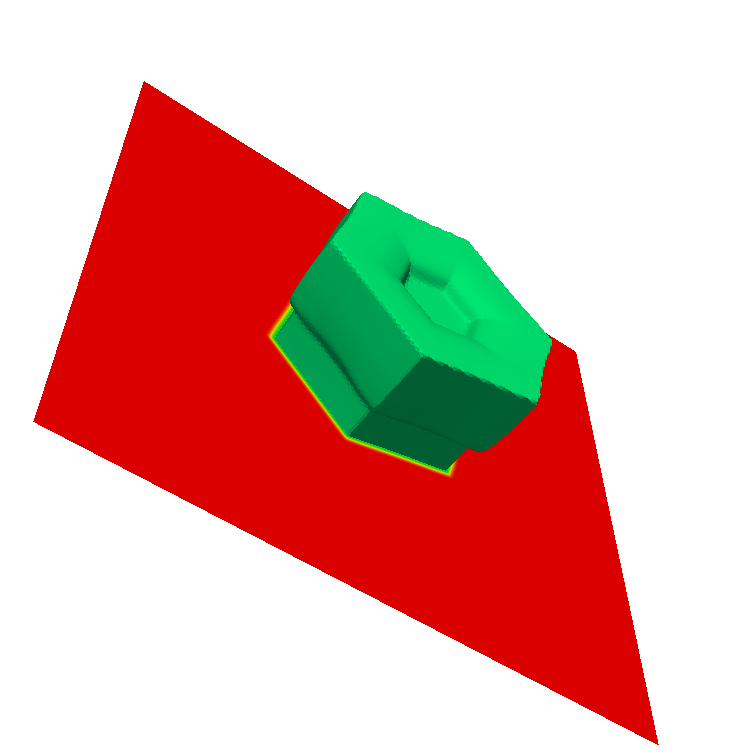}
\includegraphics[angle=-0,width=0.19\textwidth]{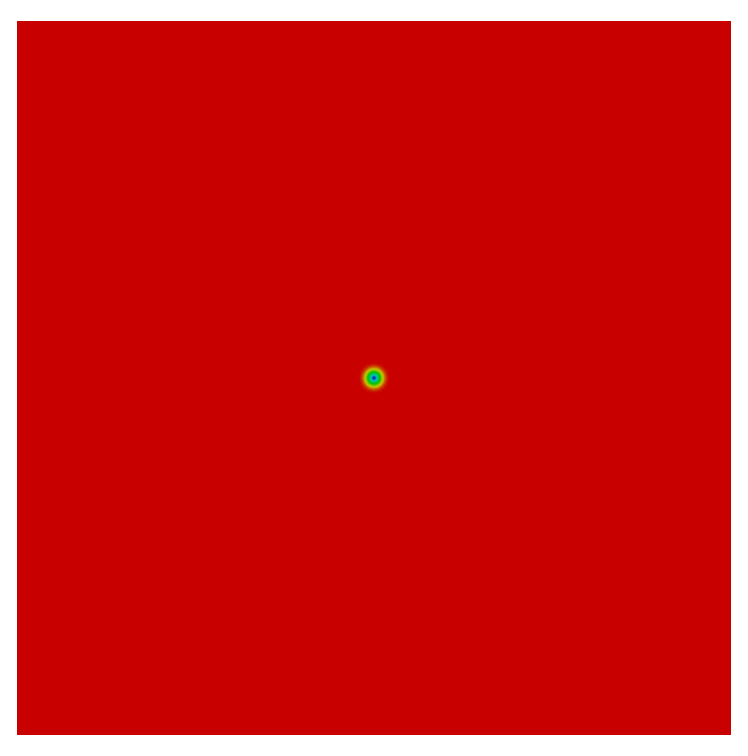}
\includegraphics[angle=-0,width=0.19\textwidth]{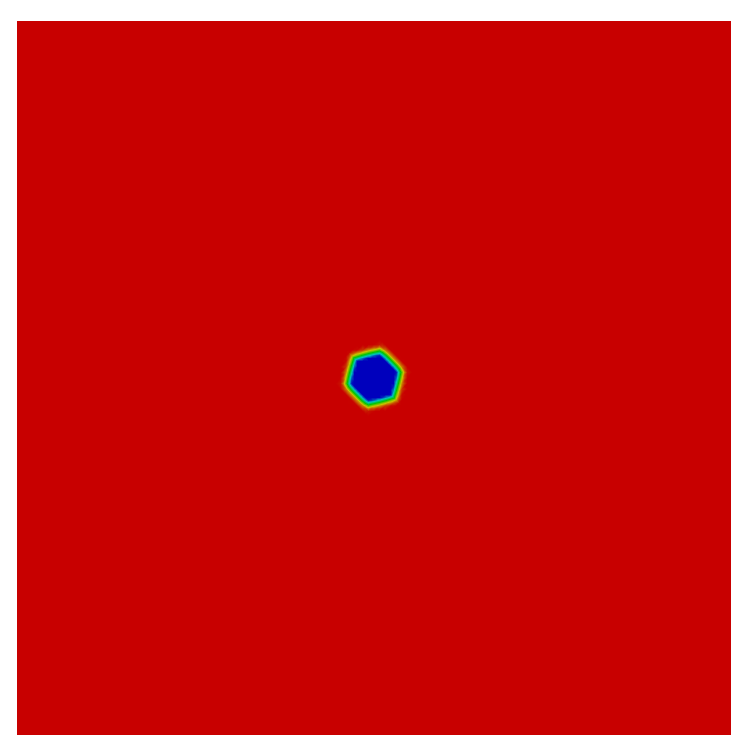}
\includegraphics[angle=-0,width=0.19\textwidth]{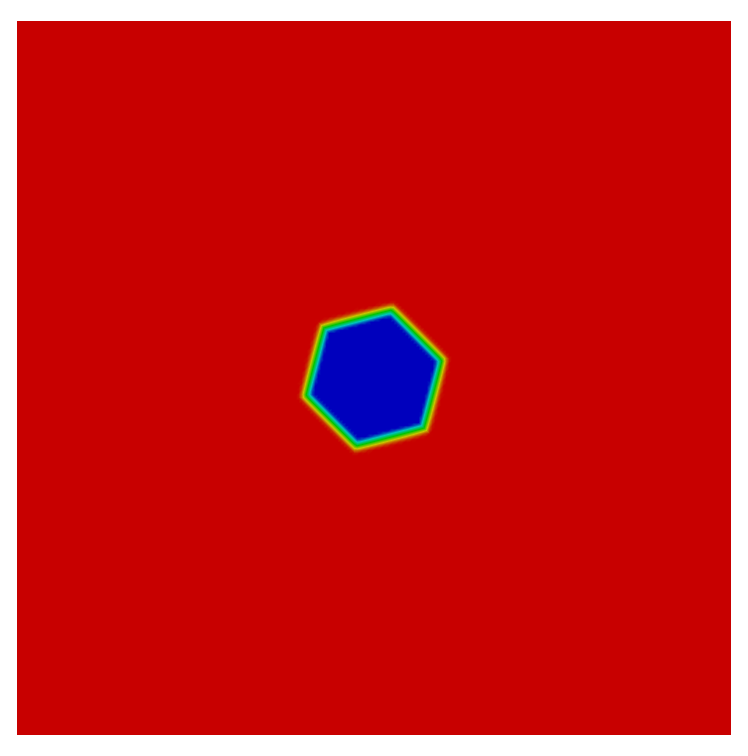}
\includegraphics[angle=-0,width=0.19\textwidth]{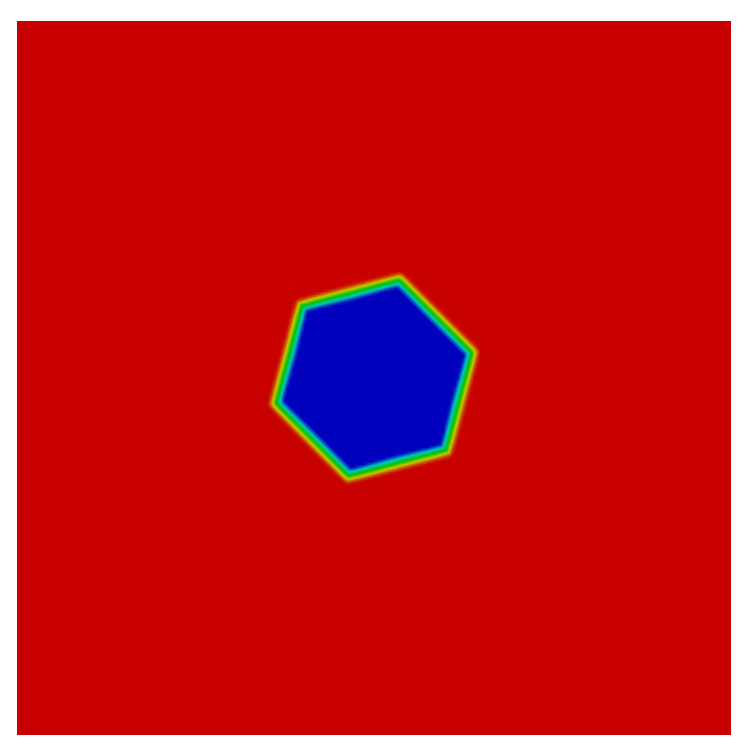}
\includegraphics[angle=-0,width=0.19\textwidth]{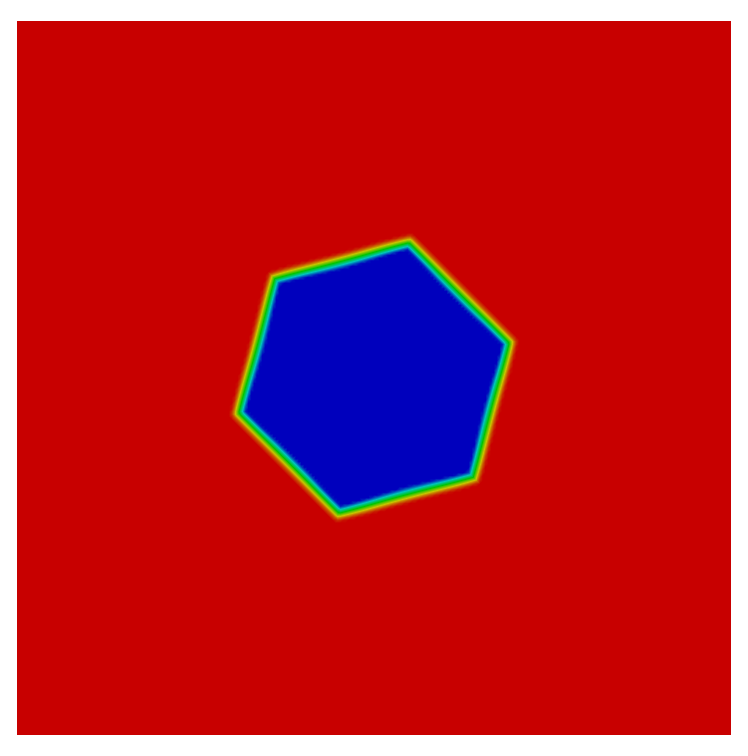}
\fi
\caption{($\epsilon^{-1} = 2\,\pi$, {\sc ani$_4^\star$}, (\ref{eq:varrho})(ii),
$\alpha = 0.03$, $\rho=0.01$, $\uD = -2$, $\Omega=(-8,8)^3$)
Snapshots of the solution at times $t=0,\,0.1,\,0.5,\,1,\,2$.
[This computation took $22$ hours.]
}
\label{fig:3dxBSrhoii_2pi}
\end{figure}%
With the next simulation we wish to highlight the effect that the choice of the
mobility coefficient $\beta$ can have on the evolution. If we replace
$\beta=\gamma$ with $\beta = \beta_{\rm flat,3}$, where
\begin{equation*} % \label{eq:betaflat}
\beta_{\rm flat,\ell}(\vec{p}) 
:= [p_1^2 + p_2^2 + 10^{-2\ell}\,p_3^2]^\frac12 \qquad \forall\ \vec{p}\in\R^d
\end{equation*}
is defined as in \citet[Eq.\ (16)]{jcg}, 
and if we keep all of the remaining parameters as before, then we
obtain the results shown in Figure~\ref{fig:3dhexBSrhoii_2pi}. Clearly, the
growing crystal now assumes the shape of a flat prism.
\begin{figure}
\center
\ifpdf
\includegraphics[angle=-0,width=0.19\textwidth]{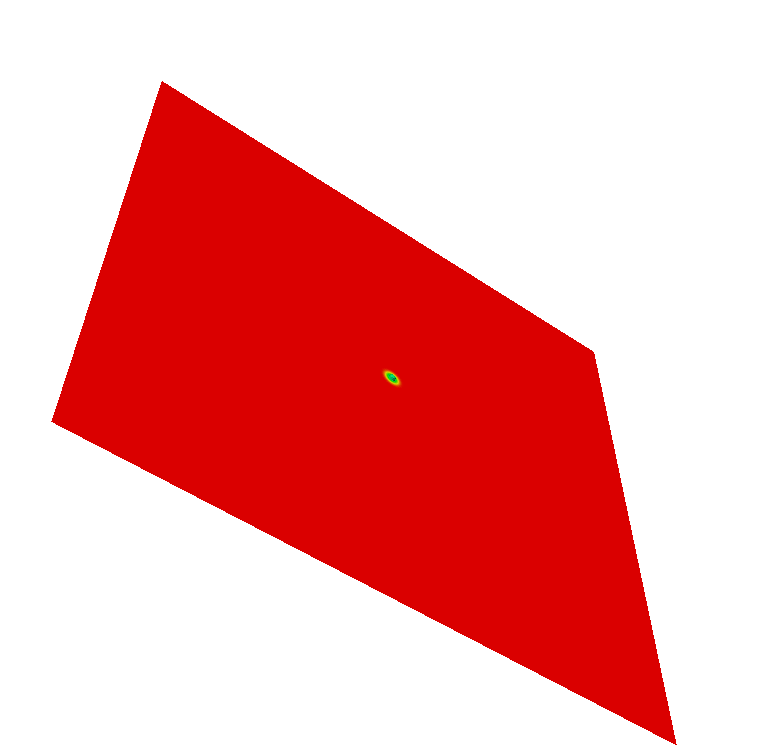}
\includegraphics[angle=-0,width=0.19\textwidth]{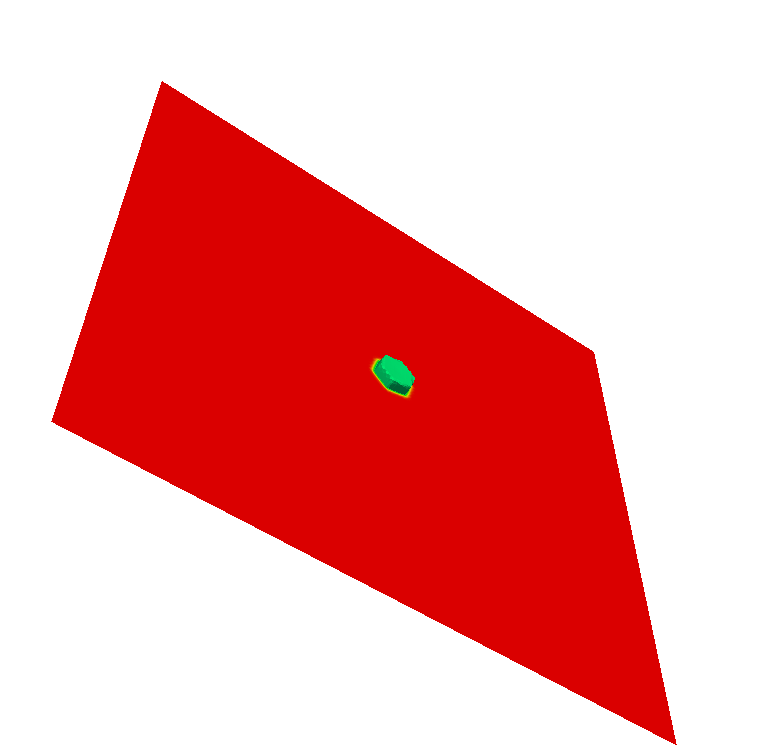}
\includegraphics[angle=-0,width=0.19\textwidth]{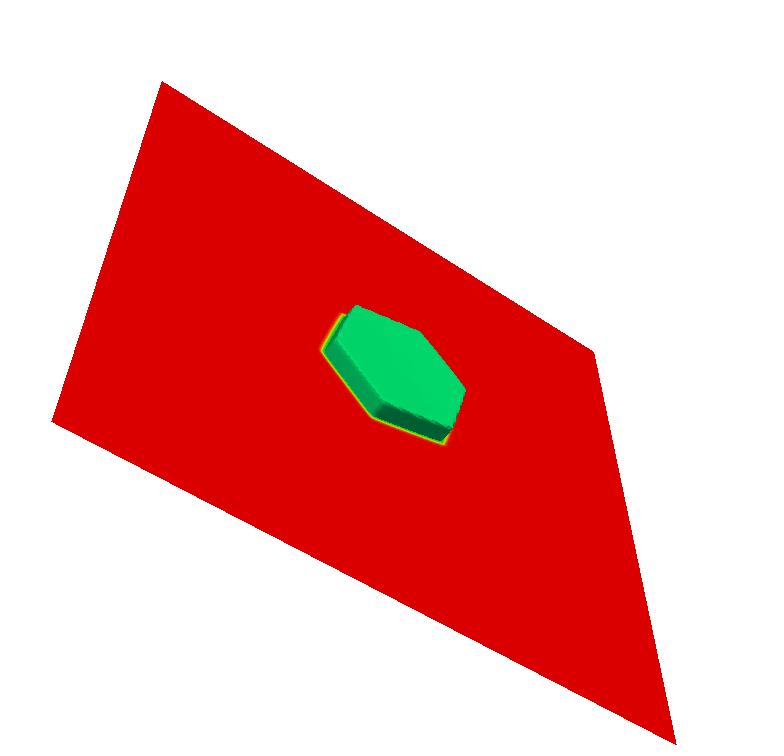}
\includegraphics[angle=-0,width=0.19\textwidth]{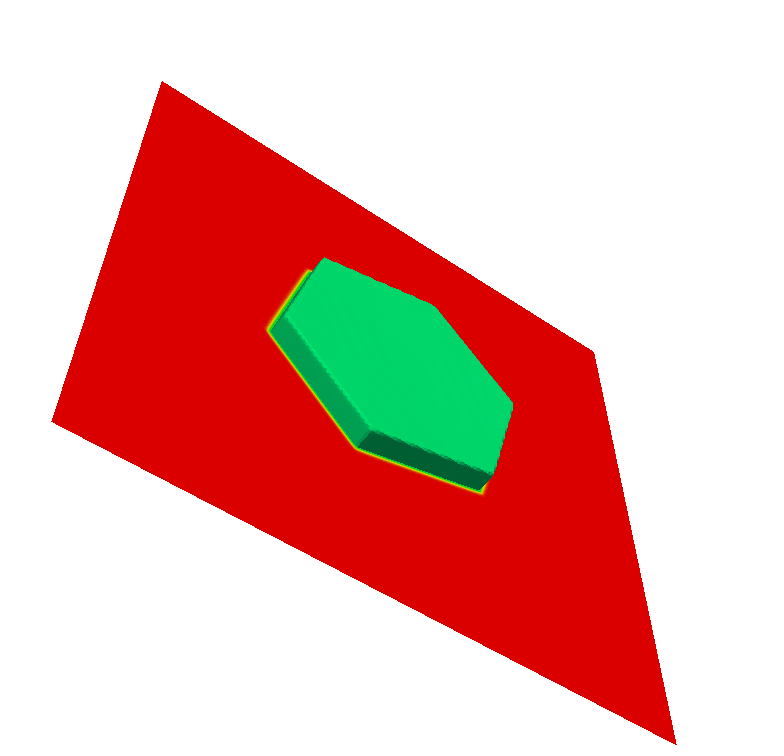}
\includegraphics[angle=-0,width=0.19\textwidth]{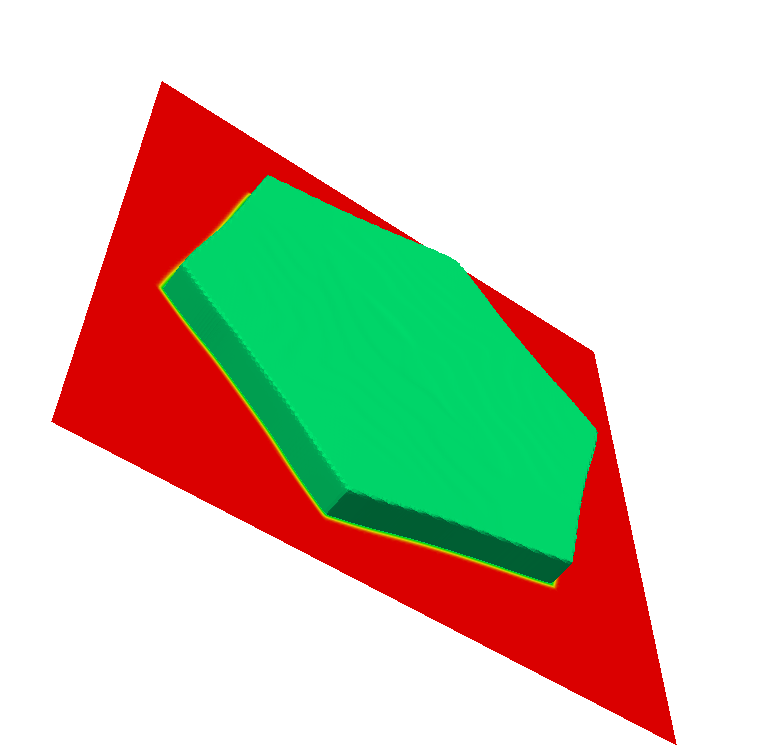}
\includegraphics[angle=-0,width=0.19\textwidth]{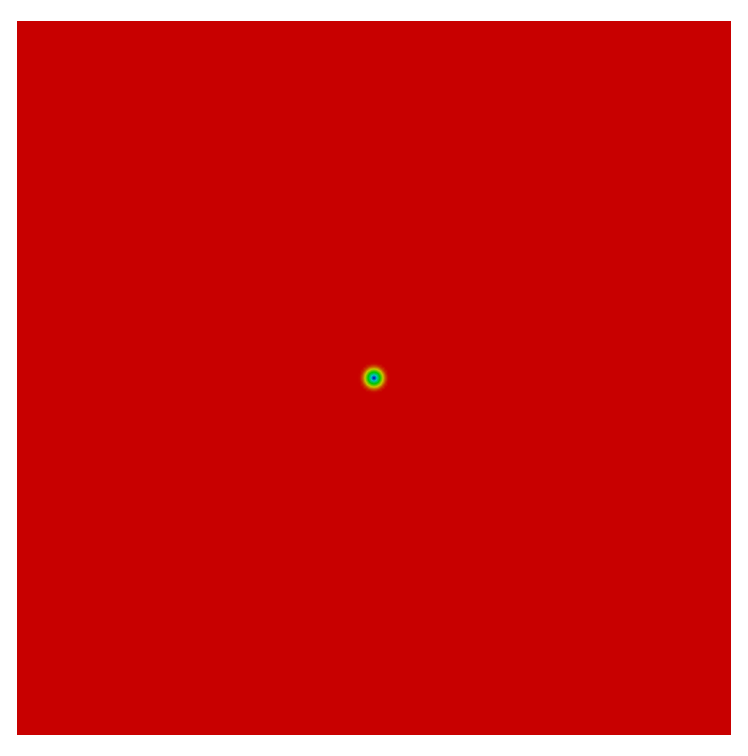}
\includegraphics[angle=-0,width=0.19\textwidth]{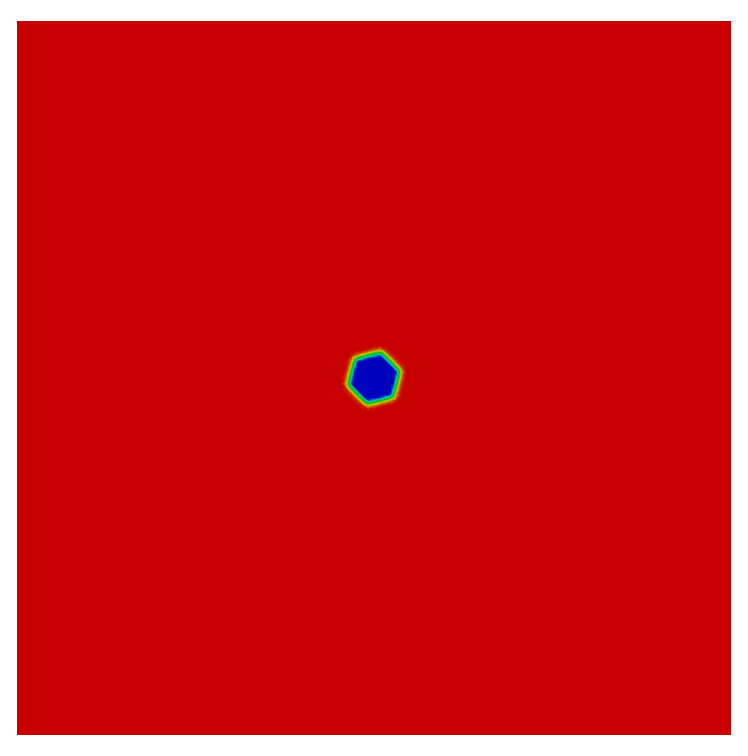}
\includegraphics[angle=-0,width=0.19\textwidth]{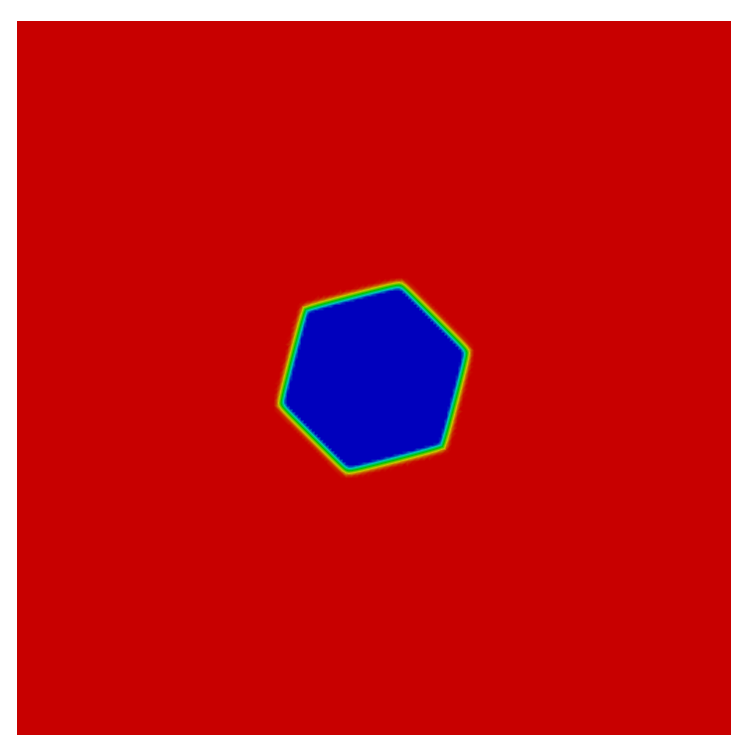}
\includegraphics[angle=-0,width=0.19\textwidth]{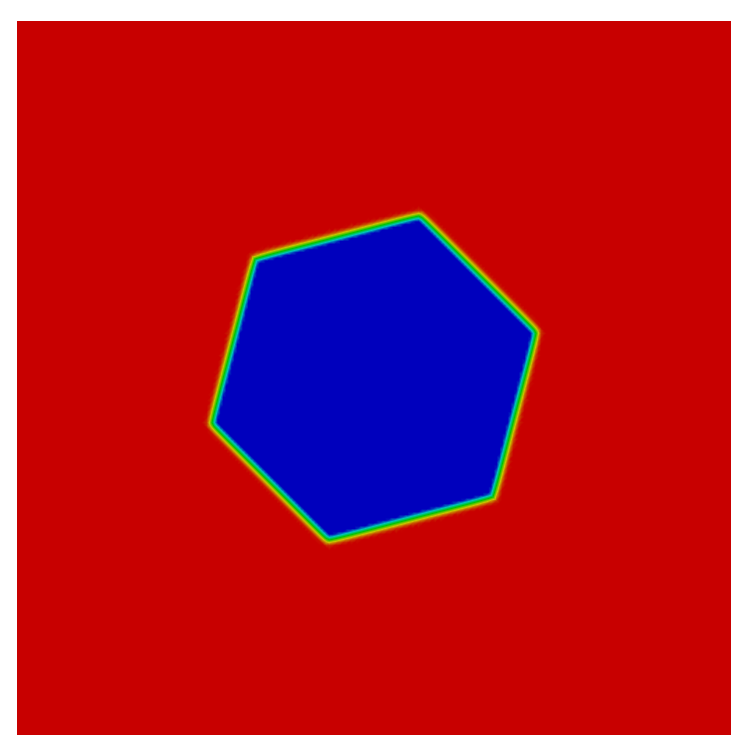}
\includegraphics[angle=-0,width=0.19\textwidth]{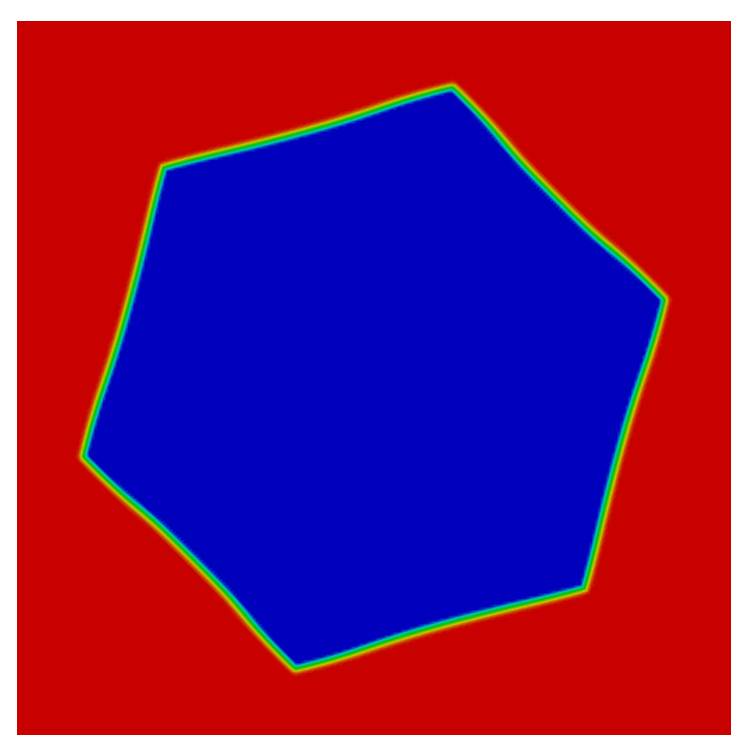}
\fi
\caption{($\epsilon^{-1} = 2\,\pi$, {\sc ani$_4^\star$}, (\ref{eq:varrho})(ii),
$\alpha = 0.03$, $\rho=0.01$, $\uD = -2$, $\Omega=(-8,8)^3$)
Snapshots of the solution at times $t=0,\,0.1,\,0.5,\,1,\,2$.
[This computation took $2$ days.]
}
\label{fig:3dhexBSrhoii_2pi}
\end{figure}%
Similarly, if we choose the mobility coefficient
$\beta = \beta_{\rm tall,2}$, where
\begin{equation*} % \label{eq:betatall}
\beta_{\rm tall,\ell}(\vec{p}) 
:= [10^{-2\ell}\,(p_1^2 + p_2^2) + p_3^2]^\frac12\qquad \forall\ \vec{p}\in\R^d
\end{equation*}
is defined as in \citet[Eq.\ (17)]{jcg}, 
then we obtain the simulation presented in 
Figure~\ref{fig:3dhextallBSrhoii_2pi}. 
This time the initially spherical crystal grows into a tall hexagonal prism.
\begin{figure}
\center
\ifpdf
\includegraphics[angle=-0,width=0.19\textwidth]{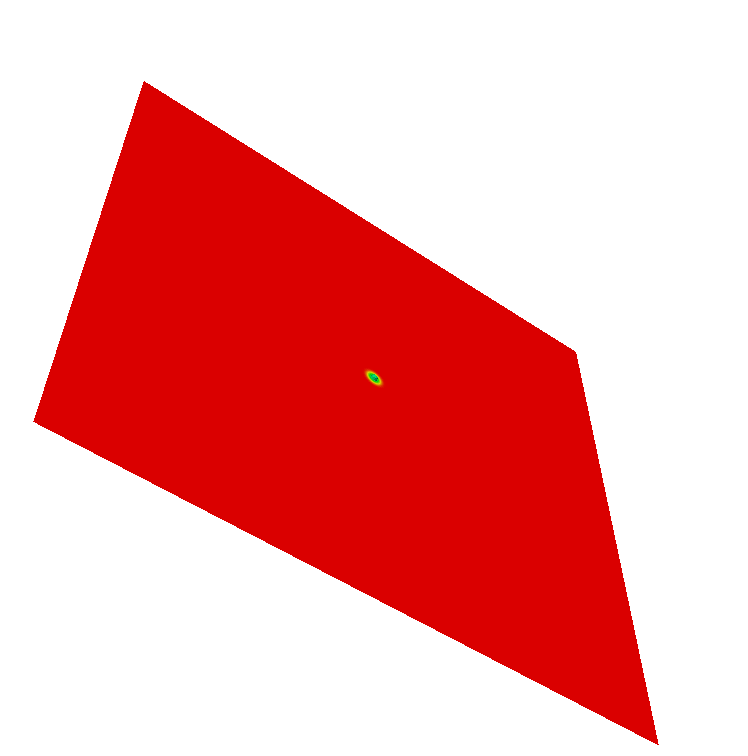}
\includegraphics[angle=-0,width=0.19\textwidth]{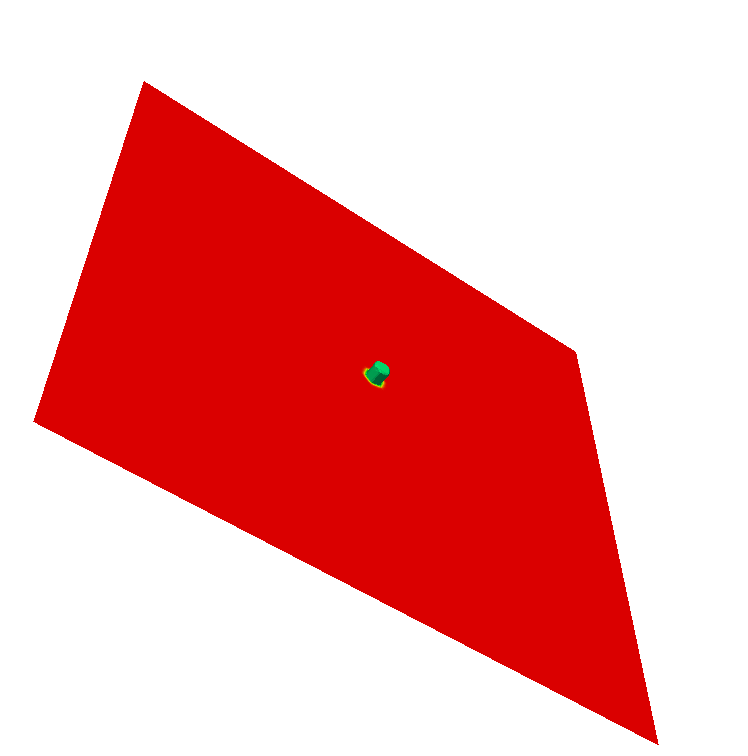}
\includegraphics[angle=-0,width=0.19\textwidth]{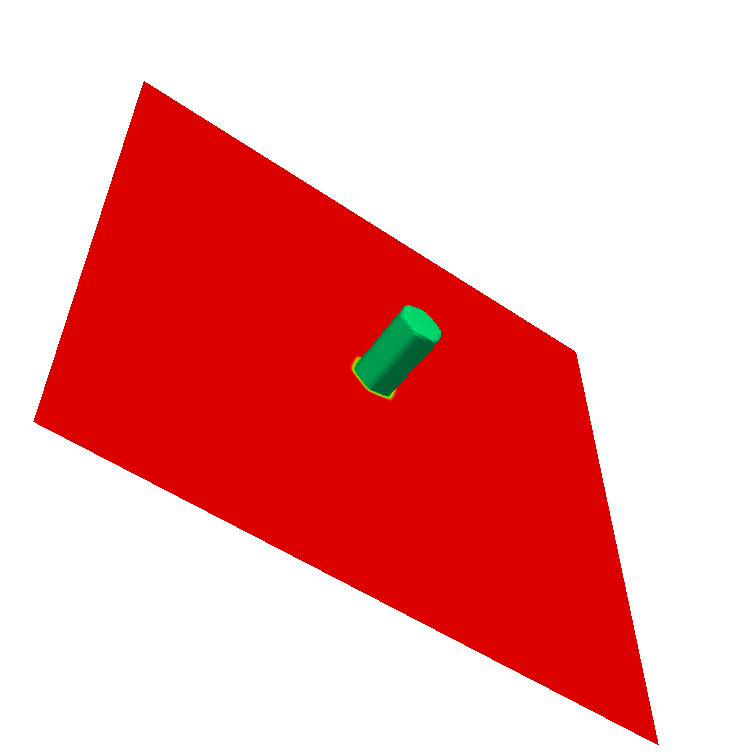}
\includegraphics[angle=-0,width=0.19\textwidth]{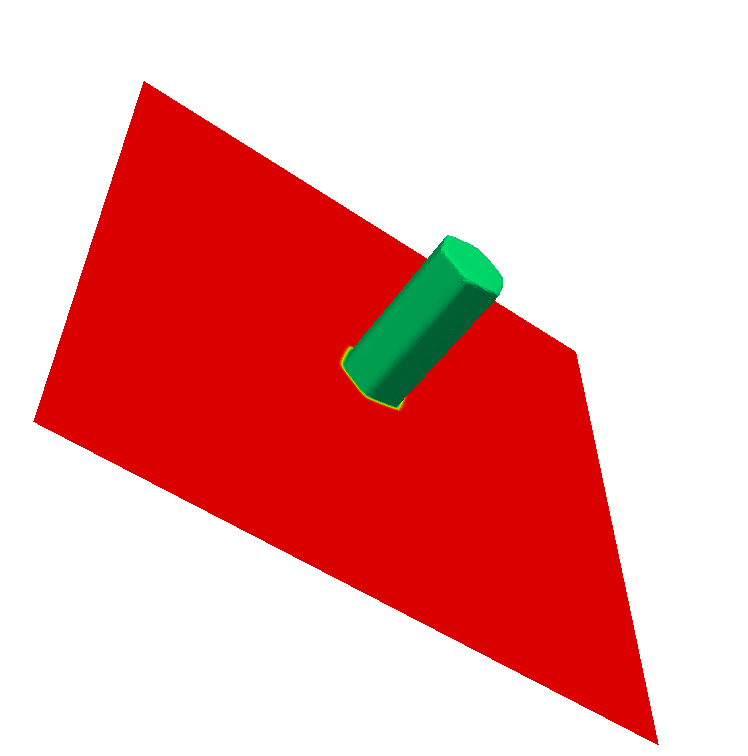}
\includegraphics[angle=-0,width=0.19\textwidth]{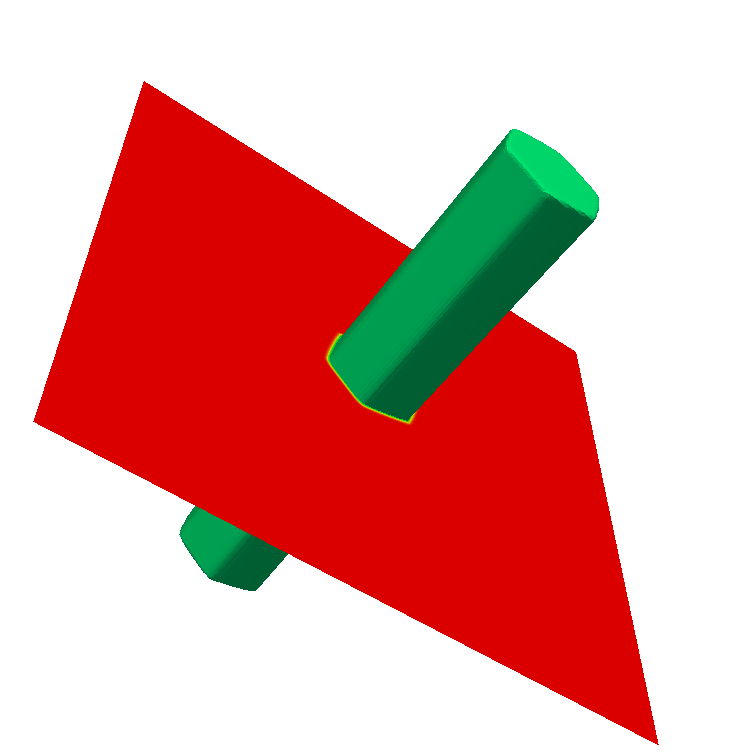}
\includegraphics[angle=-0,width=0.19\textwidth]{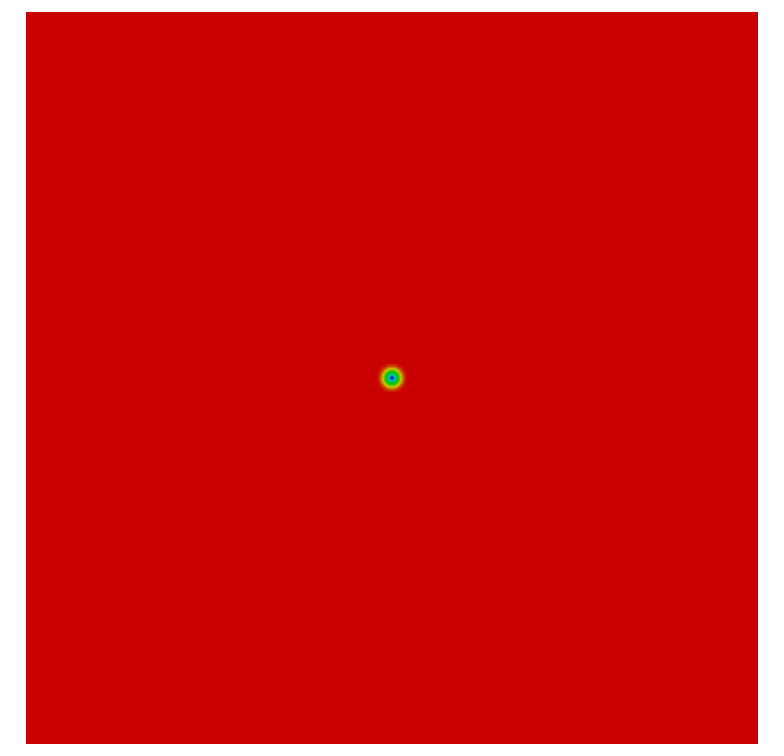}
\includegraphics[angle=-0,width=0.19\textwidth]{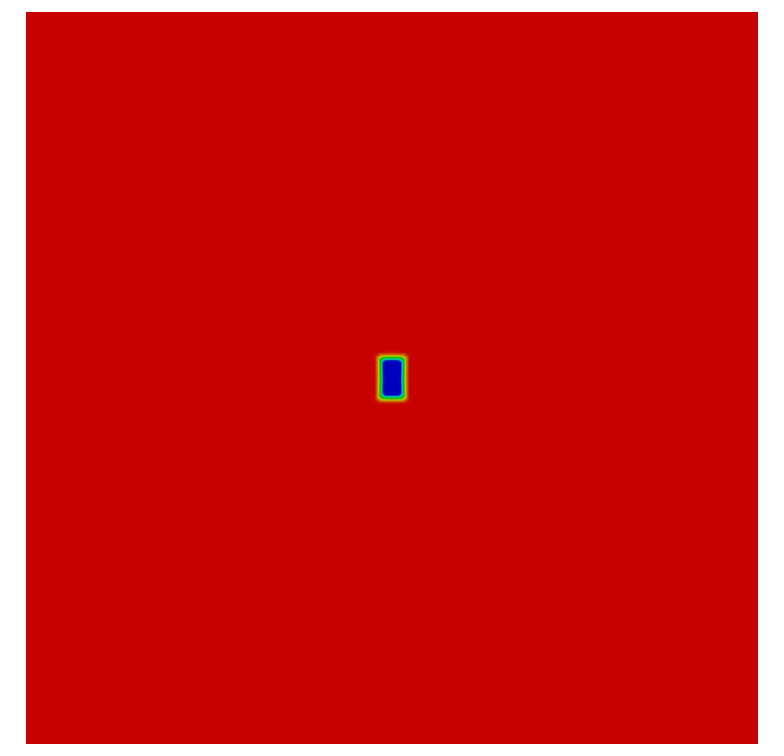}
\includegraphics[angle=-0,width=0.19\textwidth]{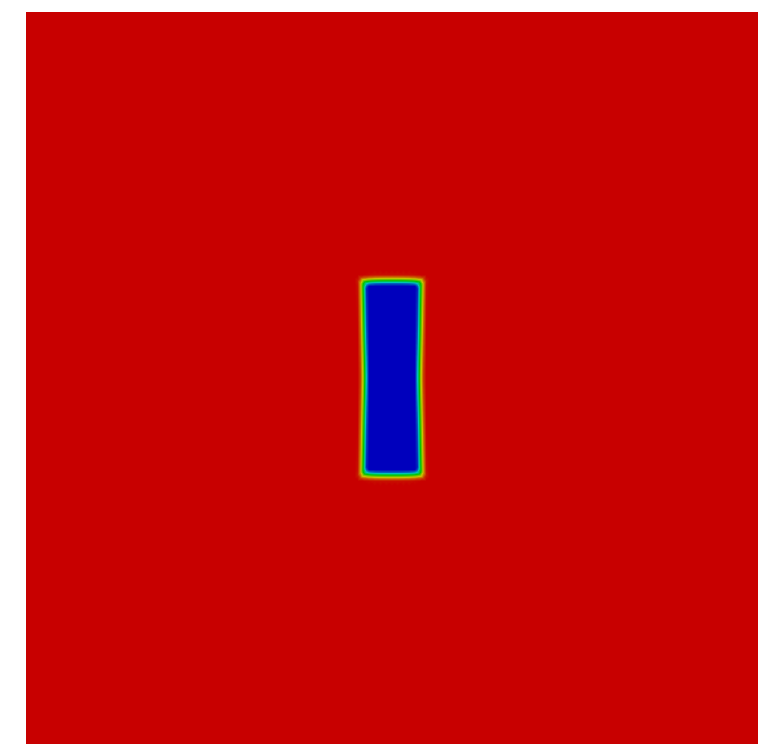}
\includegraphics[angle=-0,width=0.19\textwidth]{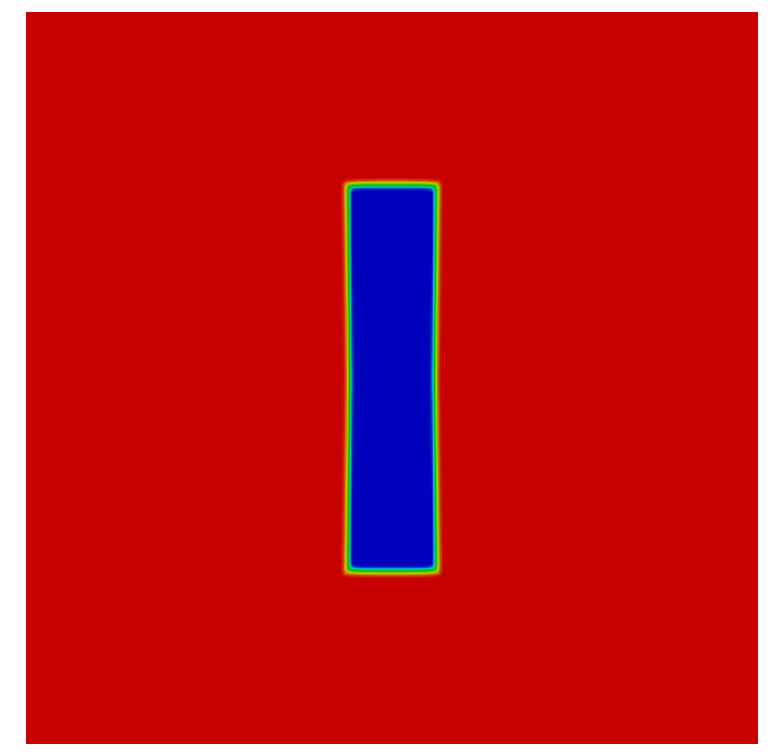}
\includegraphics[angle=-0,width=0.19\textwidth]{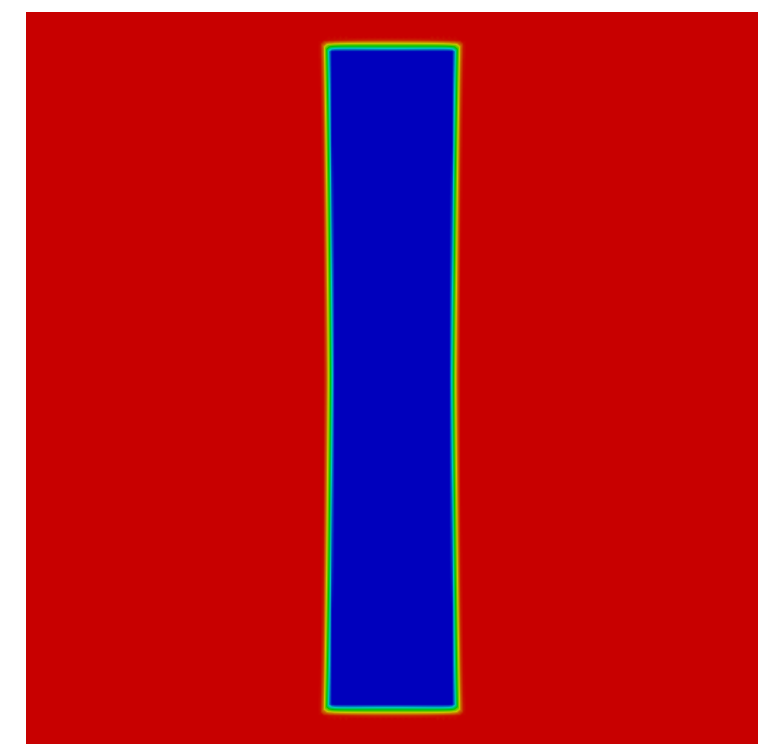}
\fi
\caption{($\epsilon^{-1} = 2\,\pi$, {\sc ani$_4^\star$}, (\ref{eq:varrho})(ii),
$\alpha = 0.03$, $\rho=0.01$, $\uD = -2$, $\Omega=(-8,8)^3$)
Snapshots of the solution at times $t=0,\,0.1,\,0.5,\,1,\,1.8$.
[This computation took $26$ hours.]
}
\label{fig:3dhextallBSrhoii_2pi}
\end{figure}%
It is discussed in \cite{Libbrecht05} that different mobility coefficients
$\beta$ are responsible for the various snow crystal shapes seen in nature.
In this context we remark that 
(\ref{eq:1a}--e) also appears in solidification from a supersaturated solution.
In this case $-u$ is a suitably scaled concentration with $-\uD$ being the 
scaled supersaturation, see e.g.\ \cite{jcg} for more details.

\subsection{Stefan problem in three space dimensions} \label{sec:54}
In this subsection we present a simulation for the full Stefan problem 
in three space dimensions for the anisotropy {\sc ani$_9$}. 
To this end, we consider the physical parameters
$\vartheta=1$, $\alpha=10^{-3}$, $\rho=0.01$, $\uD = -0.5$ and let
$\Omega = (-4,4)^3$.
A numerical computation for $\epsilon^{-1} = 16\,\pi$, together with 
$N_f = 1024$, $N_c = 64$, $\tau = 10^{-4}$ and $T=0.4$ can be seen in
Figure~\ref{fig:3dnewStefanii_16pi}. Similarly to the results in 
Figure~\ref{fig:3dBSrhoii_2pi} we observe that the growing crystal exhibits the
typical six symmetric side arms that are common in simulations of dendritic
growth.
\begin{figure}
\center
\ifpdf
\includegraphics[angle=-0,width=0.19\textwidth]{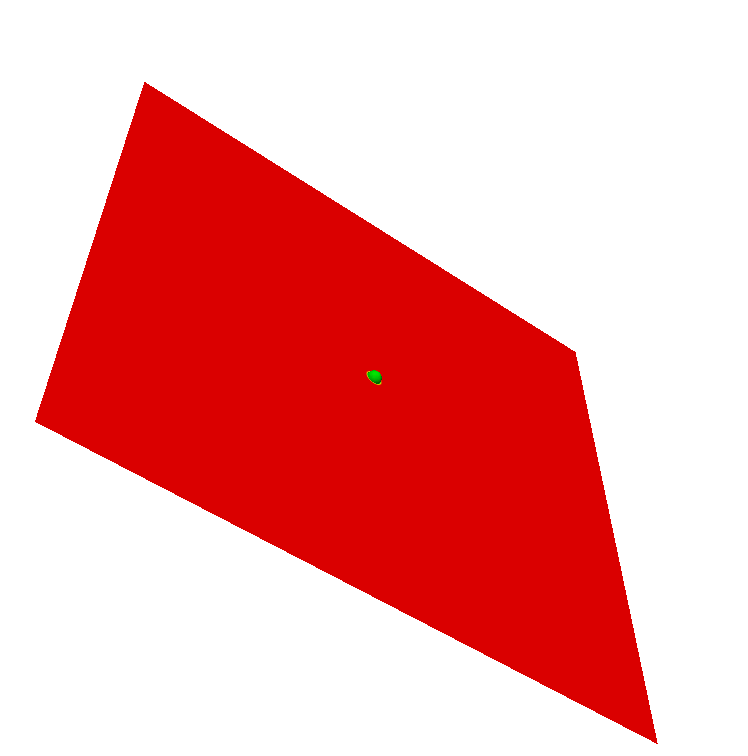}
\includegraphics[angle=-0,width=0.19\textwidth]{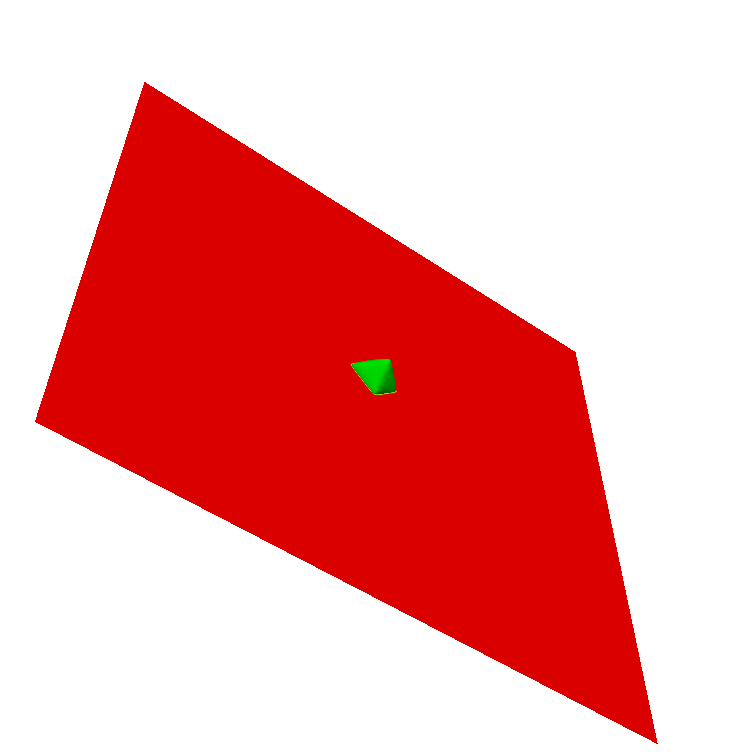}
\includegraphics[angle=-0,width=0.19\textwidth]{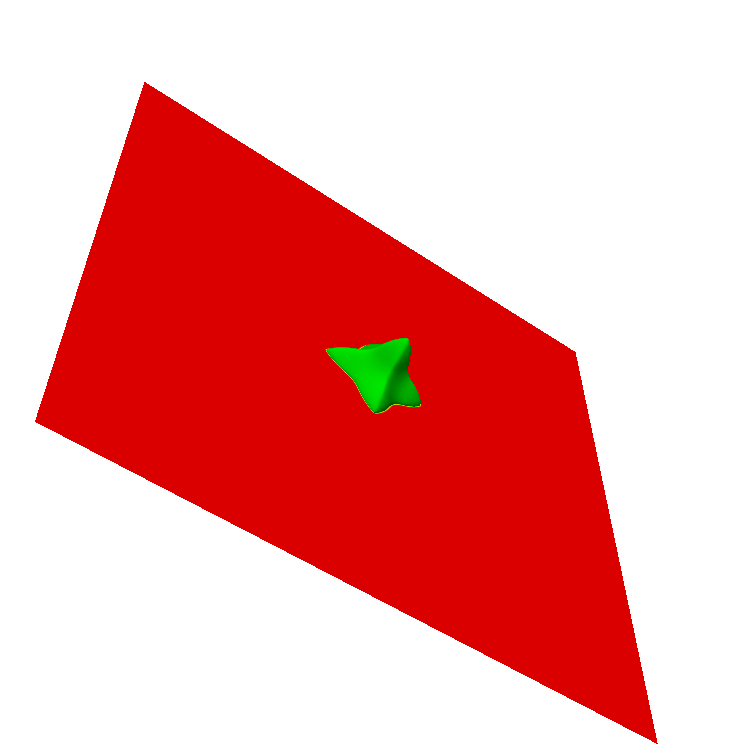}
\includegraphics[angle=-0,width=0.19\textwidth]{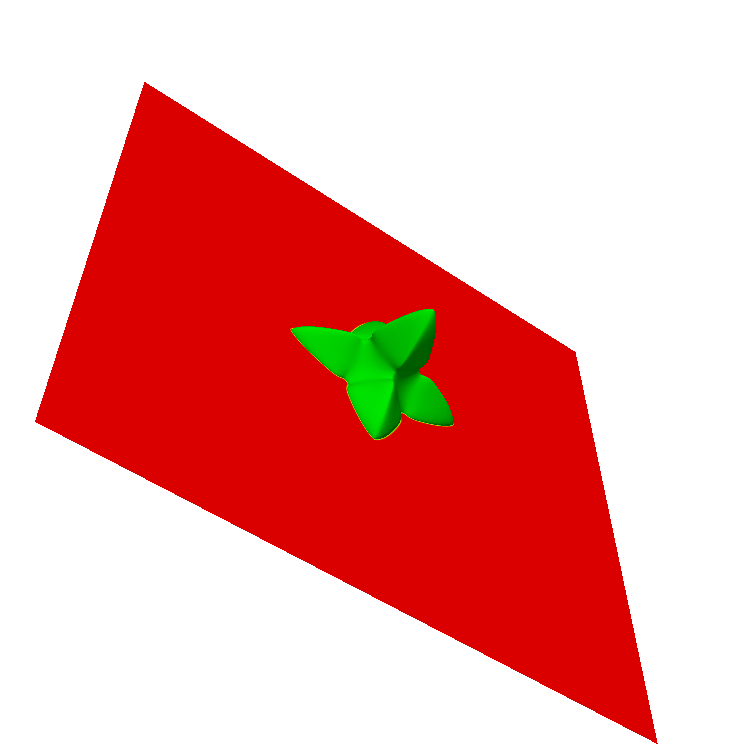}
\includegraphics[angle=-0,width=0.19\textwidth]{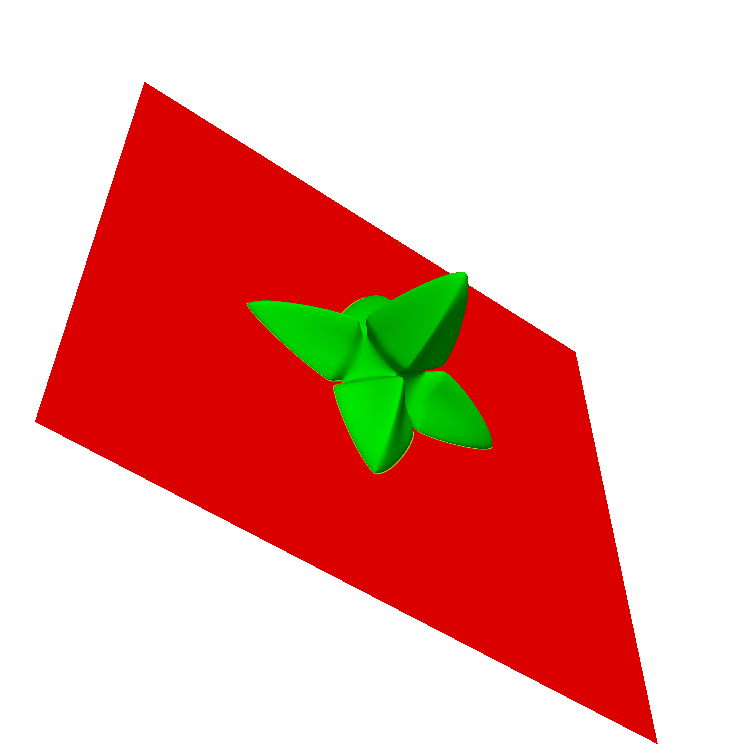}
\includegraphics[angle=-0,width=0.19\textwidth]{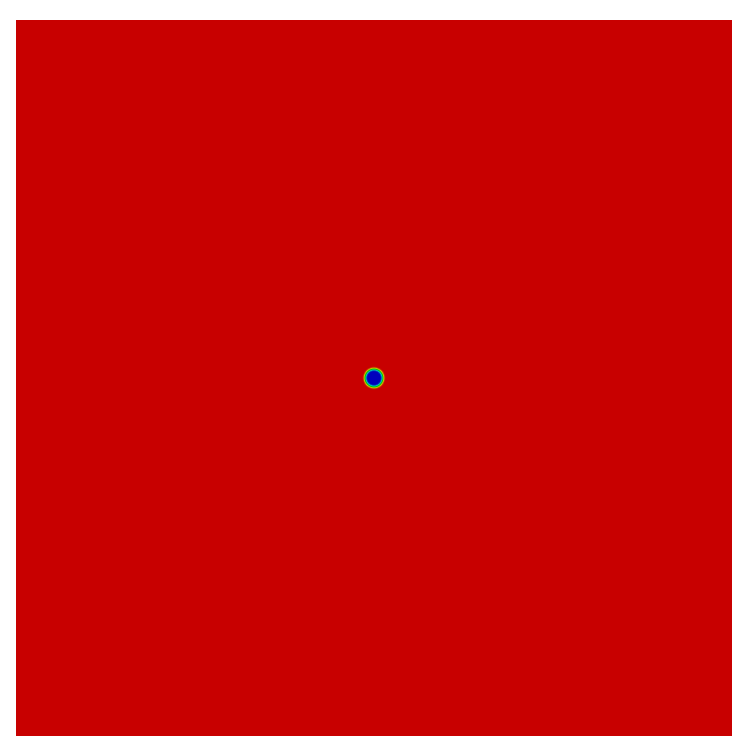}
\includegraphics[angle=-0,width=0.19\textwidth]{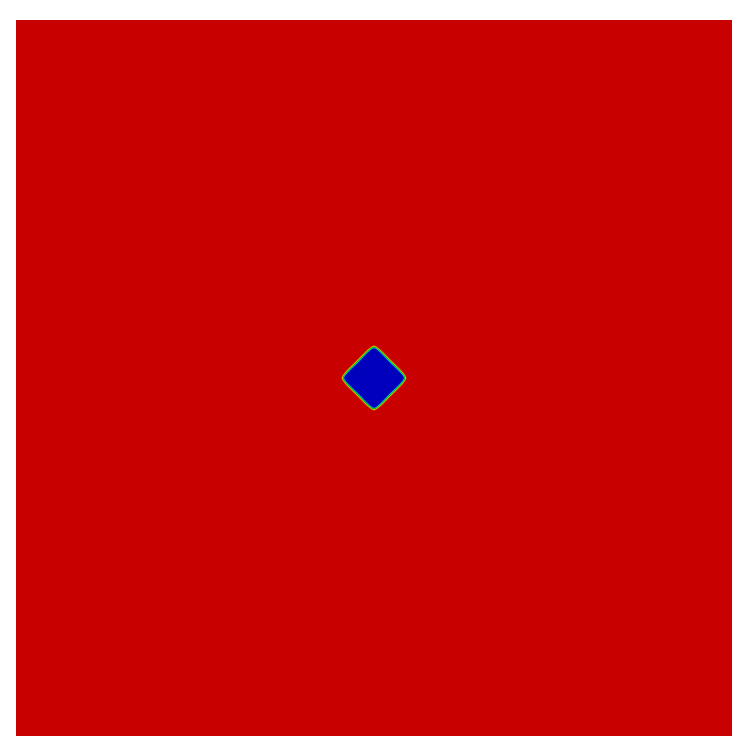}
\includegraphics[angle=-0,width=0.19\textwidth]{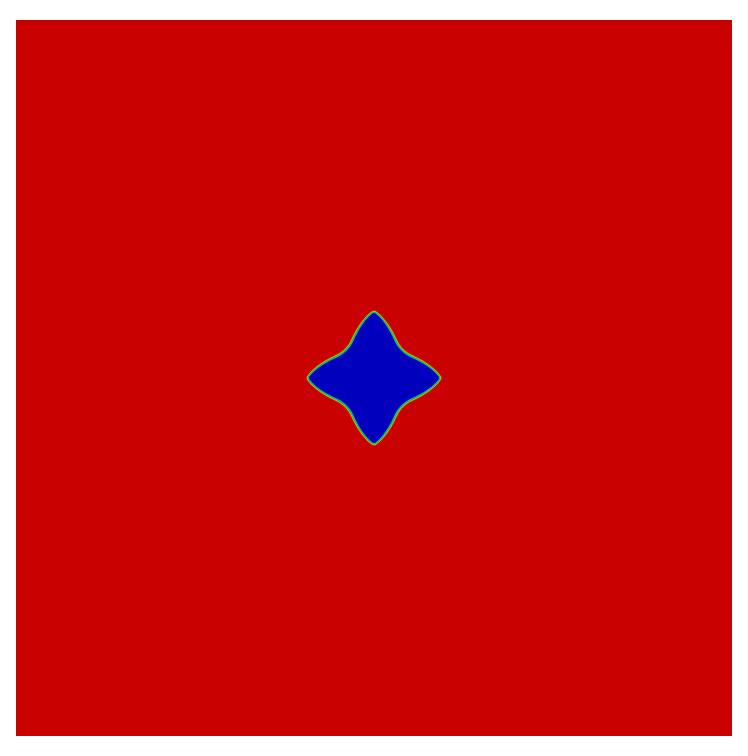}
\includegraphics[angle=-0,width=0.19\textwidth]{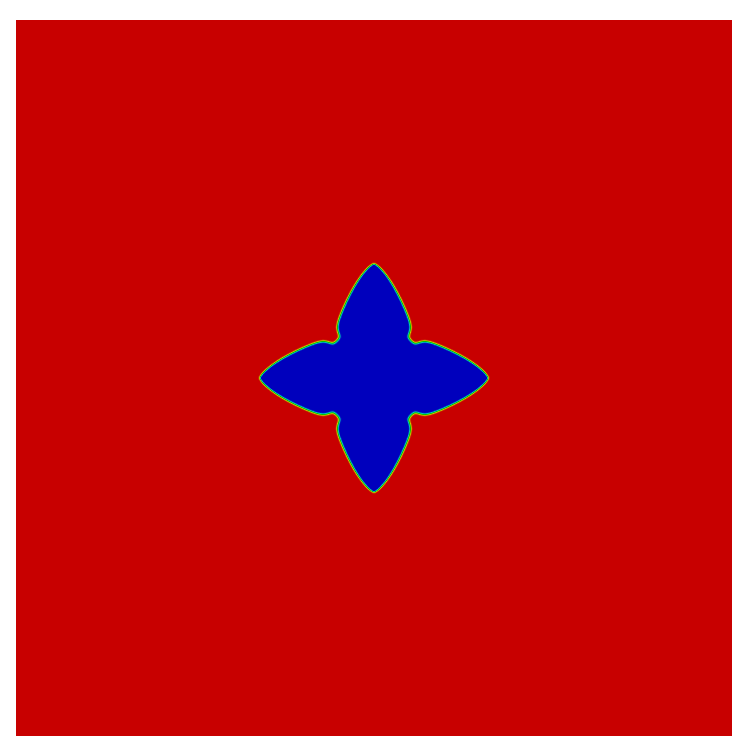}
\includegraphics[angle=-0,width=0.19\textwidth]{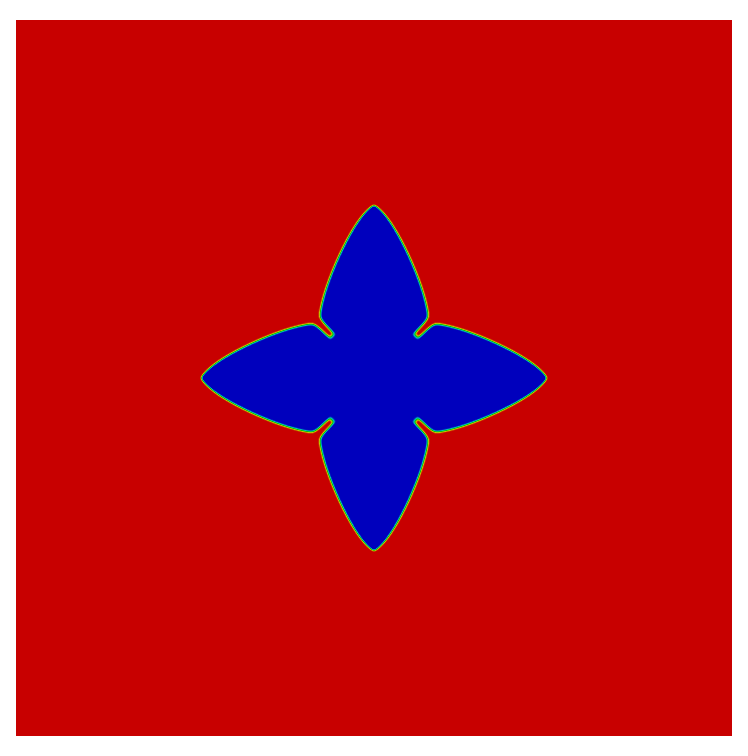}
\fi
\caption{($\epsilon^{-1} = 16\,\pi$, {\sc ani$_9$}, 
(\ref{eq:varrho})(ii), $\alpha=10^{-3}$, $\rho=0.01$, $\uD = -0.5$, 
$\Omega=(-4,4)^3$)
Snapshots of the solution at times $t=0,\,0.1,\,0.2,\,0.3,\,0.4$.
[This computation took $20$ days.]
}
\label{fig:3dnewStefanii_16pi}
\end{figure}%


\begin{thebibliography}{}
\bibitem[Luckhaus(1990)Luckhaus]{Luckhaus90}
Luckhaus, S. (1990).
\newblock Solutions for the two-phase {S}tefan problem with the
  {G}ibbs--{T}homson law for the melting temperature.
\newblock {\em European J. Appl. Math.}, {\bf 1}(2), 101--111.
\end{thebibliography}
\end{document}